\documentclass[a4paper,onecolumn,nopdfoutputerror,superscriptaddress,10pt,accepted=2023-01-17,issue=3,volume=5, shorttitle=papers/compositionality-5-3]{compositionalityarticle}
\usepackage[utf8]{inputenc}
\usepackage[english]{babel}
\usepackage[T1]{fontenc}
\usepackage{amsmath}
\usepackage{hyperref}

\usepackage[all]{xy} 
\usepackage[usenames,dvipsnames]{pstricks}
\usepackage{amssymb,mathrsfs,euscript,enumitem,amsmath,bbold,amsthm}

\newtheorem{theorem}{Theorem}[section]
\newtheorem{proposition}[theorem]{Proposition}
\newtheorem{lemma}[theorem]{Lemma}

\newtheorem{corollary}[theorem]{Corollary}

\newtheorem*{blow up}{Blow-up axiom}
\newtheorem*{weak_blow up}{Weak blow-up axiom}

\theoremstyle{definition}
\newtheorem{definition}[theorem]{Definition}
\newtheorem{example}[theorem]{Example}

\newtheorem{remark}[theorem]{Remark}

\newtheorem*{assu}{Assumptions}

\def\Cross#1{\mathop{\hbox{\LARGE $\times$}}\limits_{\mbox{\scriptsize $#1$}}}

\def\Iso{{\ttO}_{\tt iso}}

\def\Deltaalg{{\tt \Delta}_{\it alg}}

\def\Alg#1{\hbox{$#1$-{\tt Alg}}} 
\def\Oper#1{\hbox{$#1$-{\tt Oper}}} 

\def\Groterm{{\sf 1}_{\int_\ttP \calO}}

\def\colorop #1(#2;#3){{#1}
   \left(\rule{0pt}{15pt}\right.
         \hskip -3mm \begin{array}{c}
	              #3\\#2
                     \end{array}
         \hskip -3mm \left. 
   \rule{0pt}{15pt} \right)
}

\def\hGr{{\frac12\Gr}}
\def\Dio{{\tt Dio}}
\def\Whe{{\tt Whe}}
\def\Or{{\rm Or}}
\def\Set{{\tt Set}}
\def\Sur{{\rm Surj}}
\def\Ordn{{{\tt Ord}_n}}

\def\frC{{\mathfrak C}}
\def\Ar{{\it Ar}}

\def\calP{{\mathcal P}}
\def\SFac{{\sf SFac}}
\def\UFB{{\sf UFib}}
\def\Rig{{\sf Rig}} 
\def\QBI{{\sf QBI}}
\def\Fac{{\sf Fac}} 
\def\WBU{{\sf WBU}} 
\def\SBU{{\sf BU}}
\def\SGrad{{\sf SGrad}}  
\def\sFSet{{\tt sFSet}}
\def\Fib{{\it Fib}}
\def\dash{{\hbox{\hskip .15em -\hskip .1em}}}
\def\Su{{\EuScript S}}
\def\OpCat{{\tt OpCat}}
\def\Ord{{\mathbb {\tt \Delta}_{\rm semi}}}

\def\oP{{\EuScript P}}
\def\calP{\oP}

\def\End{{\EuScript E}nd}

\def\ttV{{\tt V}}

\def\oC{\raisebox{.7em}{\rotatebox{180}{$\EuScript P$}}}
\def\ggGrc{{\tt ggGrc}}
\def\ttC{{\tt C}}
\def\RTr{{\tt RTr}}\def\PRTr{{\tt PRTr}}\def\PTr{{\tt PTr}}

\def\Im{{\it Im}}
\def\ttprGr{{\tt prGr}}

\def\im{isomorphism}

\def\Markl{{\EuScript M}}

\def\fib{\triangleright}
\def\FIB{\blacktriangleright}

\def\Term{{\tt Term}}
\def\term{{\sf 1}}
\def\bfk{{\mathbb k}}

\def\Surj{{\Fin_{\rm semi}}}

\def\Vect{{\tt Vect}}
\def\Per{{\tt Per}}
\def\ttP{{\tt P}}

\def\kompozice{\relax{}}
\def\FB{{\tt fib}}\def\OF{{\tt opfib}}\def\BF{{\tt bifib}}
\def\LT{\ttO_{\tt LTrm}}

\def\calO{{\EuScript O}}
\def\rada#1#2{{#1,\ldots,#2}}

\def\Rada#1#2#3{#1_{#2},\dots,#1_{#3}}

\def\epi{\twoheadrightarrow}

\def\ot{\otimes}
\def\id{1\!\!1}

\def\term{{\boldsymbol 1}}
\def\QO{{ \ttO_{\tt qb}}}
\def\ttO{{\tt O}}
\def\ttP{{\tt P}}
\def\Grc{\tt Grc}
\def\Tr{\tt Tr}

\def\bbN{{\mathbb N}}
\def\Leg{{\rm Leg}}
\def\Vert{{\rm Ver}}

\def\Flag{{\rm Flg}}
\def\Fin{{\tt Fin}}
\def\Vines{{\tt Vines}}
\def\Bq{{\tt Bq}} 

\def\Gr{{\tt Gr}}

\def\and{{\mbox { and }}}
\def\DO{{ \ttO_{\tt ord}}}
\def\DP{{ \ttP_{\tt ord}}}
\def\inv#1{{#1}^{-1}}
\def\martin#1\endmartin{\noindent{{\bf Martin:}\ {\color{red} #1} 
    \hfill\rule{10mm}{.75mm}} \break}

\title[Operadic categories as natural environment]{Operadic categories as a 
natural environment for Koszul duality}
\author{Michael Batanin}
\email{bataninmichael@gmail.com}
\author{Martin~Markl}
\email{markl@math.cas.cz}
\affiliation{Institute of Mathematics of the Czech Academy of Sciences, {\v Z}itn{\'a} 25,
         115 67 Prague 1, The Czech Republic\\ 
		 and
		 MFF UK, Sokolovsk\'a 83, 186 75 Prague 8, The Czech Republic}

\def\qb{quasi\-bijection}
\def\Qb{Quasibijection}

\begin{document}
\bibliographystyle{plain}

\begin{abstract}
This is the first paper of a series
which aims to set up
the cornerstones of Koszul duality for operads over operadic
categories.
To this end we single out additional
properties of operadic categories under which the theory of quadratic
operads and their Koszulity can be developed, parallel to the traditional one by
Ginzburg\textendash{}Kapranov. We then investigate how these
extra properties interact with discrete operadic (op)fibrations, which
we use as a powerful tool to construct new operadic categories from
 old ones. We pay particular attention to the operadic category
of graphs, giving a full description of this category (and its variants)
as an operadic category, and proving that it satisfies all the additional 
properties.

Our present work provides an answer to a question formulated in
Loday's last talk,
in 2012: ``What encodes types of
operads?''. In the second and third papers of our
series
we continue Loday's program by answering
his second question: ``How to construct Koszul duals to these
objects?'', and proving Koszulity of some of the most relevant operads.
\end{abstract}

\maketitle

\setcounter{secnumdepth}{3}
\setcounter{tocdepth}{1}

\tableofcontents

\section*{Introduction}

Operads are a powerful foundation for handling composition and
substitution of various structures. While at first the underlying
combinatorics of operads concerned how trees are composed and
contracted, mathematics and mathematical physics soon found the
need for composing also more general graphs, leading to more
complex notions.

The present work is the first one in a series of articles which 
lays down the basic stones of ``operadic calculus'' for our
general theory of ``operad-like'' structures. By them we mean, besides
the classical operads in the sense of Boardman{\textendash}Vogt \cite{boardman-vogt:73} and
 May~\cite{may:1972},
and their more recent variants such as cyclic, modular or wheeled
operads~\cite{getzler-kapranov:CPLNGT95,getzler-kapranov:CompM98,mms}, 
also diverse versions of PROPs such as
properads~\cite{vallette:TAMS07}, dioperads~\cite{gan},
\hbox{$\frac12$PROPs}~\cite{mv}, and still more exotic stuff as
permutads and pre-permutads~\cite{loday11:_permut} or protoperads~\cite{leray}. 
Also Batanin's $n$-operads \cite{batanin:conf,batanin:AM08} appear in our scope.
One may vaguely characterize operad- and 
PROP-like structures as those generalizing
compositions of multivalued functions. 

\bigskip

\noindent 
{\bf History.}
To our knowledge, the first attempt to systematize this kind of objects was
made by the second author in a 2008 paper~\cite{markl:handbook}. He considered
structures with operations modeled  by contractions along edges of
graphs (called ``pasting schemes'' in this context) 
of the type particular to a concrete
situation. These schemes were required to satisfy an important property
of {\em hereditarity\/}, which is a specific stability  under 
contractions of subgraphs. This property was later
redressed into categorical garment in the 
notion of a {\em Feynman  category\/}~\cite{kaufmann-ward:Fey}. 
Hereditarity however played an important r\^ole already
in~\cite{borisov-manin} and in  unpublished work of
Melli\`es\textendash{}Tabareau~\cite{MelliesTabareau-TAlgTheoriesKan}. 
Let us close this brief  historical account by
mentioning Getzler's work on regular patterns~\cite{getzler:operads-revisited} predating 
Feynman categories, see also the 
follow-ups~\cite{BKW,comprehensive}. Finally, in~\cite{BB} an approach to
general operad-like structures through the use of polynomial monads was developed.
We are not commenting in this work on the connections between 
these approaches and ours,  since this topic deserves a separate paper.  

\bigskip

\noindent 
{\bf The setup.}
Our approach is based on the notion of
an {\em operadic category\/}. The idea goes back to the first author's
work on higher category theory using a higher generalization of
non-symmetric  operads~\cite{batanin:en}. In this
formalism, a higher version of the \hbox{Eckmann\textendash Hilton} argument was described by
reformulating the classical notion of a (symmetric) operad and
Batanin's notion of an $n$-operad in such a~way that a comparison of
the two notions became possible \cite{batanin:conf,batanin:AM08}. The
fruitfulness of this idea was then confirmed in \cite{batanin:br}.

In our work on the duoidal Deligne conjecture we came to
understand that the same categorical scheme is very useful and,
indeed, necessary for the study of many other standard and
nonstandard operad-like structures.  Thus the concept of operadic
categories was introduced by the authors in~\cite{duodel}.

Operadic categories are essentially the most distilled algebraic
structures which contain all information determining operad-like
structures of a given type along with their algebras.
Morphisms in operadic categories possess {\em fibers\/} whose properties
are modeled by the preimages of
maps between finite sets. Unlike in
Barwick's operator categories  \cite{barwick}, the fibers need not
be pullbacks.  Each  operadic category $\ttO$ has its {\em operads\/}
and each $\ttO$-operad $P$ has its category of {\em $P$-algebras\/}.  

An archetypal operadic category is the skeletal category $\Fin$ of
finite sets.
Its operads are classical one-colored symmetric operads.  
As we demonstrate in this paper, various hereditary categories of graphs 
are operadic. Examples of different
scent are Batanin's  $n$-trees and $n$-ordinals, or the operadic category
supporting permutads.
For the reader's convenience we recall definitions of operadic
categories and related notions in 
the opening Section~\ref{a0}. 
The background scheme of our approach is the triad
\begin{eqnarray*}
&\boxed{\hbox {level 1: $\ttO^+$-operads}}&
\\
&\Downarrow&
\\
&\boxed{\hbox {level 0: $\ttO$-operads = ${\sf 1_{\ttO^+}}$-algebras}}&
\\
&\Downarrow&
\\
&\boxed{\hbox {level -1: algebras of $\ttO$-operads}}&
\end{eqnarray*}
in which ``$\Downarrow$'' means ``is governed by.'' At level~$0$ one sees 
operads over an operadic category~$\ttO$. 
We consider algebras for these operads as objects at level~$-1$. 
It turns out that $\ttO$-operads are algebras for the constant operad
 ${\sf 1}_{\ttO^+}$ over a certain operadic category $\ttO^+$
called the \hbox{{\em $+$-construction\/}} of $\ttO$, which we place at
level $1$. The triad can  be continued upwards to~infinity. The theory
of $+$-constructions will be developed in a future paper.

An example is the {\em classical triad\/} in which $\ttO$ is 
the operadic category $\Fin$ of finite sets. $\Fin$-operads  simultaneously  appear as algebras of the
constant operad ${\sf 1}_\RTr$ over the operadic 
category $\RTr$ of rooted trees, which is $\Fin^+$. At level~$-1$ we
find algebras for the classical operads. 

Strong inspiration for our setup was the seminal
paper by Getzler and Kapranov~\cite{getzler-kapranov:CompM98}, who  realized that
modular operads are algebras
over a certain (hyper)operad. They thus constructed 
levels $0$ and $-1$ of the triad for  the operadic
category $\ggGrc$ of connected genus-graded ordered graphs, 
cf.~Example \ref{Zaletame_si_jeste_do_konce_Safari?}.
It turns out that $\ggGrc^+$ at level $1$ is the  category of
graphs from $\ggGrc$ with a hierarchy of nested subgraphs. The
resulting scheme is the {\em Getzler{\textendash}Kapranov triad\/}. 

The novelty of our approach is that
we systematically put the structures we want to study
at level $-1$ so that they appear as algebras over a certain
operad.  For instance, cyclic operads in our setup are algebras
over the constant operad ${\sf 1}_{\Tr}$  over the operadic category $\Tr$
of trees, though they themselves are  {\em not\/} operads over any
operadic category.

\bigskip

\noindent {\bf Aims of the present and future work.} In this paper we focus on
categorical and combinatorial foundations of
Koszul duality for operads over operadic categories. In the follow-up
\cite{sydney2} we introduce the notion of quadraticity for
operads over operadic categories, and all other
ingredients of the duality theory for operads including the Koszul
property. We will then prove that operads describing the most common
structures are Koszul. 
This provides an answer to the two questions in
Loday's last talk~\cite{loday:talk} mentioned in the~Abstract.

Our series of papers is continued by~\cite{BMO} in which 
we construct explicit minimal models for the (hyper)operads governing
modular, cyclic and ordinary operads, and wheeled properads. 
The final paper of this series will be devoted to
the $+$-construction in the context of operadic categories.

\bigskip

\noindent 
{\bf The plan.}
In Section~\ref{a0} we recall operadic categories and related notions,
using almost verbatim the material of~\cite{duodel}. 
In Section~\ref{Pojedu_vecer_nebo_ted?} we single out some finer
additional properties of operadic categories which will ensure in our second paper \cite{sydney2}  that free
operads over these categories are of a~particularly nice form. 
Section~\ref{Ceka_mne_Psenicka.} is devoted to our construction of an important 
operadic category of graphs and we show that it  satisfies all these extra 
requirements.  We will also see that several subtle properties of
graphs may be conveniently expressed in the language of our~theory.
In Section~\ref{Minulou_sobotu_jsem_odletal_vlekarskou_osnovu.} we
recall from~\cite{duodel}
discrete (op)fibrations and the related Grothendieck construction, 
and use it as a tool for producing new operadic categories from
old~ones. 

Free operads over operadic categories will play an important r\^ole
both in the definition of
quadraticity and of the dual dg operad needed for
the formulation of the Koszul property in the follow-up
\cite{sydney2}. 
As we noticed
for classical operads in~\cite{markl:zebrulka}, the construction of
free operads is more structured if
one uses, instead of the standard definition, a modified one.
Let us explain what we mean by this.

Traditional operads in the spirit of May~\cite{may:1972} are 
collections $\{\oP(n)\}_{n\geq 1}$  of
$\Sigma_n$-modules with composition laws
\begin{subequations}
\begin{equation}
\label{po 6ti dnech}
\gamma : \oP(k) \ot \oP(n_1) \ot \cdots \ot \oP(n_k) \to \oP(n_1 + \cdots +
n_k),\
k, \Rada n1k \geq 1,
\end{equation} 
satisfying appropriate
associativity and equivariance axioms; notice that we implicitly 
assume that $\oP(0)$ is empty.
However, in~{\cite[Definition~1.1]{markl:zebrulka}} we suggested
a definition based on binary composition laws
\begin{equation}
\label{letel_jsem_ve_snehovych_prehankach}
\circ_i : \oP(m) \ot \oP(n) \to \oP(m+n-1), \ n \geq 1,\ 1 \leq i \leq m.
\end{equation}
\end{subequations}
It turned out that under some
quite standard assumptions, for instance in the presence of units,
augmentations or connectivity, 
the two notions are equivalent, see
e.g.~\cite[Observation~1.2]{markl:zebrulka} 
or~\cite[Proposition~13]{markl:handbook}, though there are
structures possessing composition laws~(\ref{po 6ti dnech})
only~\cite[Example~19]{markl:handbook}. Operad-like
structures based on ``partial compositions''
in~(\ref{letel_jsem_ve_snehovych_prehankach}) were later called
Markl operads.  

Operads over general operadic categories also exist in two
disguises which are, under favorable conditions, equivalent -- 
in the form where the compositions in all inputs are
made simultaneously; this is how they were introduced
in~\cite{duodel} -- and in Markl form whose
composition laws are binary.
The crucial advantage of the latter is, as in the classical case,
that free Markl operads are naturally graded by the length of
the chain of compositions. The theory of Markl operads and its
relation to the original formulation of operad theory over operadic
categories given in~\cite[Section~1]{duodel} together with the
necessary background material occupies 
Sections~\ref{Velikonoce s Jarkou na chalupe.}
and~\ref{section-markl}.
To help the reader navigate  the paper, we include an index
of terminology and notation. 

\subsection*{Acknowledgments}

The first author acknowledges the financial
support of the Australian Research Council (grant
No.~DP130101172).
The second author was supported by 
grant GA \v CR 18-07776S, Praemium Academiae and RVO: 67985840.
Both authors acknowledge the hospitality of the Max Planck
  Institute for Mathematics in Bonn where this work was initiated.
  We express our gratitude to Joachim Kock and the referees for 
useful suggestions and comments that 
led to substantial improvement of our paper. 

\section{Operadic categories and their operads}
\label{a0}

In this preliminary
section we recall, for the convenience of the reader, some basic
definitions from \cite{duodel}. The reader may also wish to look at
Lack's paper \cite{lack} for a characterization of operadic categories
in the context of skew monoidal categories, or at~\cite{GKW} by
Garner, Kock and Weber for yet another point of view.  For brevity we
use the terms {\it operadic category} and {\it operadic functor}
for what have been defined as a {\it strict operadic
  category} and a {\it strict operadic functor} in \cite{duodel}. 
More general equivalence-invariant operadic categories will be the subject
of  upcoming work of Batanin, Kock and Weber.

Let $\Fin$ be the skeletal category of finite
sets (denoted in~\cite{duodel} by $\sFSet$).  The objects of
this category are the linearly ordered sets $\bar{n} =\{1,\ldots,n\} ,
n\in \bbN.$ Morphisms are arbitrary maps between these sets.  We
define the $i$th fiber $f^{-1}(i)$ of a morphism $f: T \to S$,
$i\in S$, as the pullback of $f$ along the map $\bar{1}\to S$ which picks up the element $i$, so this is the object $f^{-1}(i)=\bar{n}_i \in \Fin$ which is isomorphic as
a linearly ordered set to the preimage $\big\{j\in T \ | \ f(j) =i\big\}$.
Any commutative triangle 
 in $\Fin$
  \[
    \xymatrix@C = +1em@R = +1em{
      T      \ar[rr]^f \ar[dr]_h & & S \ar[dl]^g
      \\
      &R&
    }
\]
then induces a map $f_i:h^{-1}(i)\to g^{-1}(i)$ for any $i\in R$.  This
assignment is a functor $\Fib_i:\Fin/R\to \Fin$. Moreover, for
any $j\in S$ we have the equality $f^{-1}(j) =
f^{-1}_{g(j)}(j).$ The above structure on the category $\Fin$ motivates the
following abstract definition.

Recall that an object $t$ in a category $\ttO$ is a {\it local terminal object} if it is a terminal object in its connected component.
An {\em  operadic category\/} is a 
category $\ttO$  equipped with a {\em cardinality
functor}  $|\dash|:\ttO\to \Fin$ having the following properties. We
require that each connected component of $\ttO$ has a chosen local terminal object
$U_c$, $c\in \pi_0(\ttO)$.  We also assume that for every $f:T\to S$
in $\ttO$ and every element $i\in |S|$, there is given an object
$f^{-1}(i) \in \ttO,$ which we will call {\it the $i$-th fiber\/} of~$f$,
such that $|f^{-1}(i)| = |f|^{-1}(i).$ We also require that 
\begin{itemize}  \item[(i)] For any $c\in \pi_0(\ttO)$, $|U_c| = 1$.

\item[(ii)]
For any $T \in \ttO$ and each $i\in |T|$, the fiber $\id_T^{-1}(i)$ of
the identity  $\id_T:T\to T$ is a~chosen local terminal object.
\item[(iii)] For any commutative diagram in $\ttO$
  \begin{equation}
    \label{Kure}
    \xymatrix@C = +1em@R = +1em{
      T      \ar[rr]^f \ar[dr]_h & & S \ar[dl]^g
      \\
      &R&
    }
\end{equation}
  and every $i\in |R|$, one is given a map
  \[
  f_i: h^{-1}(i)\to g^{-1}(i)
  \]
  such that $|f_i|: |h^{-1}(i)|\to |g^{-1}(i)|$ is the map
  $|h|^{-1}(i)\to |g|^{-1}(i)$ of sets induced by
  \[
  \xymatrix@C = +1em@R = +1em{ |T| \ar[rr]^{|f|} \ar[dr]_{|h|} & & |S|
    \ar[dl]^{|g|}
    \\
    &\phantom{.} |R| \,.& }
  \]
  We moreover require that this assignment forms a functor ${\Fib}_i:
  \ttO/R \to \ttO$  called {\em fiber functor}.   Moreover, if
  $R=U_c$, the functor ${\Fib}_1$ must be the domain functor
  $\ttO/R \to \ttO$. The last condition says
that  the unique fiber of the canonical morphism $!_T: T\to
  U_c$ is $T$.
  \item[(iv)] 
In the situation of Axiom (iii), for any $j\in |S|$,  one has the equality
\begin{equation}
\label{karneval_u_retardacku}
  f^{-1}(j) = f_{|g|(j)}^{-1}(j).
\end{equation}

\item[(v)] 
Let 
\[
\xymatrix@C = +2.5em@R = +1em{ & S \ar[dd]^(.3){g} \ar[dr]^a & 
\\
T  \ar[ur]^f    \ar@{-}[r]^(.7){b}\ar[dr]_h & \ar[r] & Q \ar[dl]^c
\\
&R&
}
\]
be a commutative diagram in $\ttO$ and let $j\in |Q|, i = |c|(j).$
Then by Axiom (iii) the diagram
\[
    \xymatrix@C = +1em@R = +1em{
      h^{-1}(i)      \ar[rr]^{f_i} \ar[dr]_{b_i} & & g^{-1}(i) \ar[dl]^{a_i}
      \\
      &c^{-1}(i)&
    }
\]
commutes, so it induces a morphism $(f_i)_j: b_i^{-1}(j)\to a_i^{-1}(j).$ By Axiom (iv) we have $$a^{-1}(j)=a_i^{-1}(j) \ \mbox{and} \ b^{-1}(j)=b_i^{-1}(j).$$
We  then require the equality
$f_j = (f_i)_j$.
\end{itemize}

We will also assume that the set $\pi_0(\ttO)$ of connected components
is {\em small\/} with respect to  a sufficiently big ambient universe.

An {\it  operadic functor\/} between  operadic categories is a
functor $F: \ttO \to \ttP$ over $\Fin$ which preserves fibers in
the sense that $F\big(f^{-1}(i)\big) = F(f)^{-1}(i)$, for any $f : T
\to S \in \ttO$ and $i \in |S| = |F(S)|$.  We also require that $F$
preserves the chosen local terminal objects, and
that $F(f_i) = F(f)_i$ for $f$ as in~(\ref{Kure}).  This gives the
category ${\OpCat}$ of  operadic categories and their operadic
functors.

Let $(\ttV,\otimes,\bfk)$ be a (closed) symmetric monoidal
category. Thanks to MacLane's coherence theorem  
we will also assume that associativity and unit 
constrains in $\ttV$  are identities. 
For a~family  $E = \{E(T)\}_{T
  \in \ttO}$ of objects of $\ttV$ and a morphism $f:T\to S$ 
let 
\[
E(f) = \bigotimes_{i \in |S|} E({T_i}),\ T_i := \inv f(i).
\]
In the following definition  we tacitly use
equalities~(\ref{karneval_u_retardacku}).

\begin{definition} 
\label{Jarca_u_mne_prespala!}
An \emph{operad over $\ttO$} (or simply an \emph{$\ttO$-operad})  in $\ttV$ is a family  $\calP = \{\calP(T)\}_{T
\in \ttO}$ of objects of $\ttV$ together with units
  \[
  I\to \calP(U_c),\ c \in \pi_0(\ttO),
  \]
and composition laws
  \[
  \gamma_f: \calP(f) \otimes \calP(S)\to \calP(T),\ f:T\to S,
  \]
satisfying the following axioms.

  \begin{itemize}
  \item[(i)] Let $T \stackrel f\to S \stackrel g\to R$ be morphisms in
    $\ttO$ and $h := gf : T \to R$ as in~(\ref{Kure}).  Then the
    following diagram of composition  laws of $\calP$ combined with the
    canonical isomorphisms of products in $\ttV$ commutes:
    \[
    \xymatrix@C = 2em@R = .4em{
\ar@/^2.5ex/[rrd]^(.56){\hskip .5em\bigotimes_{i}\gamma_{f_i} \ot \id}
      \ar[dd]_(.45){\id \ot \gamma_g}
\displaystyle\bigotimes_{i
        \in |R|} 
      \calP(f_i) \ot \calP(g) \ot \calP(R) & &    
      \\  & &  \ar@/^/[dl]_{\gamma_h}
     \calP(h) \ot \calP(R)\ .
\\
      \ar[r]^(.77){\gamma_f}{\rule{0pt}{2em}}  
\displaystyle\bigotimes_{i \in |R|}
      \calP(f_i) \ot \calP(S) \cong   \calP(f) \ot \calP(S)&
 \calP(T)&
    }
    \]
  \item[(ii)] The composite
    \[
    \xymatrix@1@C = +2em{ \calP(T) \ar[r]& \rule{0pt}{2em} \displaystyle
      \bigotimes_{i\in |T|} \bfk\! \ot\! \calP(T) \ar[r]&\rule{0pt}{2em}
      \displaystyle \bigotimes_{i\in |T|} \calP(U_{c_i})\! \ot\! \calP(T)\ar[r]^(.56)= &
      \calP(\id_T)\! \ot\! \calP(T)
 \ar[r]^(.65){\gamma_{\id}}&\calP(T) }
    \]
    is the identity for each $T \in \ttO$. 
  \item[(iii)] The composite 
    \[
    \xymatrix@1@C = +2.2em{ \calP(T)\! \ot \! \bfk \ \ar[r]&\ \calP(T)\! \ot\!
      \calP(U_c) \ \ar[r]^=
&\ \calP(!_T)\! \ot\! \calP(U_c)\
      \ar[r]^{\hskip 1.8em \gamma_{!_T}}&\ \calP(T)}
    \]
   is the identity for each $T \in \ttO,$    where  $!_T: T\to
  U_c$ is the unique morphism.
    
  \end{itemize}
\end{definition}

A {\em morphism\/} $\calP' \to \calP''$  
of $\ttO$-operads in $\ttV$ is a
collection of $\ttV$-morphisms  
\hbox{$\calP'(T) \to \calP''(T)$}, $T \in \ttO$, commuting with the
composition laws and units.  
We denote by $\Oper\ttO(\ttV)$ (or simply by $\Oper\ttO$ if $\ttV$ is
understood) the category of
$\ttO$-operads in $\ttV$.  Each operadic functor $F:\ttO\to \ttP$ 
induces the restriction $F^*:  \Oper\ttP\to \Oper\ttO$.

\begin{example} 
A primary example of an operadic category is the category
  $\Fin$, while  the cardinality functor \hbox{$|\!-\!|:\ttO\to\Fin$} 
is an example
  of an operadic functor. Thus $\Fin$ is the terminal object in
  the category of operadic categories and operadic
  functors. The category of $\Fin$-operads is isomorphic to the
  category of classical one-colored (symmetric)~operads.
  \end{example} 
  \begin{example} The subcategory  $\Surj \subset \Fin$ of nonempty finite sets and
surjections is also an operadic category. Operads over $\Surj$ are
classical one-colored symmetric operads without nullary operations. 
We will call such operads \emph{constant-free}. 
\end{example}

\begin{example}
\label{Dnes_na_muslicky.} 
  The category of vines $\Vines$ \cite{lavers,weber} is another
  example of an operadic category. It has the same objects as $\Fin$
  but a morphism $\bar{n}\to \bar{m}$ is an isotopy class of merging
  descending strings in $\mathbb{R}^3$ (called {\em vines\/}) like in the
  following picture:
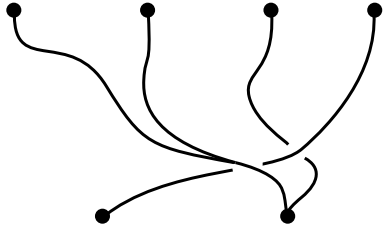
\begin{center}
{
\begin{pspicture}(5,-1.504111)(-0,1.7)
\psdots[dotsize=0.2](0.1,1.39)
\psdots[dotsize=0.2](1.88,1.39)
\psdots[dotsize=0.2](3.52,1.39)
\psdots[dotsize=0.2](1.28,-1.35)
\psdots[dotsize=0.2](3.74,-1.35)
\psdots[dotsize=0.2](4.9,1.39)
\psbezier[linewidth=0.04](0.108106375,1.3233984)(0.108106375,0.5233984)(0.8481064,1.1633984)(1.3281064,0.3833984)(1.8081064,-0.3966016)(1.9081063,-0.43660158)(3.0281065,-0.63660157)
\psbezier[linewidth=0.04](1.8881063,1.3433985)(1.9281064,0.5633984)(1.8481064,0.8433984)(1.8281064,0.4633984)(1.8081064,0.08339841)(2.0081065,-0.3566016)(3.0481064,-0.63660157)
\psbezier[linewidth=0.04](4.8881063,1.3433985)(4.9281063,0.5633984)(4.3281064,-0.116601594)(3.9281063,-0.4566016)(3.5281065,-0.7966016)(2.1081064,-0.6766016)(1.2681063,-1.3766016)
\psbezier[linewidth=0.04](3.5281065,1.3433985)(3.5681064,0.5633984)(3.1481063,0.5833984)(3.2281063,0.24339841)(3.3081064,-0.09660159)(3.7281063,-0.3566016)(3.7481065,-0.3966016)
\psbezier[linewidth=0.04](3.9681063,-0.57660156)(4.1881065,-0.69660157)(4.1681066,-0.8966016)(3.9281063,-1.0966016)(3.6881063,-1.2966015)(3.7281063,-1.3346504)(3.7411063,-1.4166015)
\pscustom[linewidth=0.106000006,linecolor=white]
{
\newpath
\moveto(3.4081063,-0.6766016)
\lineto(3.3414402,-0.6766016)
\curveto(3.3081064,-0.6766016)(3.2481065,-0.6816016)(3.22144,-0.6866016)
\curveto(3.1947732,-0.6916016)(3.1414394,-0.70160156)(3.1147726,-0.7066016)
\curveto(3.0881064,-0.7116016)(3.0347726,-0.7216016)(3.0081065,-0.7266016)
}
\psbezier[linewidth=0.04](3.0481064,-0.63660157)(3.4481063,-0.7366016)(3.6281064,-0.8766016)(3.6681063,-0.9966016)(3.7081063,-1.1166016)(3.6881063,-1.0166016)(3.7281063,-1.3166016)
\end{pspicture} 
}
\end{center}

There is a canonical identity-on-object functor $|-|:\Vines\to\Fin$
which sends a vine to the function assigning to the top endpoint
of a string its bottom endpoint.  A fiber of a vine
$v:\bar{n}\to\bar{m}$ is equal to the fiber of
$|v|:\bar{n}\to\bar{m}.$ The rest of the operadic category structure
on $\Vines$ is quite obvious.  The category of $\Vines$-operads is
isomorphic to the category of braided operads \cite{Fiedorowicz}. This
fact can be easily proved using the equivalent definition of braided
operad given in \cite{batanin:br}.

In fact, using Weber's theory \cite{weber} one can associate an
operadic category $\ttO(G)$ to each group operad $G$ (see \cite{yoshida}
for the definition) such that $\ttO(G)$-operads are
exactly $G$-operads. The operadic categories $\Fin$ and $\Vines$ are
special cases $\ttO(\Sigma)$ and  
$\ttO({\it  Braid})$ of this construction for the symmetric group and 
braid group operads, respectively. 
We will provide the details elsewhere. 
\end{example}

\begin{example}
  \label{puget}
  Let $\frC$ be a set.  Recall from \cite[Example~1.7]{duodel} (see also
  \cite[Example~10.2]{lack}) that a {\em $\frC$-bouquet\/} is a map $b:
  X\!+\!1\to \frC,$ where $X\in \Fin.$ In other words,
  a~$\frC$-bouquet is an ordered $(k+1)$-tuple $(i_1,\ldots,i_k;i),\ X
  = \bar{k}$, of elements of $\frC$. It can also be thought of as a~planar corolla all of whose  edges (including the root) are colored.
  The extra color $b(1) \in \frC$ is called the {\em root color\/}. The finite set
  $X$ is the {\em underlying set\/} of the bouquet~$b$.

  A map of $\frC$-bouquets $b \to c$ whose root colors coincide is an
  arbitrary map $f: X\to Y$ of their underlying sets. There are no maps between $\frC$-bouquets with different root colors. We denote the resulting category of
  $\frC$-bouquets by $\Bq(\frC)$.

  The cardinality functor $|\dash|:\Bq(\frC)\to \Fin$ 
assigns to a bouquet $b :X+1\to \frC$ its underlying
  set~$X$.  The fiber of a map $b \to c$ given by $f : X \to Y$ over
  an element $y\in Y$ is the $\frC$-bouquet whose underlying set is
  $f^{-1}(y),$ the root color coincides with the color of $y$ and the
  colors of the elements are inherited from the colors of the elements
  of $X$.

 Operads over the category $\Bq(\frC)$ of $\frC$-bouquets 
are ordinary $\frC$-colored operads. Therefore, for
each $\frC$-colored collection $E = \{E_c\}_{c\in \frC}$ of objects of
$\ttV$ one has the {\em endomorphism $\Bq(\frC)$-operad\/}
$\End^{\Bq(\frC)}_E$, namely the ordinary colored
endomorphism operad~\cite[\S1.2]{berger-moerdijk:CM07}.
\end{example}

\begin{example}
\label{finset+deltaalg}
The category $\Deltaalg$ of
finite ordinals (including the empty one) has an obvious structure of
an operadic category.  Operads over $\Deltaalg$ are ordinary nonsymmetric
operads~\hbox{\cite[Prop.~3.1]{batanin:AM08}}.
\end{example}

\begin{example}
\label{Musim na tu zatracenou chalupu.}
The cartesian product in the category of operadic categories exists
and is given by a pullback over $\Fin$ using cardinality functors. In
particular, for any operadic category 
$\ttO$ and any set of coulours $\frC$, the  \hbox{$\Bq(\frC)\!\times\!
    \ttO$}\,-operads are $\frC$-colored $\ttO$-operads \hbox{\cite[page~1637]{duodel}}.
Likewise, the product \hbox {$\Vines\! \times\! \ttO$} with the operadic category
of vines of Example~\ref{Dnes_na_muslicky.} describes braided versions
of $\ttO$-operads. The product  \hbox{$\Deltaalg\! \times\! \ttO$} is
isomorphic to
the subcategory $\DO\subset \ttO$ of morphisms for which $|f|$ is
order preserving.\label{Jel jsem na Martinove horskem kole.}
\end{example}

\begin{example}
Another important example is the operadic category $\Ordn$ of
$n$-ordinals, $n \in \bbN$, see~\cite[Sec.~II]{batanin:conf}. $\Ordn$-operads
are Batanin's pruned $n$-operads which are allowed to take values not only in
ordinary symmetric monoidal categories, but in more general {\em  globular\/}
monoidal $n$-categories. 
Although  $\Ordn$ does not fulfill the
additional properties required for some constructions in this work,
it was a crucial motivating example for our definition of operadic categories.
\end{example}

For each operadic category $\ttO$ with $\pi_0(\ttO)=\frC$, there is a
canonical operadic ``arity'' functor 
\begin{equation}
\label{Jarca}
\Ar: \ttO\to \Bq(\frC)
\end{equation}
giving
  rise to the factorization
  \begin{equation}
    \label{eq:factor}
\xymatrix@R=1em{&\ttO \ar@/_/[ld]_{\Ar} \ar@/^/[rd]^{|\dash|}   & 
\\
 \Bq(\frC) \ar[rr]^{|\dash|} &&\Fin
}
\end{equation}
of the cardinality functor $|\dash| : \ttO \to \Fin$. 
It is constructed as follows.

Recall
that the {\em $i$th source\/} $s_i(T)$ of an object $T \in
\ttO$ is the
$i$th fiber of the identity automorphism of $T$, i.e.\ 
$s_i(T) :=\id_T^{-1}(i)$ for $i \in |T|$. We denote by $s(T)$ the set of all sources of
$T$. For an object $T \in \ttO$ we denote by $\pi_0(T) \in
\pi_0(\ttO)$ the connected component to which $T$ belongs. Similarly,
for a subset $X$ of objects of $\ttO$, 
\[
\pi_0(X) := \{\pi_0(T) \ | \    T \in X\} \subset \pi_0(\ttO).
\]

  
  The bouquet $\Ar(T) \in \Bq(\frC)$ is
  defined as
 $b: s(T)+1 \to \frC$, where $b$ associates to each fiber $U
  \in s(T)$ 
the corresponding connected component $\pi_0(U)\in \frC$, and $b(1)
  := \pi_0(T)$.  The assignment
  $T \mapsto \Ar(T)$\/ extends to an operadic functor.

\begin{example}
\label{Zitra_do_Osnabrucku}
  For a $\frC$-colored collection $E = \{E_c\}_{c\in \frC}$ in $\ttV$ and an
  operadic category $\ttO$ with $\pi_0(\ttO) = \frC$, one defines the {\em
    endomorphism $\ttO$-operad\/} $\End^\ttO_E$ as the restriction
  \[
    \End^\ttO_E := \Ar^*\big(\End^{\Bq(\frC)}_E\big)
  \]
  of the $\Bq(\frC)$-endomorphism operad of Example~\ref{puget} along the
  arity functor $\Ar$~of~(\ref{Jarca}).
\end{example}

The following definition was given in \cite[Definition~1.20]{duodel}. 
\begin{definition}
\label{Jaruska_ma_chripecku}
An {\em algebra\/} over an $\ttO$-operad $\calP$ in $\ttV$ is a
collection $A = \{A_c\}_{c\in \pi_0(\ttO)}$, $A_c \in \ttV$, equipped with an
$\ttO$-operad map $\alpha : \calP \to \End^\ttO_A$.  
\end{definition}

An algebra structure is thus
provided by suitable structure maps
\[
\alpha_T : \calP(T)\ot \bigotimes_{c \in \pi_0(s(T))} A_c  \to
A_{\pi_0(T)}, \ T \in \ttO,
\]
We denote by
$\Alg\calP(\ttV)$ (or simply by $\Alg\calP$ when $\ttV$ is clear from the
context) the category of $\calP$-algebras and their morphisms.

\section{Sundry facts about operadic categories}
\label{Pojedu_vecer_nebo_ted?}

The aim of this section is to study some finer properties of operadic
categories and formulate some additional axioms and their consequences
required for our future constructions concerning the category of
graphs  and Koszul duality theory.

\bigskip

\noindent
{\bf Conventions.} 
Chosen local terminal objects of an operadic category $\ttO$ will be denoted by
$U$ with various decorations such as $U',U''$,~etc. 
The notation $U_c$ will mean the chosen local terminal object of a component 
$c \in \pi_0(\ttO)$.
We 
will sometimes call these chosen local terminal objects the {\em trivial\/}
ones.
\Qb{s} will be indicated by $\sim$, \im{s} by $\cong$; a preferred
notation for both of them will be something resembling permutations, like
$\sigma$, $\omega$, $\pi$,~etc.

A {\em quasibijection} is a 
 morphism $f:T\to S$ in $\ttO$ such that, for each $i\in
 |S|$,  we have
$f^{-1}(i) = U_{d_i}$ for some $d_i\in 
\pi_0(\ttO)$. 
To avoid any possible confusion, we assert that this definition 
implies that a~map between objects with the empty set of fibers is
also a \qb; {in~\cite{duodel} such morphisms were called
  {\em trivial\/}.}
Note that an isomorphism is not necessary a 
quasibijection and that a quasibijection is not necessary an isomorphism.

We will denote by $\QO\subset \ttO$ the subcategory of \qb{s}, and by
$\DO\subset \ttO$ the subcategory of morphisms for which $|f|$ is
order preserving. Notice that $\DO$ unlike $\QO$ has a~natural
structure of an operadic category, cf.~Example~\ref{Musim na tu
  zatracenou chalupu.}.\label{9 dni}

The following
Lemma \ref{l1}(iii) and Corollary \ref{invertingqb}(i) were also proved 
by Lack as~\cite[Lemma 8.2]{lack}.

\begin{lemma}
\label{l1}
Consider the commutative diagram in an operadic category
\[
\xymatrix@C=4em{S \ar[rd]^{f''} \ar[d]_{f'}
\\
T'  \ar[r]^(0.4)\sigma & T''\,.
}
\]
Let  $j\in |T''|$ and $|\sigma|^{-1}(j) = \{i\}$ for some 
$i\in |T'|$.
Then 
\begin{enumerate} 
\item[(i)] 
The unique fiber of the induced map \hbox{$f'_j : \inv{f''}(j) \to
\inv{\sigma}(j)$} equals $\inv{f'}(i)$. 
\item[(ii)]  If $\inv{\sigma}(j)$ is
trivial, in particular  if $\sigma$ is a \qb,~then 
\begin{equation*}
\inv{f'}(i) = \inv{f''}(j).\end{equation*}
\item[(iii)]  If both $\sigma$ and $f''$ are quasibijections then $f'$ is a quasibijection. 
\end{enumerate}
\end{lemma}

\begin{proof}
By Axiom~(iv) of an operadic category,
$\inv{f'}(i) = \inv{f'_j}(i)$ which readily gives the
first part of the lemma. 
If  $\inv{\sigma}(j)$ is trivial, then the fiber of $f'_j$ equals $
\inv{f''}(j)$ by Axiom~(iii). This proves the second and third part of the lemma.
\end{proof}

\begin{lemma}
\label{l2}
Consider the commutative diagram in an operadic category
\[
\xymatrix{S' \ar[rr]^\pi_\sim \ar[dr]_{f'} && S'' \ar[ld]^{f''}
\\
&T&
}
\]
where $\pi$ is a \qb. Then all $\pi_i: \inv{f'}(i) \to
\inv{f''}(i)$, $i \in |T|$, are \qb{s} too.
\end{lemma}

\begin{proof} Immediate from Axiom~(iv).
\end{proof}

\begin{corollary}\label{invertingqb}
In any operadic category $\ttO$,
\begin{enumerate}
\item[(i)] quasibijections are closed under composition. 
\item[(ii)] 
If a quasibijection is an isomorphism, then its inverse is also a
quasibijection. 
\end{enumerate}
\end{corollary}

\begin{proof} 
The first statement follows from Lemma \ref{l2} when $f''$ is a quasibijection. Indeed, in this case we have a quasibijection $!=\pi_i: \inv{f'}(i) \to
U_{c_i}$, for each $i \in |T|,$ but the fiber of such a morphism must be equal to $\inv{f'}(i).$ 

The second statement follows readily from part (iii) of Lemma
\ref{l1}.  %
\end{proof}

\begin{lemma}
\label{l3}
Consider the commutative diagram in an operadic category
\begin{equation}
\label{s1}
\xymatrix@C=4em{S' \ar[d]_{f'}\ar[dr]^{f} \ar[r]^\pi  & S'' \ar[d]^{f''}
\\
T' \ar[r]^\sigma &\ T''\,.
}
\end{equation}
Let  $j\in |T''|$ and $|\sigma|^{-1}(j) = \{i\}$ for some 
$i\in |T'|$. Diagram~(\ref{s1}) determines:
\begin{itemize}
\item[(i)]
the map $f'_j:  \inv{f}(j) \to
\inv{\sigma}(j)$ whose unique fiber equals  $\inv{f'}(i)$, and
\item[(ii)]
{the induced map  $\pi_j :  
\inv{f}(j) \to
\inv{f''}(j)$.}
\end{itemize}
If $\inv{\sigma}(j)$ is
trivial, in particular if $\sigma$ is a \qb,~then
$\pi$ induces a map
\begin{equation}
\label{zitra_na_prohlidku_k_doktoru_Reichovi}
\pi_{(i,j)} : \inv{f'}(i) \to  \inv{f''}(j)
\end{equation}
which is a \qb\ if $\pi$ is.
\end{lemma}

\begin{proof}
The first part immediately follows from Lemma~\ref{l1} and
Axiom~(iii). Under the assumption of the second part, one applies Lemma \ref{l1}(ii) to get an 
equality    $\inv{f'}(i)  = \inv{f}(j).$     
Then
$\pi_{(i,j)}$ is defined as the composite
\[
\pi_{(i,j)} : \inv{f'}(i)  = \inv{f}(j) \stackrel{\pi_j}\longrightarrow  
\inv{f''}(j).
\]
The rest follows
from Lemma~\ref{l2}.
\end{proof}

Thus, in the situation of Lemma~\ref{l3} with $\sigma$ a \qb, one has the {\em
  derived sequence\/} 
\begin{equation}
\label{e1}
\left\{\pi_{(i,j)} : 
\inv{f'}(i) \to \inv{f''}(j), j = |\sigma|(i)\right\}_{i \in |T'|}
\end{equation}
consisting of \qb{s} if $\pi$ is a \qb. 

\medskip

Central constructions of this work will require the following:

\begin{blow up}\label{bu} Let $\ttO$ be an operadic category.
Consider the corner
\begin{equation}
\label{c1}
\xymatrix@C=3.5em{S' \ar[d]_{f'}  &
\\
T' \ar[r]^\sigma_\sim &T''
}
\end{equation}
in which $\sigma$ is a \qb\ and $f' \in \DO$. Assume we are given
objects $F''_j$, $j \in |T''|$ together with a collection of maps
\begin{equation}
\label{e2}
\big\{\pi_{(i,j)} : \inv{f'}(i) \to F''_j,\ j = |\sigma|(i)\big\}_{i \in |T'|}.
\end{equation}
Then the corner~(\ref{c1}) can be completed uniquely  into the
commutative square
\begin{equation}
\label{eq:5}
\xymatrix@C=3.5em{S' \ar[d]_{f'} \ar[r]^\pi  & S'' \ar[d]^{f''}
\\
T' \ar[r]^\sigma_\sim &T''
}
\end{equation}
in which $f'' \in \DO$,  $\inv{f''}(j) = F''_j$ for $j \in |T''|$, and
such that the
derived sequence~(\ref{e1}) induced by $\pi$ coincides with~(\ref{e2}).
\end{blow up}

The requirement that $f',f'' \in \DO$ is crucial, otherwise 
the factorization would not be unique even in ``simple'' operadic
categories such as $\Fin$. It will sometimes suffice 
to assume the blow-up for $\sigma = \id$ only, i.e.\ to assume

\begin{weak_blow up}\label{wbu}
For any $f' :S' \to T$ in $\DO$ and
morphisms $\pi_i : \inv{f'}(i) \to F''_i$ in $\ttO$, $i \in |T|$, 
there exists a unique factorization of $f'$
\[
\xymatrix@C=2em@R=1.2em{S' \ar[rr]^\omega \ar[dr]_{f'} && S'' \ar[ld]^{f''}
\\
&T&
}
\] 
such that $f'' \in \DO$ and $\omega_i = \pi_i$ for all $i \in |T|$.
\end{weak_blow up}

Notice that $\omega \in \DO$ (resp.~$\omega \in \QO$) if and only if
$\pi_i \in \DO$  (resp.~$\pi_i \in \QO$) for all $i \in |T|$. 

\begin{remark}
If we require the weak blow-up axiom only for order-preserving $\pi_i$ then it 
simply means that the fiber functor
\[
\DO/T \to \DO^{ |T|}
\]
is a discrete opfibration. Such a condition for an operadic category
$\ttO$ (not only for its subcategory $\DO$) is closely related to
Lack's condition \cite[Proposition 9.8]{lack} which ensures that the
natural tensor product of $\ttO$-collections, which is only skew
associative in general, is genuinely associative. In fact, as we will
show elsewhere, under some restrictions natural in our context, the
weak blow-up axiom implies Lack's condition (see also Remarks 13--14 of \cite{GKW}  for other
important connections).
\end{remark}

\begin{corollary}
\label{zase_mam_nejaky_tik}
If the weak blow-up axiom is satisfied in $\ttO$, then
\[
\QO \cap \DO = \ttO_{\mathtt{disc}},
\]
the discrete category with the same
objects as $\ttO$. In particular, the only \qb{s}\ in $\DO$ are the
identities. 
\end{corollary}

\begin{proof}
It is clear that each identity belongs to $\QO \cap \DO$. On the other
hand, assume that $\phi : S \to T \in \QO \cap \DO$. Since it is a
\qb, all its fibers are trivial, $\inv{\phi}(i) = U_i$ for $i\in
|T|$. Consider now the two factorizations of $\phi$,
\begin{equation}
\label{posledni_den_v_Srni}
\xymatrix@R=1em{&\ar@{=}[ld]_{\id_S} \ar[dd]^\phi S & & \ar[dd]_\phi  \ar[dr]^\phi    S&
\\
S\ar[dr]^\phi & &  & &T\ar@{=}[ld]_{\id_T}
\\
&T&  & T \,.&
}
\end{equation} 
In the left triangle we have
$(\id_T)_i :U_i= \inv{\phi}(i) \to
\inv{\phi}(i) = U_i$, for $i\in |T|$, and therefore  $(\id_T)_i = \id_{U_i}$ by the
terminality of $U_i$. Let us turn our attention to the right triangle.

By Axiom~(ii) of an operadic category, all fibers of an identity
are trivial, thus
\[
(\phi)_i :U_i= \inv{\phi}(i) \to
\inv{{\id_T}}(i) = U_c
\]
for some chosen local terminal $U_c$. Since any morphism between trivial 
objects is an identity we see that both factorizations
in~(\ref{posledni_den_v_Srni}) are determined by the collection
$\id_{U_i} : \inv{\phi}(i) \to U_i$, $i \in |T|$, so by the
uniqueness in the blow-up axiom, they are the same.
\end{proof}

\begin{example}
Corollary~\ref{zase_mam_nejaky_tik} shows the power of the blow-up
axiom and illustrates how it determines the nature of an operadic
category. While it is satisfied in operadic categories underlying
``classical'' examples of operads, it is violated e.g.~in the category
of vines recalled in Example~\ref{Dnes_na_muslicky.}, whose  operads are braided operads, or in
Batanin's category of \hbox{$n$-trees}, whose  operads are $(n-1)$-terminal (but not pruned) globular \hbox{$n$-operads}~\cite[Section~4]{batanin:AM08}.

Let us look at vines first. 
For the automorphism $s \in \Vines(\bar{2},\bar{2})$ represented
by 

\begin{equation*}
\psscalebox{.8 .8} 
{
\begin{pspicture}(0,-2)(17.5,-3.7)
\definecolor{colour0}{rgb}{0.99607843,0.99607843,0.99607843}
\pscustom[linecolor=black, linewidth=0.04]
{
\newpath
\moveto(6.2,1.4985577)
}
\pscustom[linecolor=black, linewidth=0.04]
{
\newpath
\moveto(7.6,-1.5014423)
\lineto(7.836842,-1.9483173)
\curveto(7.955263,-2.1717548)(8.239474,-2.4811296)(8.405263,-2.5670671)
\curveto(8.571053,-2.6530046)(8.855263,-2.8592548)(8.973684,-2.9795673)
\curveto(9.092105,-3.0998797)(9.2578945,-3.3405046)(9.4,-3.7014422)
}
\psdots[linecolor=black, dotsize=0.24](7.6,-1.5014423)
\psdots[linecolor=black, dotsize=0.24](9.4,-1.5014423)
\psdots[linecolor=black, dotsize=0.24](9.4,-3.7014422)
\psdots[linecolor=black, dotsize=0.24](7.6,-3.7014422)
\psdots[linecolor=colour0, dotstyle=o, dotsize=0.6, fillcolor=white](8.4,-2.7)
\pscustom[linecolor=black, linewidth=0.044]
{
\newpath
\moveto(9.4,-1.5014423)
\lineto(9.163158,-1.9483173)
\curveto(9.044737,-2.1717548)(8.760526,-2.4811296)(8.594737,-2.5670671)
\curveto(8.428947,-2.6530046)(8.144737,-2.8592548)(8.026316,-2.9795673)
\curveto(7.907895,-3.0998797)(7.7421055,-3.3405046)(7.6,-3.7014422)
}
\end{pspicture}
}
\end{equation*}
and an integer $k$, one has the commutative diagram
\begin{equation}
\label{dva}
\hskip 1em
\xymatrix@C=3.5em{
\bar{2} \ar[r]^{s^{2k}}  \ar[d]_{\id}   & \bar{2} \ar[d]^{s^{-2k}} 
\\
\bar{2} \ar[r]^{\id}  & \ \bar{2}\,.
}
\end{equation}
The sequences of the fibers of both vertical maps are the same 
for each $k$, namely
$\{\bar{1}, \bar{1}\}$, and the derived sequence~(\ref{e1}) consists
of the identities, $\pi_{(1,1)} = \pi_{(2,2)} = \id$. 
Hence the upper horizontal map in~(\ref{dva}) is {\em not} uniquely
determined by its associated derived sequence, which violates the
requirement that the map $\pi$ in diagram~(\ref{eq:5}) is unique.
The map $s^{2k}$ is a nontrivial quasibijection in $\Vines_{\tt Ord}$.

In some categories the extension of the
corner~(\ref{c1}) into~\eqref{eq:5} may not exist.  We illustrate
it on the category of Batanin's $n$-trees,
cf.~\cite[Section~4]{batanin:AM08} for necessary definitions and
notation. The fibers of the map $f'$ of $2$-trees in 
\begin{equation}
\label{Bruhns}
\unitlength.5cm
\raisebox{-10em}{
\begin{picture}(10,10)(3,-.5)
\thicklines
\put(1,-.5){
\multiput(0,0)(0,1){3}{
\multiput(0.2,0)(0.2,0){14}{\makebox(0,0){$\cdot$}}
}
\put(0.5,2.5){\makebox(0,0)[b]{\scriptsize $1$}}
\put(2.5,2.5){\makebox(0,0)[b]{\scriptsize $2$}}

\put(1.5,1){\line(1,1){1}}
\put(1.5,1){\line(-1,1){1}}
\put(1.5,1){\line(0,-1){1}}
}
\put(4.5,.5){\vector(1,0){7}}
\put(8,.8){\makebox(0,0)[b]{{\scriptsize \id}}}
\put(12,-.5){
\multiput(0,0)(0,1){3}{
\multiput(0.2,0)(0.2,0){14}{\makebox(0,0){$\cdot$}}
}
\put(0.5,2.5){\makebox(0,0)[b]{\scriptsize $1$}}
\put(2.5,2.5){\makebox(0,0)[b]{\scriptsize $2$}}

\put(1.5,1){\line(1,1){1}}
\put(1.5,1){\line(-1,1){1}}
\put(1.5,1){\line(0,-1){1}}
}
\put(1,6){
  \multiput(0,0)(0,1){3}{
  \multiput(0.2,0)(0.2,0){14}{\makebox(0,0){$\cdot$}}
  }
  \put(1.5,0){\line(1,1){1}}
  \put(1.5,0){\line(-1,1){1}}
  \put(.5,1){\line(0,1){1}}
  \put(2.5,1){\line(0,1){1}}
\put(1.5,-1){\vector(0,-1){2}}\put(1,-2){\makebox(0,0)[r]{{\scriptsize $f'$}}}
\put(0.5,2.5){\makebox(0,0)[b]{\scriptsize $1$}}
\put(2.5,2.5){\makebox(0,0)[b]{\scriptsize $2$}}
}
\end{picture}
}
\end{equation}
are
\[
\unitlength.5cm
  \begin{picture}(10,2.5)(2,.5)
\thicklines
\put(0,0){
  \multiput(0,0)(0,1){3}{
  \multiput(0.2,0)(0.2,0){14}{\makebox(0,0){$\cdot$}}
  }
  \put(1.5,0){\line(1,1){1}}
  \put(1.5,0){\line(-1,1){1}}
  \put(.5,1){\line(0,1){1}}
\put(0.5,2.5){\makebox(0,0)[b]{\scriptsize $1$}}
\put(-.7,1){\makebox(0,0)[r]{$f'^{-1}(1) =$}}
}
\put(5,.8){\makebox(0,0)[b]{and}}
\put(11,0){
  \multiput(0,0)(0,1){3}{
  \multiput(0.2,0)(0.2,0){14}{\makebox(0,0){$\cdot$}}
  }
  \put(1.5,0){\line(1,1){1}}
  \put(1.5,0){\line(-1,1){1}}
  \put(2.5,1){\line(0,1){1}}
\put(2.5,2.5){\makebox(0,0)[b]{\scriptsize $1$}}
\put(-.7,1){\makebox(0,0)[r]{$f'^{-1}(2) =$}}
\put(3.2,1){\makebox(0,0)[l]{.}}
}
\end{picture}
\]
Take
\[
\unitlength.5cm
  \begin{picture}(10,3.3)(2,.2)
\thicklines
\put(0,0){
  \multiput(0,0)(0,1){3}{
  \multiput(0.2,0)(0.2,0){14}{\makebox(0,0){$\cdot$}}
  }
  \put(1.5,0){\line(1,1){1}}
  \put(1.5,0){\line(-1,1){1}}
  \put(.5,1){\line(0,1){1}}
\put(0.5,2.5){\makebox(0,0)[b]{\scriptsize $1$}}
\put(-.7,1){\makebox(0,0)[r]{$F''_1: =$}}
}
\put(5.74,.8){\makebox(0,0)[b]{and}}
\put(11,0){
  \multiput(0,0)(0,1){3}{
  \multiput(0.2,0)(0.2,0){14}{\makebox(0,0){$\cdot$}}
  }
\put(1.5,0){\line(0,1){2}}
\put(1.5,2.5){\makebox(0,0)[b]{\scriptsize $1$}}
\put(-.7,1){\makebox(0,0)[r]{$F''_2: =$}}
\put(3.2,1){\makebox(0,0)[l]{.}}
}
\end{picture}
\]
Defining the maps in~(\ref{e2}) as $\pi_{(1,1)} = \id$ and taking
$\pi_{(2,2)}$ to be
the obvious unique morphism, it is easy to check that the corner in~\eqref{Bruhns} cannot be completed to~(\ref{eq:5}).  
The unique map
\[
\unitlength.5cm
  \begin{picture}(10,3)(0,.2)
\thicklines
\put(7.4,0){
  \multiput(0,0)(0,1){3}{
  \multiput(0.2,0)(0.2,0){14}{\makebox(0,0){$\cdot$}}
  }
  \put(1.5,0){\line(1,1){1}}
  \put(1.5,0){\line(-1,1){1}}
  \put(.5,1){\line(0,1){1}}
\put(0.5,2.5){\makebox(0,0)[b]{\scriptsize $1$}}
}
\put(4,1){\vector(1,0){2}}
\put(0,0){
  \multiput(0,0)(0,1){3}{
  \multiput(0.2,0)(0.2,0){14}{\makebox(0,0){$\cdot$}}
  }
\put(1.5,0){\line(0,1){2}}
\put(1.5,2.5){\makebox(0,0)[b]{\scriptsize $1$}}
}
\end{picture}
\]
provides an example of a quasibijection in $\hbox{2-{\tt
    Trees}}_{\tt Ord}$ which is not the identity, not even an isomorphism.
\end{example}
%

\begin{definition}
\label{dnes_prednaska_na_Macquarie}
An operadic category $\ttO$ is {\em factorizable\/} 
if each morphism $f \in \ttO$ decomposes, not necessarily uniquely, 
as $\phi  \sigma$ for some $\phi \in \DO$ and
$\sigma \in \QO$ or, symbolically, $\ttO = \DO 
\QO$. 
\end{definition} 

\begin{definition}
\label{zase_jsem_podlehl}
An operadic category $\ttO$ is {\em strongly factorizable\/}  if 
each morphism $f : T\to S $ decomposes  {\em uniquely\/} 
as $\phi \sigma$ for some $\phi \in \DO$ and
$\sigma \in \QO$ such that the induced map between the fibers
\[
\sigma_i:f^{-1}(i)\to \phi^{-1}(i)
\]
is an order-preserving quasibijection for each $i\in |S|$.
\end{definition}

The first part of the following lemma has the same conclusion as  Corollary \ref{zase_mam_nejaky_tik} but the assumptions are different.

\begin{lemma} 
\label{Slozim si Tereje snad uz tuto sobotu.}
In a strongly factorizable operadic category, any order-preserving quasibijection is an identity.  In particular,  the
morphisms on fibers induced by the \qb\ $\sigma$ in the unique
factorization $f = \phi\sigma$  are always the identities.
\end{lemma} 

\begin{proof} Let $\sigma:T\to S$ be an order-preserving quasibijection. Then there are two factorizations of the unique morphism $!:T\to U_c.$ Namely
$T\stackrel{\id}{\to} T\to U_c$ and $T\stackrel{\sigma}{\to} S\to U_c.$
Since such a~factorization must be unique we have $\sigma=\id.$  
\end{proof} 


\begin{lemma} 
\label{I_laboratorni_vysetreni_mne_ceka.}
Assume that in $\ttO$ all \qb{s} are invertible,
 $\ttO$ is  factorizable and satisfies the weak blow-up axiom.
Then $\ttO$ is strongly factorizable and  satisfies the 
blow-up~axiom.
Schematically
\[
\hbox {\rm
 {\QBI} \& {\Fac} \& {\WBU}} 
\Longrightarrow  \hbox {\rm {\SBU} \& \SFac}, 
\]
with the obvious meaning of the abbreviations.
\end{lemma}

\begin{proof} 
Let $f:T\to S$ be a morphism in $\ttO$. We factorize it into a
\qb\ $\omega$ followed by an order-preserving $\eta:T'\to
S$ as in the left upper triangle of
\begin{equation}
\label{Nechce_se_mi_byt_pres_vikend_v_Praze.}
\xymatrix@C=3.5em{T \ar[r]^f \ar[d]^\omega_\sim   &S
\\
T' \ar[ur]^\eta \ar[r]^\pi_\sim &\ Q\,. \ar[u]^\phi
}
\end{equation}
By virtue of Lemma \ref{l2} for $i\in |S|$, let $\pi_i :  \eta^{-1}(i) \to  f^{-1}(i)$ be the 
\qb\ inverse to $\omega_i:
f^{-1}(i)\to \eta^{-1}(i)$. Using the weak blow-up axiom we uniquely 
factorize $\eta$ into  $\phi \kompozice \pi$ such that  $\pi$ on fibers
induces the morphisms $\pi_i$, $i \in |S|$, see the lower right triangle
of~(\ref{Nechce_se_mi_byt_pres_vikend_v_Praze.}). Notice that
$\pi$ is a quasibijection as well. We thus
have a factorization of $f$ into a  \qb\ $\sigma  := \pi \kompozice \omega$
followed by $\phi\in \DO$. By functoriality of the fiber functor, $\sigma$ induces the identities
$\sigma_i = \pi_i \omega_i = \id$ of
the fibers.

Suppose we have two such factorizations of $f$, namely   
\[
\xymatrix@R=1em@C=3.5em{
&Q'\ar[rd]^{\phi'} \ar@{-->}[dd]^p_\sim   &
\\
T \ar[ru]^{\sigma'}_\sim\ar[rd]^{\sigma''}_\sim   &&S \,.
\\
&Q''\ar[ru]^{\phi''}&
}
\]
By our assumptions \qb{s} are invertible, hence we have a unique \qb\
$p:Q'\to Q''$ which induces identities of the fibers over $S$ by Corollary \ref{zase_mam_nejaky_tik}. It follows from the
uniqueness part of the weak blow-up  axiom that $p =
\id$. So the decomposition $f= \phi \kompozice \sigma$ is unique, thus $\ttO$ is
strongly~factorizable.

It remains to prove the general version of the blow-up axiom. Let 
\begin{equation}
\label{Zase mne boli koleno ze nemohu behat.}
\xymatrix@C=3.5em{S' \ar[d]_{f'}  &
\\
T' \ar[r]^\sigma_\sim &T''
}
\end{equation}
be the corner for the blow-up axiom as in~(\ref{c1}) and 
$\pi_{(i,j)} : \inv{f'}(i) \to F''_j$,\ $j = |\sigma|(i)$ for $i \in
|T'|$, a~collection of maps. 
By the weak blow-up axiom  we have a unique factorization 
\hbox{$S'\stackrel{\gamma}{\to} S''\stackrel{g}{\to} T'$} of $f'$  as in
\[
\xymatrix@C=3.5em{S' \ar[d]_{f'}  \ar[r]^\gamma   
& S'' \ar[r]^\varpi_\sim \ar[dl]_g & Q\ar[dl]_\eta
\\
T' \ar[r]^\sigma_\sim &T''&
}
\]
such that $\gamma_i = \pi_{(i,j)}$ for $i \in |T'|$.
We then apply the strong factorization axiom to $\sigma \kompozice g$ and get a
factorization $S''\stackrel{\varpi}{\to} Q\stackrel{\eta}{\to} T''$ 
where $\varpi \in \QO$  and  $\eta \in\DO$. 

Since $\varpi$ is a \qb,
the derived sequence $\varpi_{(i,j)}$ consists of
order-preserving quasibijections. We already established that
$\ttO$ is strongly factorizable,  thus each $\varpi_{(i,j)}$ is the
identity by Lemma~\ref{Slozim si Tereje snad uz
  tuto sobotu.}, therefore
\[
(\varpi \kompozice \gamma)_{(i,j)} =
\varpi_{(i,j)} \kompozice \gamma_i = \gamma_i =
\pi_{(i,j)}, \ \hbox {for each  $i \in |T'|$.}
\]
We conclude that
\begin{equation}
\label{Dnes bude svickova.}
\xymatrix@C=3.5em{S' \ar[d]_{f'}\ar[r]^\pi  & Q\ar[d]^\eta
\\
T' \ar[r]^\sigma_\sim &T''
}
\end{equation}
with $\pi := \varpi \kompozice \gamma$ 
is a completion of the corner~\eqref{Zase mne boli koleno ze nemohu
  behat.} required by the  blow-up axiom.

To prove that the completion~\eqref{Dnes bude svickova.} is unique, 
assume that
$S'\stackrel{\pi'}{\to} Q'\stackrel{\eta'}{\to} T''$ is another
completion of~\eqref{Zase mne boli koleno ze nemohu behat.}, and~let 
\[
\xymatrix@R=1em{&P\ar[dr]^\lambda  &
\\
S'\ar[rr]^\pi \ar[ur]^\rho && Q
} \raisebox{-1em}{\hbox{\hskip 1em resp.\ \hskip 1em}}
\xymatrix@R=1em{&P'\ar[dr]^{\lambda'}  &
\\
S'\ar[rr]^{\pi'} \ar[ur]^{\rho'} && Q'
}
\]
be the unique factorizations of $\pi$ resp.~$\pi'$ to a quasibijection
inducing identities on fibers followed by an order-preserving map. The
existence of factorizations of this type is
guaranteed by the already proven $\SFac$ combined
with $\QBI$ assumed in Corollary~\ref{zase_mam_nejaky_tik}.
Consider the commutative diagrams
\begin{equation}
\label{Za chvili pujdu s Jarkou na svickovou.}
\xymatrix@C = +2.5em@R = +1em{ & P \ar[dd]^(.3){\eta\lambda} \ar[dr]^\lambda & 
\\
S'  \ar[ur]^\rho    \ar@{-}[r]^(.7){\pi}\ar[dr]_a & \ar[r] & \ Q\ , \ar[dl]^\eta
\\
&T''& 
}
\hskip 1em
\xymatrix@C = +2.5em@R = +1em{ & P' \ar[dd]^(.3){\eta'\lambda'} 
\ar[dr]^{\lambda'} & 
\\
S'  \ar[ur]^{\rho'}    \ar@{-}[r]^(.7){\pi'}\ar[dr]_a & \ar[r] & Q'\,, \ar[dl]^{\eta'}
\\
&T''&
}
\end{equation}
where $a := \eta\pi = \eta'\pi'$,  and take $j\in |Q| =  |Q'|,\ i :=
|\eta|(j) =  |\eta'|(j)$.
By Axiom (iii) of an operadic category the diagram
\begin{subequations}
\begin{equation}
\label{rho}
    \xymatrix@C = +1em@R = +1em{
      a^{-1}(i)      \ar[rr]^{\rho_i} \ar[dr]_{\pi_i} & & (\eta\lambda)^{-1}(i) \ar[dl]^{\lambda_i}
      \\
      &\eta^{-1}(i)&
    }
\end{equation}
related to the left diagram in~\eqref{Za chvili pujdu s Jarkou na
  svickovou.} commutes and induces a morphism
$(\rho_i)_j: \pi_i^{-1}(j)\to \lambda_i^{-1}(j)$ between fibers. 
By Axiom (iv) we have
\[
\lambda^{-1}(j)=\lambda_i^{-1}(j) \ \mbox{and} \
\pi^{-1}(j)=\pi_i^{-1}(j).
\] 
Axiom (v) thus gives
$\rho_j = (\rho_i)_j$,  but $\rho_j= \id$  by construction, hence
$(\rho_i)_j =\id.$ 

Since $\lambda$ and $\eta$ are order preserving,
the morphism $\lambda_i$ in~\eqref{rho} is also order preserving.  
We see that~(\ref{rho}) represents the unique factorization of the
morphism $\pi_i$ to a~quasibijection inducing identities on fibers
followed by an order-preserving map.  In exactly the same manner
we obtain a unique factorization
\begin{equation}\label{rho'}
    \xymatrix@C = +1em@R = +1em{
      a^{-1}(i)      \ar[rr]^{\rho'_i} \ar[dr]_{\pi'_i} & & (\eta'\lambda')^{-1}(i) \ar[dl]^{\lambda'_i}
      \\
      &\eta'^{-1}(i)&
    }
\end{equation}
\end{subequations}
related to the right diagram in~\eqref{Za chvili pujdu s Jarkou na svickovou.}.

Notice that, by assumption, the induced morphism
$\pi_i:a^{-1}(i)\to \eta^{-1}(i)$ in~(\ref{rho}) is equal to
$\pi'_i:a^{-1}(i)\to \eta'^{-1}(i)$ in~(\ref{rho'}).
Hence 
$\rho_i = \rho'_i$, and thus the quasibijection
$q := \rho\rho'^{-1}$  in the diagram
\[
\xymatrix@R=1em@C=3.5em{
&P'\ar[rd]^{\eta'\lambda'} \ar@{-->}[dd]^q_\sim   &
\\
S' \ar[ru]^{\rho'}_\sim\ar[rd]_{\rho}^\sim   &&T''
\\
&P\ar[ru]_{\eta\lambda}&
}
\]
induces 
the identity maps between the fibers of $\eta\lambda$ and~$\eta'\lambda'$.
The uniqueness part of the weak blow-up axiom tells us that $q= \id$, $P= P'$ and
$\eta\lambda = \eta'\lambda'$. Using the above facts, we can
modify the right diagram in~(\ref{Za chvili pujdu s Jarkou na
  svickovou.}) so that we will now be comparing the diagrams
\[
\xymatrix@C = +2.5em@R = +1em{ & P \ar[dd]^(.3){\eta\lambda} \ar[dr]^\lambda & 
\\
S'  \ar[ur]^\rho    \ar@{-}[r]^(.7){\pi}\ar[dr]_a & \ar[r] & Q \ar[dl]^\eta
\\
&T''& 
}
\hskip 1em \raisebox{-2.3em}{\hbox{and}} \hskip 1em
\xymatrix@C = +2.5em@R = +1em{ & P \ar[dd]^(.3){\eta\lambda} 
\ar[dr]^{\lambda'} & 
\\
S'  \ar[ur]^{\rho}    \ar@{-}[r]^(.7){\pi'}\ar[dr]_a & \ar[r] 
&Q'\,. \ar[dl]^{\eta'}
\\
&T''&
}
\]

By assumption, the morphisms between fibers of $a$ and $\eta$,
resp.~$a$ and $\eta'$, induced by $\pi$, resp.~$\pi'$, coincide. Since
$\rho$ is invertible by $\QBI$, the morphisms induced by
$\lambda = \pi \rho^{-1}$, resp.~$\lambda' = \pi' \rho^{-1}$, between
the fibers of $\eta\lambda$ and $\eta$, resp.\ $\eta\lambda$ and $\eta'$,
coincide as well, thanks to the functoriality of the fiber functors.
Using the uniqueness in $\WBU$ we conclude that $\lambda = \lambda'$
and $\eta = \eta'$, thus finally $\pi = \pi'$ and $\eta= \eta'$ as
required. Notice that Axiom~(v) of an
operadic category played a crucial role in the second part of the proof.
\end{proof}

\begin{lemma}
\label{Podivam_se_do_Paramaty?}
Any isomorphism in an operadic category has  local 
terminal objects as its
fibers. Conversely, in a factorizable operadic category in which all
\qb{s} are isomorphisms and the weak blow-up axiom is fulfilled, a morphism whose fibers are
local terminals  is an~isomorphism.
\end{lemma}

\begin{proof}
 Let $\phi:S\to T$ be an isomorphism with inverse $\psi$. Consider
 the commutative diagram over~$T$:
\[
\xymatrix@C=2em@R=1.2em{S\ar[r]^{\phi} 
\ar[dr]_{\phi} &T\ar[r]^{\psi} \ar[d]^(.4){\id}&\ S  \ar[ld]^{\phi}
\\
&\phantom{.}T \,.&
}
\]
By functoriality of the fiber functor this diagram induces isomorphisms from the fibers of 
$\phi$ to the fibers of the identity morphism of~$T$. 
Therefore the fibers of $\phi$ are isomorphic to trivial objects,
so they are all local~terminal.

Conversely,  suppose an operadic category $\ttO$ is factorizable with
all
\qb{s} isomorphisms, and suppose that all fibers of $\phi:A\to T$ are local
terminals. By assumption, one can factorize $\phi$ as a \qb\
$\sigma$ followed by $\xi \in \DO$. The \qb\ $\sigma$ induces 
\qb{s}, hence isomorphisms, between the fibers of $\phi$ and
$\xi$. So it will be enough to show that any $\xi: R\to S$ in $\DO$
whose fibers are
local terminals is an isomorphism.

Let $F_i$ denote the fiber of $\xi$ over $i \in |S|$. Since each
$F_i$ is local terminal,  we have by assumption the
unique isomorphism $\xi_i:F_i\to U_{c_i}$ for each $i$ and some $c_i\in \pi_0(\ttO)$, and
its inverse $\eta_i :  U_{c_i} \to F_i$.  By the weak blow-up axiom 
there exists a unique
factorization of $\id:S\to S$ as $S\stackrel{a}{\to} 
Q\stackrel{b}{\to}
S$ such that $a$  induces the morphisms $\eta_i$ on the fibers.  
The following diagram
\[
\xymatrix@C=2em@R=1.2em{R\ar[r]^{\xi} 
\ar[dr]_{\xi} &S\ar[r]^{a} \ar[d]^(.3){\id}&Q  \ar[ld]^{b}
\\
&S&
}
\]
in $\DO$ commutes
and by functoriality it induces the identity morphisms between the
fibers of $\xi$ and $b$. By the uniqueness part of the weak 
blow-up axiom $Q=R$, we have $b= \xi$ and $\xi\kompozice a = \id_R$. 
Repeating the same argument we find also that $a\kompozice \xi = \id_S$,
hence $\xi$ is an isomorphism. 
\end{proof}

\begin{lemma}
\label{Podival_jsem_se_do_Paramaty!}
In a factorizable operadic category $\ttO$ in which all
\qb{s} are isomorphisms and the weak blow-up axiom is fulfilled,
each $f \in \ttO$ decomposes as $\psi \kompozice \omega$,
where $\omega$ is an isomorphism, $\psi$ is order preserving and all  local terminal fibers of $\psi$  are trivial.
\end{lemma}

\begin{proof}
Decompose $f$ into $A \stackrel\sigma\to X \stackrel\phi\to B$ with
$\sigma$ a
\qb\  and $\phi \in \DO$ using the factorizability in $\ttO$. 
By the weak blow-up axiom, one has the diagram
\begin{equation}
\label{v_Koline_s_chripkou}
\xymatrix{X \ar[rr]^{\tilde \sigma} \ar[dr]_\phi && Y \ar[ld]^\psi
\\
&B&
}
\end{equation}
in which $\psi$ has the same non-terminal fibers as $\phi$ and all
its terminal fibers  are trivial, and $\tilde \sigma$
induces the identity maps between non-terminal fibers. By Axiom (iv) of operadic categories a fiber of $\tilde\sigma$ is either equal to a fiber of the identity or to a fiber of a morphism $!:t\to U_c$ for some $c\in \pi_0(\ttO)$ where $t$ is a local terminal. Thus all 
fibers of $\tilde \sigma$ are local terminals, so it is an \im\ by Lemma
\ref{Podivam_se_do_Paramaty?}. The desired factorization of $f$ is then
given by $\psi$ in~(\ref{v_Koline_s_chripkou}) and 
$\omega := \tilde \sigma \kompozice \sigma$.
\end{proof}

In the rest of the paper, the notation 
$F \FIB T  \stackrel{\phi}\longrightarrow t$
or 
$F \FIB T \to t$
 when $\phi$ is understood 
 will express that
$\phi: T \to t$ is the unique map to a local terminal object $t$
and that $F$ is the fiber of $\phi$. It follows from
Axiom~(i) of an operadic category that each local terminal object has
cardinality $1$, so $F$ is unique and is uniquely determined by $t$. 
 
\begin{definition}
\label{Kveta_asi_spi.}
The {\em unique fiber condition\/} for an operadic category $\ttO$
requires that, if the fiber of the unique morphism 
$\phi:T\to t$ to a local terminal object $t$ is $T$, then $t$ is a
chosen local terminal object $U_c$  for some $c\in \pi_0(\ttO)$. In other words, the only situation when
$T \FIB T \to t$ is when $t$ is trivial.
\end{definition}

\begin{lemma} 
\label{twolocal}
Let $F \FIB T  \stackrel{\phi'}\longrightarrow t'$ 
and $F \FIB T \stackrel{\phi''}\longrightarrow t''$ be morphisms to local
terminal objects with the same fiber $F$. If the weak blow-up and unique
fiber conditions are satisfied, then $\phi' = \phi''$.
\end{lemma}

\begin{proof} Consider the commutative triangle 
\[
\xymatrix@C=2em@R=1.2em{T \ar[rr]^{\phi''} 
\ar[dr]_{\phi'} && t'' \ar[ld]^{\xi}
\\
&t'&
}
\]
in which $\xi$ is the unique map between the local terminal objects.
We have the induced morphism of fibers $\phi''_1:F\to 
\inv{\xi}(1)$. By Axiom~(iv) of the 
operadic categories 
the fiber of this morphism is $F$. 
As $\inv{\xi}(1)$ is local terminal by Lemma~\ref{Podivam_se_do_Paramaty?},
the unique fiber condition implies that $\phi''_1 = \, !:F\to U_c$ is the unique map to a
chosen local terminal object.  This means that the fiber
$\xi^{-1}(1)$ is $U_c$, so $\xi$ is a \qb. By
Corollary~\ref{zase_mam_nejaky_tik}, $\xi$
must be the identity.
\end{proof}

\begin{definition} 
\label{Porad_nevim_jestli_mam_jit_na_ty_narozeniny.}
An operadic category $\tt O$  is
{\em rigid\/} if, given $\phi \in \DO$, the only \im\ $\sigma$ that makes
\begin{equation}
\label{Skrabe_mne_v_krku.}
\xymatrix@C=3.5em{S \ar[d]_{\phi} \ar@{=}[r]  & S \ar[d]^{\phi}
\\
T \ar[r]^\sigma_\cong &T
}
\end{equation}
commutative is the identity $\id : T = T$.
\end{definition}

\begin{example} The category $\Fin$ is not rigid, but its subcategory
$\Surj$ of nonempty finite sets and their surjections is. \end{example}

\begin{definition} 
\label{constant-free}
An operadic category $\ttO$ is {\em constant-free\/} 
if $|f|$ is surjective for each $f \in \ttO$. 
Equivalently, $\ttO$~is constant-free
if the cardinality functor $\ttO\to \Fin$  factorizes through
the operadic category $\Surj$.
\end{definition}

\begin{lemma}
\label{Podivam_se_do_Paramaty??} If a constant-free operadic category satisfies 
the weak blow-up and the unique fiber conditions, then it is
 rigid. Schematically
\[
\hbox {\rm {\UFB} \& {\WBU}} 
\Longrightarrow  \hbox {{\Rig}}. \]

\end{lemma}

\begin{proof}
 
Since the category of finite sets and surjections
is obviously rigid, one has $|\sigma|
  = \id$  for $\sigma$
in~(\ref{Skrabe_mne_v_krku.}). For each $i\in |T|$ we have the induced morphism of the fibers
\[
\phi_i:  \phi^{-1}(i)   \to \sigma^{-1}(i)
\]
whose unique fiber is $\phi^{-1}(i)$ by Axiom~(iv) of an
operadic category. 
The fiber $\sigma^{-1}(i)$ is local terminal by
Lemma~\ref{Podivam_se_do_Paramaty?}, thus, 
by the unique fiber condition,
$\sigma^{-1}(i)$ is trivial, so $\sigma$ is a \qb. Hence, it
must be the identity by Corollary~\ref{zase_mam_nejaky_tik}.
\end{proof}

\section{The operadic category of graphs}
\label{Ceka_mne_Psenicka.}

In this section we introduce the operadic category $\Gr$ of ordered
graphs.  The adjective ``ordered''
indicates that the sets of flags, vertices and legs of
graphs in $\Gr$ have prescribed total orders. The
category $\Gr$ and its modifications will play the fundamental r\^ole in the
second paper of the series \cite{sydney2}.
We will prove that it is a constant-free strongly
factorizable operadic category satisfying the  blow-up axiom in
which all \qb{s} are invertible.
Moreover, $\Gr$ is strictly graded (see Definitions \ref{grad} and \ref{sgrad}) by the number of
edges. 
We also show that $\Gr$ satisfies the unique fiber condition and is rigid.
We start by a more structured version of the standard concept
of graphs as recalled e.g.\ in
\cite[Definition~II.5.23]{markl-shnider-stasheff:book}.

\begin{definition} 
\label{pre}
 A {\em preordered graph\/} $\Gamma$ is a pair $(g,\sigma)$ consisting
of an  order-preserving map $$g:F\to V, \ V\ne \emptyset$$ in the category $\Fin$  together with an
involution $\sigma$ on $F$.
\end{definition}

Notice that we do not require the geometric 
realization~\cite[Section~II.5.3]{markl-shnider-stasheff:book}
of preordered graphs to be connected.
Elements of $\Flag(\Gamma) := F$ are the {\em flags\/} (also called {\em half-edges\/})
of $\Gamma$ and elements of $\Vert(\Gamma) := V$ are its
{\em vertices\/}. The elements of the set $\Leg(\Gamma)$ of fixed points of $\sigma$ are called the {\em legs\/} of
$\Gamma$ while nontrivial orbits of $\sigma$ are its {\em edges\/}. The
{\em endpoints\/} 
of an edge $e = \{h_1,h_2\}$ are $g(h_1)$ and $g(h_2)$.

For any $v\in V$, the set $g^{-1}(v)$ of flags adjacent to
$v$ inherits a linear order from $F$ which we call 
the {\em local order\/} at $v$. 
We may thus equivalently
define a preordered graph as a map
$g :F\to V$ from a finite set $F$  into
a~linearly ordered set $V$ with the additional data consisting of
linear orders of each  $\inv g(v)$, $v \in V$. The lexicographic order
combining the order of $V$ with the local orders makes $F$ a finite
ordinal, and the two definitions coincide.

We will often use a short notation $(F,V)$ or $(F,g,V)$ for a preordered graph $(g,\sigma)$ if we want to specify its set of vertices and flags only. We hope it will not lead to any confusion.  

A {\em morphism\/} 
of preordered graphs $\Phi:\Gamma\to \Gamma'$ is a pair $(\psi,\phi)$
of morphisms of finite sets such that the diagram
\begin{equation}
\label{graphsmorphism}
\xymatrix@C=3.5em{F \ar[d]_{g}  & F'\ar@{_{(}->}[l]_{\psi}
\ar[d]^{g'}
\\
V \ar@{->>}[r]^{\phi} &V'
}
\end{equation}
commutes. We moreover require $\phi$ to be a surjection and require $\psi$ 
to be 
equivariant with respect to the involutions
and induce
a bijection on fixed points. Thus $\psi$
injectively maps flags to flags and bijectively legs to legs.
The pair $(\psi,\phi)$ must satisfy the following condition: If $\phi(i) \ne \phi(j)$ and
$e$ is an edge with endpoints $i$ and $j$ then there exists an edge
$e'$ in $\Gamma'$ with endpoints $\phi(i)$ and $\phi(j)$ such that $e=
\psi(e')$.
Notice that we denote by the same symbol both the map of
  flags and the obvious induced map of edges.
Preordered graphs and their morphisms form a category {\em of preordered graphs~$\ttprGr$.}

\begin{remark} 
Our definition of morphism of graphs goes back to
Getzler\textendash{}Kapranov~\cite{getzler-kapranov:CompM98} and 
Borisov\textendash{}Manin~\cite{borisov-manin}. It is simultaneously
  more structured than the former one since
  we want to take orders of flags, legs and vertices into account, and less complicated (but still
  more structured) that the latter. A detailed
  discussion of Borisov\textendash{}Manin's definition can be found in
  \cite{kock-cospan}.
\end{remark}

The {\em fiber\/} $\inv{\Phi}(i)$ 
of a map $\Phi = (\psi,\phi): \Gamma \to \Gamma'$ 
in~(\ref{graphsmorphism}) over $i\in V'$ is a preordered
graph whose set of vertices is $\phi^{-1}(i)$ and whose
set of flags is $(\phi g)^{-1}(i).$ 
The involution $\tau$ of
$\inv{\Phi}(i)$ 
is defined as 
\[
\tau(h) := 
\begin{cases}   h& \hbox {if $h\in \Im(\psi)$}
\\ 
\sigma(h)& \hbox {if $h\notin \Im(\psi)$,} 
\end{cases}
\]
where $\sigma$ is the involution of $\Gamma$.
Observe that $h\notin \Im(\psi)$ if
and only if $\sigma(h)\notin \Im(\psi)$.

\begin{definition} 
\label{pisu_jednou_rukou} 
Let $(\psi,\phi): \Gamma \to \Gamma'$ be a map of preordered graphs. 
\begin{itemize}\item[(i)] The map
$(\psi,\phi)$ is  called a {\em local reordering\/} if  $\phi = \id$ and $\psi$ is 
an isomorphism. 
\item[(ii)] The map $(\psi,\phi)$ is  called a  {\em local isomorphism\/} 
 if $\phi$ is a bijection and
$\psi$ restricts to an order preserving isomorphism $\inv{g'}(j)\cong
\inv{g}(i)$ for each $i \in V'$, $j = \phi(i)$.
\item[(iii)] The map $(\psi,\phi)$ is  called a  {\em contraction\/} 
if  $\phi$ is order~preserving. A contraction will be also called an 
{\em order preserving\/} morphism.
\item[(iv)] The map $(\psi,\phi)$ is  called a  {\em pure contraction\/} if both $\psi$ and $\phi$ are order preserving. \end{itemize} \end{definition}

\begin{lemma}
  \label{uniquepurecontraction} Let
  $\Phi_0=(\psi_0,\phi_0): (F,g,V) = \Gamma \to \Gamma_0 =
  (F_0,g_0,V_0) $ and
  $\Phi_1=(\psi_1,\phi_1): \Gamma \to \Gamma_1 = (F_1,g_1,V_1) $ be
  two pure contractions such
  that 
\[
\phi:=\phi_0 = \phi_1:V\to V_0 = V_1,
\] 
that is $\Phi_0$ and
    $\Phi_1$ are equal on vertices, and
\[
\Phi_0^{-1}(i) = \Phi_1^{-1}(i)
\] 
for  $i\in V_0$,  that is, $\Phi_0$ and $\Phi_1$ have equal fibers.
Then $\Gamma_0 = \Gamma_1$ and $\Phi_0 = \Phi_1.$ 
\end{lemma}

\begin{proof} We have to prove that $F_0 = F_1.$ For this let
  $h_0\in F_0, \ i = g_0(h_0)$ and consider
  $h=\psi_0(h_0) \in (\phi g)^{-1}(i).$ Since $h\in {\it Im}(\psi_0)$
  the flag $h$ is a fixed point of the involution $\tau_0$ of the fiber
  $\Phi_0^{-1}(i).$ Then it is also a fixpoint of the involution
  $\tau_1$ of the fiber $\Phi^{-1}_1(i).$ Hence, there is a unique
  $h_1 \in F_1$ such that $h = \psi_1(h_1).$ We then have a map
  $p:F_0\to F_1$ which is obviously an order preserving bijection
  which commutes with $g_0$ and $g_1.$ So $F_0 = F_1$ and, moreover,
  $g_0 = g_1.$
\end{proof}

\begin{definition}\label{contractiondata}  
Let $\Gamma = (F,{g}, V)$ be a preordered graph. 
\begin{itemize}
\item[(i)] {\em Contraction data} for $\Gamma$ consists of an order preserving surjection
$\phi: V\epi V'$ and for each $i\in V'$ a $\sigma$-free and $\sigma$-closed subset $E_i 
 \subset(\phi g)^{-1}(i).$ 
\item[(ii)] For $i\in V'$, the  $i$th     {\em fiber associated to the contraction data } is the graph $\Gamma_i$ given as the restriction $F_i:= (\phi g)^{-1}(i) \to \phi^{-1}(i)=: V_i$ of $g$, along the involution which agrees with $\sigma$ for the flags in $E_i$ and is otherwise trivial.  

\end{itemize}
\end{definition}

The following Lemma shows that the contraction data are in one-to-one 
correspondence with pure contraction morphisms with domain  $\Gamma$ in which   $E_i$, $i \in V'$ play the roles of  sets of edges which we contract to a vertex $i.$

\begin{lemma}
\label{contractions}  
Given a contraction data $(\phi,E_i)$ for $\Gamma = (F,g,V)$ there is  a 
preordered graph $\Gamma'$ together with a contraction $(\psi,\phi):\Gamma \to \Gamma'$ whose vertex map  $\phi$ is that from the contraction data and whose fibers are the fibers associated to the contraction data.
Moreover, there is a unique such graph for which $(\psi,\phi)$ is a pure contraction. 
\end{lemma}

\begin{proof} We construct $\Gamma'$ as the graph whose set of
  vertices is $V'$ and whose set of flags  is $F':=
  F \setminus \bigcup_{i\in V'}E_i$.  The map $g':F'\to V'$ is the
  restriction of the composite $\phi \kompozice g$, as shown in
\[
\xymatrix@C=3.5em{F \ar[d]_{g}  &  \  F':=
  F\setminus \bigcup_{i\in V'}E_i \ar@{_{(}->}[l]_(.7){\psi}
\ar[d]^{g'}
\\
V \ar@{->>}[r]^{\phi} &V'\,.
}
\]
 It is easy to see that
$(\psi,\phi)$ is a
pure contraction. Uniqueness follows from Lemma
\ref{uniquepurecontraction}. 
\end{proof}

Despite the fact that preordered graphs do not form an operadic
category, the following  version of the weak blow-up 
condition for pure contractions makes sense.

\begin{lemma}
\label{bupforcontractions} 
Let $\Phi = (\psi,\phi):\Gamma= (F,g,V)\to \Gamma' = (F',g',V')$ be a pure contraction with fibers
$\Gamma_i = (F_i,V_i)$, $i \in V'$. 
Given pure contractions $\Xi_i:\Gamma_i\to
\Lambda_i$ for each $i\in V'$, there exists a unique
factorization of $\Phi$ as a composite of pure contractions
\begin{equation}
\label{po_dlouhe_dobe_opet_pisu}
\xymatrix@C=2em@R=1.2em{\Gamma\ar[rr]^a \ar[dr]_{\Phi} && \Lambda \ar[ld]^{b}
\\
&\Gamma'&
}
\end{equation} 
such that the induced map $a_i$ of the fibers equals $\Xi_i$, $i \in V'$.
\end{lemma}

\begin{proof} 
Assume that the pure contraction $\Phi$ is given, as in Lemma
\ref{contractions}, by an order-preserving map $\phi : V \to V'$ and subsets $E_i$, $i\in V'$, of
edges. 
Suppose also that the pure contractions  $\Xi_i$ are given by
order-preserving maps
$\phi_i : \inv \phi (i) \epi C_i$, $i \in V'$, and 
subsets $E_{ij} \subset E_i$ of 
edges of $\Xi_i^{-1}(j)$ for each $j \in C_i$. We then use 
Lemma~\ref{contractions} again to build $\Lambda$ with the set of
vertices~$C$, and a pure contraction $a$
as follows. As $C$ we take
the ordinal sum 
$\bigcup_{i \in V'_i} C_i$ and
\begin{equation}
\label{vcera_do_mne_zase_sili}
\phi_a  :=  \bigcup_{i \in V'_i} \phi_i    : V =  \bigcup_{i \in V'_i}
\phi^{-1}(i) \longrightarrow C.
\end{equation} 
The pure contraction $a$ is then determined by $\phi_a : V \to C$ and
the subsets of edges $E_{ij}$,
$j \in C_i$, $i \in V'$.
It is  easy to
check that $\Gamma'$ is a result of a further pure contraction~$b$ associated to the subsets of edges $E_i\setminus \bigcup_{j}E_{ij}.$  
The uniqueness of the construction is clear.
\end{proof}

Another version of the weak blow-up condition is described in

\begin{lemma}\label{bupforisos}
Let $\Phi:\Gamma= (F,g,V) \to \Gamma' =  (F',g',V')$ be a pure
contraction with fibers $\Gamma_i$, $i \in V'$. Given  local
isomorphisms $\Xi_i:\Gamma_i\to \Lambda_i$ for each $i\in V'$,
there exists a unique factorization of $\Phi$ as
in~(\ref{po_dlouhe_dobe_opet_pisu})
in which $a$ is a local isomorphism inducing 
the prescribed maps $\Xi_i$ on the fibers, and $b$ a pure~contraction.
\end{lemma}

\begin{proof} 
Let $\Lambda_i = (G_i,C_i)$. 
We construct $\Lambda$ 
as the graph whose set of vertices $C$ equals
the ordinal sum  $\bigcup_{i \in V'_i} C_i$, and the set $F''$ of 
flags the ordinal sum $\bigcup_{i \in V'_i} G_i$. There is an
obvious isomorphism $\psi_a$ between the set $F''$ of flags of $\Lambda$ and the
set  $F$ of flags of $\Gamma$ induced by the local isomorphism between the
fibers. We transport the involution of $\Gamma$ to the flags of
$\Lambda$ along this isomorphism. Then $a := (\psi_a,\phi_a)$ 
with $\phi_a$ as in~(\ref{vcera_do_mne_zase_sili}) is the requisite
local isomorphism. It is easy to
check 
that $\Gamma'$ is a result of a  pure contraction of $\Lambda$  for which 
the contraction data consist of the order preserving map $\phi_b: \bigcup_{i \in V'_i} C_i\to V', \ \phi_b(c) = i \ \text{if} \ c\in C_i$ together with the  collection of subsets $\{\psi_a^{-1}(E_i), i\in V'\}$ where $\{E_i, i \in V'\}$ is the collection of the  contraction subsets for $\Phi.$  The uniqueness of factorization  is clear again. \end{proof}

The last version of the weak blow-up axiom which we will need is 

\begin{lemma}
\label{local}
Let $\Phi:\Gamma= (F,g,V) \to \Gamma' =  (F',g',V')$ be a morphism between preordered graphs with fibers $\Gamma_i$, $i \in V'$. 
Given  local
reorderings $\Xi_i:\Gamma_i\to \Lambda_i$ for each $i\in V'$,
there exists a unique factorization of $\Phi$ as
in~(\ref{po_dlouhe_dobe_opet_pisu})
in which $a$ is a local reordering that induces
the prescribed maps on the fibers.
\end{lemma}

Notice that we did not require $\Phi$ to be a
pure contraction. When $\Phi$ is a pure contraction,  $b : \Lambda \to
\Gamma'$ need not be pure, but it is a contraction. 

\begin{proof}[Proof of Lemma \ref{local}]
Each vertex $j$ of $\Gamma$
belongs to a unique fiber of $\Phi$. So the prescribed reorderings of
the fibers determine a reordering at each vertex of  $\Gamma$.  
We thus  construct $\Lambda$ as the graph with the same
vertices as $\Gamma$ but with the local orders modified according to
the above reorderings.
The map $a: \Gamma \to \Lambda$ 
is then the  related  local reordering map. Since it is an
isomorphism, it determines the map $b : \Lambda\to \Gamma'$ uniquely. 
\end{proof}

\begin{proposition}
\label{fact} 
Any morphism  $\Phi$  of preordered graphs
\begin{equation*}
\xymatrix@C=3.5em{F \ar[d]_{g}  & F'\ar@{_{(}->}[l]_{\psi}
\ar[d]^{g'}
\\
V \ar@{->>}[r]^{\phi} &V'
}
\end{equation*}
 can be
  factorized as a local isomorphism followed by a  pure 
  contraction followed by a local reordering. Symbolically
\begin{equation}
\label{mam_prohlidky}
\Phi = {\it Reo}\kompozice {\it Cont}\kompozice {\it Li}.
\end{equation}
\end{proposition}

\begin{proof}
We first factorize $\phi$ as a bijection $\pi:V\to V''$ followed by an order-preserving map $\xi: V'' \to V'$ such that $\pi$ restricts to an order-preserving
isomorphism $\phi^{-1}(i) \cong \xi^{-1}(i)$ for each $i\in V'$, 
cf.~the bottom row of
\begin{equation}
\label{vcera_s_Jarkou_na_Korabe}
\xymatrix@C=4em{F \ar[r]^\eta_\cong  \ar[d]_{g}     &  \ar[d]_{g''} F'' 
& F' \ar@{^{(}->}[l]_{\eta \kompozice \psi}  \ar[d]_{g'} 
 \ar@/_2em/[ll]_{\psi} 
\\
\ar@{->>}@/_2em/[rr]^{\phi} V \ar[r]^\pi_\cong&  \ar@{->>}[r]^{\xi} V'' &\ V'\,.  
}
\end{equation}
We then factorize  $\pi \kompozice g$ into the composite
 $F\stackrel{\eta}{\longrightarrow} F''\stackrel{g''}{\longrightarrow} V''$ where 
$\eta$ induces an order-preserving isomorphism $(\pi g)^{-1}(j)
\cong (g'')^{-1}(j)$ for each $j \in V''$, cf.~the left square
in~(\ref{vcera_s_Jarkou_na_Korabe}). 
We induce an involution on $F''$ from $F$ via the isomorphism
$\eta$. The pair $(\eta^{-1},\pi)$ is the required local isomorphism $ {\it Li}$
in~(\ref{mam_prohlidky}).

The pair $( \eta \kompozice\psi, \xi)$ in the right  square
of ~(\ref{vcera_s_Jarkou_na_Korabe}) is a morphism of graphs as well. We
factorize $\eta \kompozice \psi$ as a bijection $\mu:F'\to F'''$ followed
by an order-preserving monomorphism $\lambda:F'''\to F''$ as in
\[
\xymatrix@C=3em{F'' \ar[d]_{g''} & \ar@{_{(}->}[l]_\lambda    
F''' \ar[dr]^(.6){g'''}    & 
 \ar[l]_\mu^\cong F'\ar@/_2em/[ll]_{\eta \kompozice \psi}
  \ar[d]^{g'}  
\\
V'' \ar@{->>}[rr]^\xi  && \ V'\,.
}
\]
We finally define $g''':F'''\to V'$ as $\xi\kompozice g''\kompozice
\lambda$.
Since  $\xi\kompozice g''\kompozice \lambda \kompozice \mu = g'$, the diagram
\[
\xymatrix@C=3.5em{F''' \ar[d]_{g'''}  & F'\ar[l]^\cong_{\mu}  \ar[d]^{g'}
\\
V' \ar[r]^{\id} & V'
}
\]   
commutes. It is a reordering morphism playing the r\^ole of
${\it Reo}$ in~(\ref{mam_prohlidky}). The pair $(\lambda, \xi)$, which
is clearly a pure contraction, is ${\it Cont}$ in~(\ref{mam_prohlidky}).
\end{proof}

\begin{corollary} 
Any isomorphism of preordered graphs can be factorized into a local
isomorphism followed by a local reordering, symbolically
${\it Iso} = {\it Reo}\kompozice \, {\it Li}$.
\end{corollary}

\begin{proof}
The statement follows from Proposition~\ref{fact} combined with the obvious
fact that the only pure contractions that are isomorphisms are the
identity maps.   
\end{proof}

\begin{corollary}  
Any morphism $\Phi = (\psi,\phi)$ such that $\phi: V' \to V''$ is
order preserving is a~composite of a pure contraction
followed by a local reordering. 
\end{corollary}
   
\begin{proof}
Another consequence of Proposition~\ref{fact}.
Notice that the decomposition of $\phi$ in the bottom row
of~(\ref{vcera_s_Jarkou_na_Korabe}) was specified so that if $\phi$
is order preserving, $\pi$ must be 
the identity, thus $\eta = \id$ as well, so 
${\it  Li}$ in~(\ref{mam_prohlidky}) is the identity morphism.
\end{proof}

For each natural number $n \geq 0$, let $1_n$ (the {\em corolla\/}) be the graph
$\bar{n}\to \bar{1}$ with the trivial involution. The corollas are
not  local terminal objects in $\ttprGr$ since there are exactly
$n!$ morphisms from any graph $\Gamma$ with $n$ legs to $1_n.$ Any
such a morphism is completely determined by a linear order of the legs
of $\Gamma.$

\begin{definition}
\label{Radeji_bych_sel_s_Jaruskou_na_vyhlidku_sam.}
The category of {\em ordered graphs $\Gr$\/} is the coproduct of
the categories ${\ttprGr}/1_n$ for $n\ge 0$. \end{definition} 

\begin{lemma} The category of ordered $\Gr$ is equipped with  an operadic category structure.
\end{lemma}

\begin{proof}
We describe the main ingredients of the operadic category $\Gr$ and
leave the verification of the axioms as an exercise. Since objects
and morphisms of $\Gr$ are defined using  finite
sets and their morphism as building blocks, such a verification reduces to the
properties of the operadic category $\Fin.$

The cardinality functor assigns to a
graph the (linearly ordered) set of its vertices. Because of this we will often identify a vertex $i$ of a graph $\Gamma$ with its image in the ordinal $|\Gamma|.$

A~morphism  $\Phi : \Gamma \to \Gamma'$ of ordered graphs, i.e.\ a diagram
\begin{equation}
\label{Ruda_je_rasista.}
\xymatrix@R=.3em@C=2.5em{
F  \ar[dddd]_{g} && \ar@{_{(}->}[ll]_\psi \ar[dddd]^{g'} F'
\\
&\bar{n} \ar@{^{(}->}[lu] \ar[dd]  \ar@{_{(}->}[ru] &
\\
{}
\\
&\bar{1}  &
\\
V \ar@{->>}[ru] \ar[rr]^\phi &&V'\ar@{->>}[lu]
}
\end{equation}
induces for each $i \in V'$  a commutative diagram
\begin{equation}
\label{Jana_taky.}
\xymatrix@C=3.5em{(\phi g)^{-1}(i) \ar[d]_{g}  
& 
(g')^{-1}(i)\ar@{_{(}->}[l]_(.4){\psi}  \ar[d]^{g'}
\\
\phi^{-1}(i) \ar[r]^{\phi} & \bar{1}
}
\end{equation}
in $\Fin$ in which
the morphisms $g,\phi,g'$ and $\psi$ are the restrictions of the 
corresponding morphisms from~(\ref{Ruda_je_rasista.}). 
We interpret the right vertical morphism as a corolla by imposing the
trivial involution on $(g')^{-1}(i)$. Due to the definition of
fibers of maps of preordered graphs, the diagram above represents a map 
of the fiber  of $\Phi$ over $i$ to a corolla, which makes it an
ordered graph.
We take it as the definition of the {\em fiber in $\Gr$\/}. 
In other words, 
the fiber gets a linear order on its legs from the ordinal $(g')^{-1}(i)$.
Finally, the chosen local terminal objects in $\Gr$ are  $c_n =
\id:1_n\to 1_n$, that is corollas whose global order of legs 
coincides with the local order at this unique vertex.
\end{proof}

It follows from the commutativity of the upper triangle
in~(\ref{Ruda_je_rasista.}) that the map $\psi$ preserves the global
orders of legs,
therefore morphisms of ordered graphs induce
order-preserving bijections of the legs of graphs.
Thus the category $\Gr_{\tt ord}$ then consists of morphisms
(\ref{Ruda_je_rasista.}) 
in which, moreover, $\phi$ is order preserving, that is, the order of
vertices is preserved.

A \qb\ $\Phi : \Gamma \to \Gamma'$ in $\Gr$ 
is a morphism~(\ref{Ruda_je_rasista.}) each of whose 
fibers is the chosen local terminal object $c_n$ for some $n \geq 0$. 
It is clear that in this case both $\phi$ and $\psi$ must be bijections
and, moreover, the local orders on $(g')^{-1}(i) $ and $(\phi g)^{-1}(i)$
coincide for each $i \in V'$. 
In other words, \qb{s} are local isomorphisms over
$1_n$. So  $\Gamma'$ as an ordered graph is obtained from
$\Gamma$ by reordering its vertices. We thus have:

\begin{lemma} 
\label{Pozitri kontrola u Vondry.}
All quasibijections in $\Gr$ are invertible.
\end{lemma}


\begin{lemma} 
\label{Clemens zitra odleta.}
The operadic category $\Gr$ is factorizable.
\end{lemma}

\begin{proof} 
Given a morphism $\Gamma\to \Gamma'$ over $1_n$ we use 
Proposition~\ref{fact} to factorize it  as a local isomorphism $\Gamma
\to \Gamma''$ followed by a composite of a pure contraction and a local reordering. This last composite is an order-preserving morphism $\Gamma'' \to
\Gamma'$. We have a~commutative diagram 
\[
\xymatrix@C=3.5em{F'' \ar[d]_\cong  & F'\ar@{_{(}->}[l]_{} 
\\
F  & \bar{n}\ar@{_{(}->}[l]_{}  \ar@{_{(}->}[u] \ar@{_{(}-->}[lu]}
\]
of flags of the corresponding graphs. All maps in this
diagram induce isomorphisms of the sets of legs. 
Thus there is a unique monomorphism $\bar{n} \hookrightarrow F''$ which
makes the diagram commutative and which, moreover, induces an
isomorphism of the sets of legs.  
Therefore the factorization described 
above is the factorization over $1_n$ as required. 
\end{proof}

Lemma~\ref{Dnes_s_Clemensem_do_Matu} below involves a local reordering morphism
$\Upsilon = (\sigma,\id): \Gamma'' \to \Gamma'''$ of ordered graphs. 
Recall that such an
$\Upsilon$ induces the identity between the vertices of the graphs  $\Gamma''$
and $\Gamma'''$, i.e.\ their vertices are ``the same.''  The $\sigma$ part of this morphism amounts to a permutation of flags adjacent to   a vertex $i$ of $\Gamma'.$ 
This observation
is important for the formulation of:

\begin{lemma}
\label{Dnes_s_Clemensem_do_Matu}
Consider a commutative diagram 
\[
\xymatrix@C=1em@R=1em{
\Gamma \ar[rrrr]^f \ar[rd]_{\Phi = (\psi,\phi)}
&&&& \Gamma' \ar[llld]
\ar[ddll]
\\
&\Gamma''  \ar[rd]_{\Upsilon= (\sigma,\id)}&&&
\\
&&\Gamma'''&&
}
\]
of ordered graphs in which $ \Upsilon$
is a local reordering. Then the
fiber $\Phi^{-1}(i),  i \in |\Gamma''|,$ is obtained from the
fiber $(\Upsilon \Phi)^{-1}(\id(i))$ 
by changing the global order of its legs according to the permutation of flags  induced by  $\sigma$ at the  vertex $i.$

Analogously
 the map between the fibers 
induced by $f$ over $i\in
|\Gamma''|$  can be obtained from the map induced by $f$ over $\id(i)\in 
|\Gamma'''|$ by a permutation of orders of legs according to  the permutation $\sigma.$ \end{lemma}

\begin{proof}
Direct verification.
\end{proof}

\begin{lemma} 
\label{Prohlidkou jsem prosel.}
The operadic category $\Gr$ satisfies the weak blow-up axiom. 
\end{lemma}

\begin{proof} 
Let $\Phi : \Gamma \to \Gamma'$ be an order-preserving 
map with fibers $\Gamma_i$, $i
\in |\Gamma'|$. Assume we are given a morphism $\Xi_i : \Gamma_i \to
\Lambda_i$ for each $i$. 

Let us first ignore the global orders of graphs involved, i.e.\ work
in the category ${\ttprGr}$ of preordered graphs.
 Using Proposition~\ref{fact}, we first factorize $\Phi$ into a pure
contraction $c$ followed by a local reordering $\rho$ as in the bottom~of
\begin{equation}
\label{az_spony_mostu_sepnou_reku_noc_a_ty_jsi_jeden}
\xymatrix@C=3em{
A \ar[r]^\beta \ar[rd]^u  &\ar[r]^\gamma\ar[d]^v  B & 
C\ar[d]^b \ar[ld]_w
\\
\Gamma \ar[r]^c \ar[u]^\alpha  \ar@/_2em/[rr]^\Phi  & \Gamma'' \ar[r]^\rho 
&\ \Gamma'\,.
}
\end{equation}

Let $\widehat \Gamma_i$ be the graph obtained from  $\Gamma_i$ by
modifying its global order according to the action of the local
reordering $\rho$ as in Lemma~\ref{Dnes_s_Clemensem_do_Matu}. Notice
that $\widehat \Gamma_i$ is the fiber of $c$ over $i \in |\Gamma''|$.  
Let~$\widehat \Lambda_i$ be the graph $\Lambda_i$ with the global order
modified in the same
manner, and $\widehat \Xi_i : \widehat \Gamma_i \to
\widehat \Lambda_i$ the induced map.
We factorize $\widehat  \Xi_i$ as a quasibijection 
followed by a pure contraction and a local
reordering, as in
\[
\widehat \Xi_i : \widehat \Gamma_i \stackrel{\alpha_i}\longrightarrow
A_i \stackrel{\beta_i}\longrightarrow B_i \stackrel{\gamma_i}\longrightarrow 
\widehat \Lambda_i, \ i \in |\Gamma''|.
\]
We then realize these families of maps as the induced 
maps between fibers  step
by step using Lemmas~\ref{bupforcontractions}, \ref{bupforisos} and \ref{local}
giving rise to preordered graphs $A,B,C$ together with the morphisms $\alpha$, $\beta$ and $\gamma$
in~(\ref{az_spony_mostu_sepnou_reku_noc_a_ty_jsi_jeden}). It is clear
that diagram
\begin{equation}
\label{po_dlouhe_dobe_opet_pisu1}
\xymatrix@C=2em@R=1.2em{\Gamma\ar[rr]^a \ar[dr]_{\Phi} && C \ar[ld]^{b}
\\
&\Gamma'&
}
\end{equation} 
with $a : = \gamma \kompozice
\beta \kompozice \alpha$ and $b : = \rho \kompozice w$ commutes. By
Lemma~\ref{Dnes_s_Clemensem_do_Matu},  $a$
induces the requisite maps between the fibers in the category of
preordered graphs. Since the forgetful functor $\Gr \to {\ttprGr}$ is
faithful, the same is true also in the category of ordered graphs.

We must prove that the graph $C$
in~(\ref{po_dlouhe_dobe_opet_pisu1}) thus constructed carries
a compatible global order. Since 
morphisms in $\ttprGr$ map legs to legs bijectively,
the unique dashed arrow in
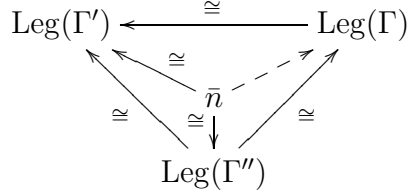
\[
\xymatrix@R=1em@C=1em{\Leg(\Gamma) && \Leg(C) \ar[ll]_\cong
\\
&\bar n  \ar[ul]_(.3)\cong \ar@{-->}[ur] \ar[d]_(.3)\cong&
\\
&\Leg(\Gamma') \ar[uul]^\cong\ar[uur]_\cong & 
}
\]
provides the requisite global order of $C$.

We need to prove that the factorization~(\ref{po_dlouhe_dobe_opet_pisu1})
is unique. Let $\Gamma = (F,V)$, $\Gamma' = (F',V'),$  
$\Lambda_i = (F_i,V_i)$ for $i \in V'$ and $C=(G,W).$ Since the map $b: C \to
\Gamma' $  is order preserving, the set $W$ of vertices of $C$ 
must be the ordinal sum $\bigcup_{i
  \in V'}V_i$ of the sets of vertices of the fibers. Likewise, the set of flags $G$ of $C$ equals
the ordinal sum $\bigcup_{i
  \in V'} F_i$. It is 
not difficult to show that also the involution on $G$ is determined by
the involutions on $F'$ and $F_i$, $i \in V'$. Thus the graph
$C$ is uniquely determined by the input~data, namely by $\Gamma'$
and the fibers $\Lambda_i$, $i\in V'$. 

Let us discuss the uniqueness of the maps in~(\ref{po_dlouhe_dobe_opet_pisu1}). 
As each vertex of $\Gamma$ belongs to a unique fiber of
$\Phi$, the horizontal arrow in the diagram
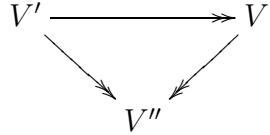
\[
\xymatrix{
V \ar@{->>}[rr]  \ar@{->>}[rd]  &&W  \ar@{->>}[ld]
\\
&V'&
}
\]
of the induced maps of vertices is uniquely determined by
the maps $\Vert(\Gamma_i) \to \Vert(\Lambda_i)$, $i\in V'$, induced
by the prescribed maps $\Xi_i$ of the fibers. Since both down-going maps
are order preserving by assumption, the right down-going map is
uniquely determined by the remaining two. 
By a similar argument, the horizontal inclusion in the diagram
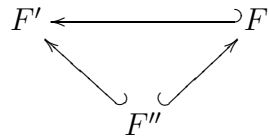
\[
\xymatrix{F&& \ar@{_{(}->}[ll]   G
\\
&F' \ar@{_{(}->}[lu] \ar@{^{(}->}[ur]  &
}
\]
of the induced maps of flags is uniquely determined by the maps 
$\Flag(\Lambda_i) \to \Flag(\Gamma_i)$, $i \in V'$, induced by the
prescribed maps of the fibers, 
so the right up-going inclusion is unique as well.
This finishes the proof.
\end{proof}

\begin{corollary}\label{GrBU}
The operadic category $\Gr$ satisfies  the blow-up axiom.
\end{corollary}

\begin{proof}
  The proof follows from Lemma~\ref{I_laboratorni_vysetreni_mne_ceka.}
  whose assumptions for the operadic category $\Gr$ were verified in
  Lemmas~\ref{Pozitri kontrola u Vondry.},~\ref{Clemens zitra odleta.}
  and~\ref{Prohlidkou jsem prosel.}.
\end{proof}

\begin{lemma}\label{GrUFC}
The category $\Gr$ satisfies the unique fiber condition.
\end{lemma}

\begin{proof}
Assume that the ordered graph $\Gamma$ is given by the left diagram below
\[
T=
\xymatrix{F \ar[d]_g  & \ar@{_{(}->}[l]_u  \ar[d]\ \bar n
\\
V \ar@{->>}[r] & \bar 1
},
\hskip 2em
t=\xymatrix{\bar n \ar[d]  & \ar@{_{(}->}[l]_\alpha  \ar[d]\ \bar n
\\
\bar 1 \ar@{->>}[r] & \bar 1
}
\]
and the local terminal object by the right one. A morphism
$\Phi:\Gamma\to t$ in $\Gr$ is characterized 
by a monomorphism $\psi : \bar n \to F$ in the diagram  
\begin{equation}
\label{NOE_v_padesatce}
\xymatrix@R=.3em@C=2.5em{
F  \ar[dddd]_{g} && \ar@{_{(}->}[ll]_\psi \ar[dddd] \bar n
\\
&\bar{n} \ar@{^{(}->}[lu]^u \ar[dd]  \ar@{_{(}->}[ru]_\alpha &
\\
{}
\\
&\bar{1}  &
\\
V  \ar@{->>}[ru] \ar[rr] &&\bar 1 \ar@{->>}[lu]
}
\end{equation}
and  by~(\ref{Jana_taky.}) its fiber $\inv{\Phi}(1)$ equals
\[
\xymatrix{F \ar[d]_g  & \ar@{_{(}->}[l]_\psi  \ar[d] \bar n
\\
V \ar@{->>}[r] & \ \bar 1\,.
}
\]
Thus $\inv{\Phi}(1) = \Gamma$ if and only if $\psi = u$. On the other
hand, the commutativity of the upper triangle
in~(\ref{NOE_v_padesatce}) implies that $u = \psi\alpha$. Since $\psi$
is a monomorphism, one sees that $\alpha = \id$, thus $t$ is the
chosen local terminal object.
\end{proof}

Since the assumptions of Lemma~\ref{Podivam_se_do_Paramaty??} are
satisfied by the operadic category of ordered graphs, one has:

\begin{corollary}\label{GrR}
The category $\Gr$ is rigid.
\end{corollary}

In the second paper \cite{sydney2} of the series we introduce the concept of quadraticity of operads over an operadic category~$\ttO$. For this purpose we add the following definitions.

\begin{definition}
\label{grad}
A {\em grading\/} on an operadic category $\ttO$ is a map $e :
{\rm Objects}(\ttO) \to \bbN$ of sets with the property that
\begin{equation}
\label{Parramatta}
e(T) + e(F_1) + \cdots + e(F_k) = e(S)
\end{equation}
for each $f : S \to T$ with fibers $\Rada F1k$. In this situation we
define the {\em grade\/} $e(f)$ of $f$ by $e(f) := e(S)- e(T)$.
\end{definition}

\begin{example}\label{canon} Each constant-free operadic
category $\ttO$ 
bears the {\em canonical grading\/} given by $e(T) := |T| -1$.
\end{example} 

\begin{definition}
\label{sgrad}
A graded operadic category $\ttO$ is {\em strictly graded\/}
if a morphism $f \in \ttO$ is an isomorphism if and only if $e(f) =
0$.
\end{definition}

We are grateful to our anonymous referee for the idea  and the proof
of the following

\begin{lemma}\label{steve?}
Let $\ttO$ be a graded operadic category. Then
\begin{itemize}
\item[(i)] for any local terminal object $t$ of $\ttO$ its grade $e(t)=0$. 
\item[(ii)]  For any isomorphism $f$ in $\ttO$ the grade $e(f) =0$.
\end{itemize}

Conversely, if $\ttO$ is factorizable, all
\qb{s} in $\ttO$ are isomorphisms, the weak blow-up axiom is fulfilled and the only objects of grade $0$ are local terminals,
then $e(f) =0$ implies that $f$ is an isomorphism, so
$\ttO$ is strictly graded.
\end{lemma} 

\begin{proof} Observe that 
the identity morphism  $\id: U  \to U$ of a trivial object has a
unique fiber~$U$, hence $2e(U) = e(U)$ by~(\ref{Parramatta}) and
$e(U) = 0$. 
Let $t$ be a local terminal object  
and $F$ be the fiber of the unique morphism $U\to t.$ Then $e(t)+e(F) = e(U) =0.$ So both $e(t)$ and $e(F)$ must be $0.$ This proves (i).
 
According to the first part of Lemma \ref{Podivam_se_do_Paramaty?}, every isomorphism $f$ has local terminals as its fibers, hence  $e(f)=0.$ This proves (ii).
The converse statement follows from the second part of Lemma
\ref{Podivam_se_do_Paramaty?}. 
\end{proof}

The grading can be transferred along operadic functors. 
\begin{lemma} 
 if $F : \ttO \to \ttP$ is an operadic functor and
$\ttP$ is graded, then $\ttO$ has a transferred  grading given by the formula
\begin{equation}
  \label{Dnes_plynari_ale_ja_jsem_v_Australii.}
e(T) := e(F(T)),\ T \in \ttO.  
\end{equation}
\end{lemma}

\begin{remark}
It is easy to see that a grading on $\ttO$ is the same as an
$\ttO$-operad in the discrete  symmetric  
monoidal category $(\bbN,+,0)$. The transfer of the grading amounts to the restriction functor between the category of operads. 
\end{remark}

\begin{lemma}
\label{GrGr} 
The operadic category $\Gr$ is  graded by the number of internal  edges of a graph. 
\end{lemma}
\begin{example}\label{grc is strict} The grading of $\Gr$ is not
  strict. Indeed, consider a unique morphism $!_{p+q}$ of the
  non-connected graph obtained as the disjoint union of corollas
  $c_p\cup c_q$ to $c_{p+q}$; recall that $c_n, n\ge 0$, are the trivial
  objects of $\Gr$. This is not an isomorphism, but $e(!_{p+q})=0.$
  On the other hand the operadic subcategory of {\em connected} graphs
  $\Grc$, cf.~Example \ref{musim_na_ten_odber}, is strictly graded for
  the transferred grading since the only connected graphs without
  internal edges are ordered corollas, which are local terminals (but
  not necessary trivial) in $\Grc.$
\end{example} 

\section{Discrete operadic (op)fibrations}
\label{Minulou_sobotu_jsem_odletal_vlekarskou_osnovu.}

In this section we focus on discrete operadic
fibrations $p: \ttO \to \ttP$. We show that 
the operadic category $\ttO$ retains
some useful properties of $\ttP$. Since, as we know from~\cite[page~1647]{duodel},
each set-valued $\ttP$-operad determines a discrete operadic fibration $p:
\ttO \to \ttP$, this gives a method to obtain new operadic
categories with controlled properties from the old ones. In the second
part of this section
we formulate similar statements for opfibrations and cooperads.

\subsection{Discrete operadic fibrations}
\label{snad_mne_kotnik_prestane_bolet}
We start by recalling Definition~2.1 of \cite{duodel}:

\begin{definition}
\label{psano_v_Myluzach}
  An operadic functor $p:\ttO\to \ttP$ is  a {\it discrete operadic
    fibration} if 
\begin{itemize}
\item[(i)]
$p$ induces a surjection $\pi_0(\ttO) \twoheadrightarrow
\pi_0(\ttP)$ and
\item[(ii)]
for any morphism $f : T\to S$ in $\ttP$ and any list of objects
  $t_1,\ldots, t_{k}, s \in \ttO$, where $k= |S|$, such that
  \[
  p(s) = S \mbox { and } p(t_i) = f^{-1}(i) \mbox { for  all} \ i
  \in |S|,
  \]
  there exists a unique $\sigma : t\to s$ in $\ttO$ such that
  \[
  p(\sigma) = f \mbox { and } t_i = \sigma^{-1}(i)\ \mbox { for all} \ i
  \in |S|.
  \]
\end{itemize}
\end{definition}

\begin{lemma} 
\label{noha_stale_boli}
Let $p:\ttO\to\ttP$ be a discrete operadic fibration and  $f:T\stackrel\sim\to
S$  a quasibijection in $\ttP$. Let $s\in \ttO$ be such that
$p(s)=S$. Then there exists a unique quasibijection $\sigma$ in~$\ttO$ with codomain $s$
such that $p(\sigma) = f$.
\end{lemma}

\begin{proof}
We invoke~\cite[Lemma~2.2]{duodel} saying that a
discrete operadic fibration induces an isomorphism of
$\pi_0$'s, plus the  fact that  operadic functors are
required to send
trivial objects to trivial ones. Therefore
$p$ establishes a bijection between the sets of trivial objects of the
categories $\ttO$ and $\ttP.$ Hence, we can uniquely complete the
data for $s$ by a list of trivial objects in place of the prescribed
fibers and construct $\sigma$
as the unique lift of these data.
\end{proof}

\begin{lemma} 
\label{kotnik}Let $p:\ttO\to\ttP$ be a discrete operadic fibration.  If in $\ttP$ all quasibijections are invertible, the same is
  true also for \qb{s} in $\ttO.$ In this case we also have that, for
  any quasibijection $f:T\stackrel\sim\to S$ in $\ttP$ and $t\in\ttO$
  such that $p(t)=T$, there exists a unique quasibijection
  $\sigma:t\stackrel\sim\to s$ such that $p(\sigma)=f.$
\end{lemma}

\begin{proof} 
Let $\sigma:t\stackrel\sim\to s$ be a quasibijection in $\ttO$.  
Consider the inverse $g:p(s)\to p(t)$ to the \qb\
$p(\sigma):  p(t) \stackrel\sim\to p(s)$. Notice that $g$ is a \qb\ by
Corollary~\ref{invertingqb}. 
Using Lemma~\ref{noha_stale_boli}, we  lift  $g$ to a  unique \qb\ $\eta:s'\stackrel\sim\to
  t$. The composite $\sigma\kompozice \eta$ is the lift 
  of the identity $p(s)\to p(s)$ so, by uniqueness, it is
  the identity as well, in particular, $s =s'$. 
The composite $\eta\kompozice\sigma$ is the identity
  for the same reason.

  The second part can be established as follows.  
Let $g:S\stackrel\sim\to T$ be the inverse quasibijection to $f$. We
   lift it to a \qb\ $\tilde g : s \stackrel\sim\to t$ in $\ttO$ and
   define $\sigma:t\stackrel\sim\to s$ to be the
  inverse of this lift. The uniqueness of the lifting 
guarantees that $\sigma$ is a lift of $f$.
\end{proof}

\begin{proposition}  
\label{Krtek_s_Laurinkou}
Let $p:\ttO\to\ttP$ be a discrete operadic fibration.
If $\ttP$ is a factorizable  operadic category in which all \qb{s} 
are invertible, then also $\ttO$ is factorizable.
\end{proposition}

\begin{proof} 
Let $\xi:t\to s$ be a morphism in $\ttO.$
  Let $T\stackrel{f}{\to}
  Z \stackrel{g}{\to} S$ be the factorization  of $p(\xi):T\to S$ into a \qb\ $f$ followed
  by an order-preserving $g \in \DP$. Let $h:Z\to T$ be the inverse to $f.$
  Using Lemma~\ref{kotnik} we
  lift $f$ to the unique and invertible
\qb~$\alpha:t\stackrel\sim\to z.$  Let $\beta:z\to t$ be its inverse. Then the morphism
$$z\stackrel{\beta}{\longrightarrow} t \stackrel{\xi}{\longrightarrow} s$$ 
is order preserving since
$$p(\xi\beta) =p(\xi)p(\beta) = p(\xi)h =g$$
is order preserving. Thus $(\xi\beta)\alpha$ is the desired factorization.
%
\end{proof}

\begin{proposition}  
\label{Vyhodil_jsem_sluchatka.}
Let $p:\ttO\to\ttP$ be a discrete operadic fibration. \begin{enumerate}\item[(i)] If the weak blow-up axiom holds in $\ttP$, it also holds in $\ttO;$  \item[(ii)]
If the blow-up axiom holds in $\ttP$, it also holds in $\ttO.$ \end{enumerate}
\end{proposition}

\begin{proof} 

For the weak blow-up axiom let  $h:T\to S$ be an order preserving morphism 
in $\ttO$ with the list of fibers $T_i,i\in|S|$, and let $\tau_i:T_i\stackrel{}{\to} F_i, i\in |S|$ be a family of morphisms. We need to find the unique factorization of $h$ in $\ttO$
\begin{equation*}
\xymatrix@C=2em@R=1.2em{T\ar[rr]^f \ar[dr]_{h} && R \ar[ld]^{g}
\\
&S&
}
\end{equation*}  
such that $g$ is order preserving and for each $i\in S$ the  induced 
morphism on fibers $f_i$ coincides with $\tau_i.$ 

We apply $p$ to $h$ and $\tau_i, i\in |S|.$ We thus obtain the input data 
for the weak blow-up axiom in $\ttP$ and we have the corresponding unique 
factorization of $p(h)$
\begin{equation*}
\xymatrix@C=2em@R=1.2em{p(T)\ar[rr]^\xi \ar[dr]_{p(h)} && \Gamma \ar[ld]^{\gamma}
\\
&p(S)&
}
\end{equation*}  
where $\xi$ acts on fibers as $\xi_i = p(\tau_i) :p(T_i) \to p(F_i).$

Invoking the lifting property of discrete operadic fibrations, we  lift $\gamma$ to $g:R\to S$ with $g^{-1}(i)= F_i.$ 
Since $\gamma$ is order preserving in $\ttP$ the morphism $g$ is also order preserving  in $\ttO.$

Observe that the fibers of $\xi $ are given by $\xi^{-1}(j) = \xi^{-1}_{|\gamma|(j)}(j), j\in |\Gamma|= |R|$ by Axiom (iv) of operadic categories, and hence, are equal to  the fibers $p(\tau_{|\gamma |(j)})^{-1}(j).$
We now use the lifting property of the operadic fibration for the second time to lift $\xi$ to a morphism $f:Q\to R$ in $\ttO$ whose fibers  are exactly $\tau_{|\gamma |(j)}^{-1}(j), j\in |R|.$

 Then the fibers of $f_i:(gf)^{-1}(i) \to g^{-1}(i) = F_i$ are  
 $\tau_{i}^{-1}(j), j\in |\gamma|^{-1}(i).$ But $\tau_i:T_i\to F_i$ has 
 exactly the same fibers, and both morphisms are liftings of $\xi_i.$ 
 Hence, by uniqueness of lifting, we have $f_i = \tau_i$ and $(gf)^{-1}(i) = T_i.$ But now we see that both $h$ and $gf$ are liftings of $p(h)$ and have the same fibers. By uniqueness of lifting again $Q=S$ and $h=gf$ and we obtained the required factorization. 

The proof of the second part of the proposition is similar. We only have to twist indices by the effect of the \qb\ $\sigma$ which is a part of the input data of the blow-up axiom. 
\end{proof}

Important examples of discrete operadic fibrations are provided by the {\it
  operadic Grothen\-dieck construction} introduced in~\cite[page~1647]{duodel}.
Assume that one is given a set-valued $\ttP$-operad 
$\calO$. One
then has the operadic category $\int_\ttP 
\calO$ whose objects are pairs $(T,t)$ where $T \in \ttP$ and $t \in \calO(T).$ A morphism $\sigma : (T,t) \to (S,s)$ for $t
\in \calO(T)$ and $s \in \calO(S)$ is a pair $(\varepsilon,f)$
consisting of a morphism $f : T \to S$ in $\ttP$ and of some
$\textstyle\varepsilon \in \prod_{i \in |S|}
\calO\big(f^{-1}(i)\big)$
such that
\[
\gamma_f(\varepsilon,s) = t,
\]
where $\gamma$ is the composition law of the operad $\calO$.  Composition of
morphisms is defined in the obvious manner.  The category 
$\int_\ttP \calO$ thus
constructed is an operadic category such that the functor $p :
\int_\ttP \calO \to \ttP$ given by
\begin{equation}
\label{bude_205?}
\mbox{$p(t) := T$ for $t \in \calO(T)$ and $p(\varepsilon,f)
  := f$}
\end{equation}
is a discrete operadic fibration. The trivial objects are given by the
operad units \hbox{$1_c \in \calO(U_c)$}.
By~\cite[Proposition~2.5]{duodel}, the above construction 
establishes an equivalence between the category
of set-valued $\ttP$-operads and the category of discrete
operadic fibrations over $\ttP$.

\begin{example} 
\label{musim_na_ten_odber}
Consider the $\Gr$-operad $C$ in $\Set$ such that 
\[
C(\Gamma) := 
\begin{cases}
1 \hbox { (one point set)} & \hbox {if  $\Gamma$ is connected}
\\
\emptyset   & \hbox {otherwise.}
\end{cases}
\]
There is a unique way to extend
  this construction to a $\Gr$-operad. The Grothendieck
  construction of $C$ produces a discrete operadic fibration
  $\Grc\to \Gr$.
  We call $\Grc$ the operadic category of {\em connected ordered graphs\/}.
\end{example}

\begin{example}  
\label{Zkusim_zavolat_Sehnalovi.}
A construction similar to the one in
  Example~\ref{musim_na_ten_odber} 
produces the operadic category $\Tr$ of trees. We consider the operad $\Pi$
with
\[
\Pi(\Gamma) := 
\begin{cases}
1 & \hbox {if  the geometric realization $\mathrm{B}(\Gamma)$ of $\Gamma$  is contractible}
\\
\emptyset   & \hbox {otherwise.}
\end{cases}
\]
The Grothendieck construction gives a discrete operadic fibration $\Tr\to \Gr$.
\end{example}

\begin{example}
\label{Zitra_s_Jarkou_k_Bartosovi.}
Let us orient  edges of a tree $T \in \Tr$ so
that they point to the leg which is the smallest in the global order. 
We say that T is
{\em rooted\/} if the outgoing half-edge of each vertex is the smallest
in the local order at that vertex.
Now define 
\[
R(T) := 
\begin{cases}
1 & \hbox {if  $T$ is rooted}
\\
\emptyset   & \hbox {otherwise.}
\end{cases}
\]
The Grothendieck construction associated to the operad $R$ gives the operadic
category $\RTr$ of {\em rooted trees\/}.
\end{example}

\begin{example}
\label{Predtim_s_Jarkou_na_krest_knihy.}
There is a unique isotopy class of
embeddings of $T \in\Tr$ into
the plane such that the local orders are compatible with the
orientation of the plane. This embedding in turn determines a cyclic order of the
legs of $T$. We say that $T$ is {\em planar\/} if this cyclic order
coincides with the cyclic order induced by the global order of the
legs. The operad
\[
P(T) := 
\begin{cases}
1 & \hbox {if  $T$ is planar}
\\
\emptyset   & \hbox {otherwise}
\end{cases}
\]
gives rise to the operadic category $\PTr$  of planar trees. In a
similar manner we obtain the operadic category $\PRTr$ of {\em planar
rooted trees\/}.
\end{example}

All the above constructions fall into the situation captured by the
following lemma whose proof is obvious.

\begin{lemma}
\label{Vcera jsme byli s Jarkou na nakupech.}
Let $i:\ttC\subset \ttP$ be a full
operadic subcategory such that 
\begin{enumerate}
\item[(i)]
the set of local chosen terminal objects of $\ttP$ coincides with 
the set of  local chosen terminal objects
  of $\ttC$, and 
\item[(ii)] 
for any morphism $f$ in $\ttP$ whose codomain and all
  fibers are in $\ttC$, the domain of $f$ is also in $\ttC$. 
\end{enumerate}
Then $i$ is a discrete operadic fibration.
\end{lemma}

\begin{remark}
\label{Koupil_jsem_si_silikonove_podpatenky.}
For a discrete operadic fibration $p:\ttO\to\ttP$, 
it is not true in general that the unique fiber condition is satisfied 
in $\ttO$ if it is satisfied in $\ttP.$  Thus it has to be verified
separately in each concrete case.
\end{remark}

\begin{example} Consider the one-object, one-morphism operadic
  category $\mathtt{1}$ whose set-valued operads are monoids. The
  operadic category $\mathtt{1}$ obviously satisfies the unique fiber
  condition.  Let $({\EuScript M},\cdot,e)$ be a monoid. The operadic
  Grothendieck construction $\int_{\mathtt 1}{\EuScript M}$ is fibered
  over~$\mathtt{1}$ and has pairs
  $(1,t) =: \mathbf{t},\ t\in {\EuScript M}$, as objects. A morphism
  from $\mathbf{x}$ to $\mathbf{y}$ is given by an element
  $a\in {\EuScript M}$ such that $ay =x.$ The fiber of such a morphism
  is $\mathbf{a}.$ The category $\int_{\mathtt 1}{\EuScript M}$ is
  connected with the trivial object $\mathbf{e} = (1,e).$ Notice that
  $t\in {\EuScript M}$ is invertible if and only if
  $\mathbf{t}$ is a local terminal object in
  $\int_{\mathtt 1}{\EuScript M}.$ Indeed, if $t$ is invertible, then  the
  equation $at = x$ has a unique solution $a = xt^{-1}$, hence there
  is a unique morphism from $\mathbf{x}$ to $\mathbf{t}$ in
  $\int_{\mathtt 1}{\EuScript M}$. The opposite implication is also
  clear. 

On the other hand, the equation $xt =x$ with invertible $t$ 
does not force, in
  general, the equality $t=e$ unless $x$ is invertible as well. For
  example, in the monoid ${\EuScript M}_2({\mathbb Z})$ of \hbox{$2\!\times\! 2$} integer-valued matrices under
the standard  matrix multiplication there are always $t\ne e$ and~$x$ which
  satisfy this equation, for instance
\[
 t := \left(\begin{array}{cc}0 & 1 \\ 1 & 0\end{array}\right),  \ x:=
 \left(\begin{array}{cc}1 & 1 \\0 & 0\end{array}\right).
\]
Since $t$ is invertible, we have that $\mathbf{t}$ is local terminal and the fiber
of $\mathbf{x} \to \mathbf{t}$ is $\mathbf{x}$, but $\mathbf{t}$ is
not the chosen local terminal. Thus $\int_{\mathtt 1}{\EuScript M}_2({\mathbb Z})$
does not fulfill the unique fiber condition.
\end{example}

We close this subsection with the following useful statement.

\begin{proposition}\label{onevsO}
Let $\calO$ be a set-valued $\ttP$-operad, $\int_\ttP \calO \to \ttP$ 
its operadic Grothendieck construction and $\Groterm$ the terminal
set-valued $\int_\ttP \calO$-operad. Then the categories of $\calO$-algebras and
$\Groterm$-algebras are isomorphic, i.e.\ 
\[
\Alg\calO \cong \Alg\Groterm.
\]
\end{proposition}

\begin{proof}
The sets of connected components of the categories $\ttP$ and $\int_\ttP \calO$
are canonically isomorphic via the correspondence
\[
U_c \longleftrightarrow 1_c \in \calO(U_c)
\]
of the chosen local terminal objects.
We use this isomorphism to identify $\pi_0(\ttP)$ with $\pi_0(\int_\ttP \calO)$.
Under this convention, the sets $\pi_0(s(T))$ of connected
components of the sources of an object
$T \in \ttP$ and the sets $\pi_0(s(t))$ 
of $t \in \calO(T)$ representing an object of
$\int_\ttP 
\calO$ are the same, and similarly $\pi_0(T) = \pi_0(t)$. The structure
operations of an $\calO$-algebra are by 
Definition~\ref{Jaruska_ma_chripecku}\begin{subequations}
\begin{equation}
\label{Dnes_vecer_s_Jarkou}
\alpha_T : \calO(T)\times\Cross{c \in \pi_0(s(T))} A_c  \longrightarrow
A_{\pi_0(T)}, \ T \in \ttP,
\end{equation}
which can be interpreted as 
families
\[
\alpha_t : \Cross{c \in \pi_0(s(T))}  A_c  \longrightarrow
A_{\pi_0(T)}, \ t \in \calO(T), \ T \in \ttP,
\]
of maps parametrized by $t \in \calO(T)$. Using the above
identifications, we rewrite the above display as 
\begin{equation}
\label{do_francouzske_restaurace.}
\alpha_t : \Cross{c \in \pi_0(s(t))} A_c  \longrightarrow
A_{\pi_0(t)}, \ t \in \calO(T), \ T \in \ttP,
\end{equation}
\end{subequations}
which are precisely the structure operations of an
$\Groterm$-algebra. It is simple to verify that the correspondence
between~(\ref{Dnes_vecer_s_Jarkou})
and~(\ref{do_francouzske_restaurace.})  
extends to an isomorphism of the categories of algebras.
\end{proof}

\subsection{Discrete operadic opfibrations}
In Subsection~\ref{snad_mne_kotnik_prestane_bolet} we recalled how
set-valued operads produce discrete operadic fibrations.
We are going to present a dual construction for cooperads. 

The notion of a cooperad over an operadic category is obtained from that of
an operad by reversing the arrows. 
A set-valued {\em $\ttP$-cooperad\/} is thus a collection  \hbox{$\oC = 
\{\oC(T)\}_{T\in \ttP}$} of sets together with structure maps
\begin{equation}
\label{za_chvili_na_ortopedii_s_patou}
\Delta_f: \oC(T) \longrightarrow    \oC(S) \times \oC(F_1)\times 
\cdots \times \oC(F_s)
\end{equation}
defined for an arbitrary $f :T \to S$ with fibers $\Rada F1s$.
The r\^ole of counits is played by the unique maps 
  \[
\oC(U_c) \to *,\ c \in \pi_0(\ttO),
  \]
to a terminal one-point set $*$. These operations are required to
satisfy axioms dual to those \hbox{in~\cite[Definition~1.11]{duodel}}.

A set-valued $\ttP$-cooperad $\oC$ leads to an operadic
category $\int^\ttP \hskip -.2em \oC$ via a dual version of the Grothen\-dieck construction
recalled in
Subsection~\ref{snad_mne_kotnik_prestane_bolet}.
The objects of $\int^\ttP \hskip -.2em  \oC$ are pairs $(T,t),$ where  $T \in \ttP$ and $t \in \oC(T).$  A morphism $\sigma : (T,t)\to (S,s)$  is a morphism $f : T \to S$ in $\ttP$ 
such that
\[
\Delta_f(t) = (s,\varepsilon)
\]
for some, necessarily unique,
$\textstyle\varepsilon \in \prod_{i \in |S|}
\oC\big(f^{-1}(i)\big)$,
where $\Delta_f$ is the structure map~(\ref{za_chvili_na_ortopedii_s_patou}).   

The category $\int^\ttP \hskip -.2em  \oC$ 
is an operadic category equipped with a~functor $p :
\int^\ttP \hskip -.2em  \oC \to \ttP$ defined by~(\ref{bude_205?}).
The trivial objects are all 
objects of the form $u \in \oC(U_c)$,  $c \in\pi_0(\ttP)$.
It turns out that the functor  $p :
\int^\ttP \hskip -.2em  \oC \to \ttP$ is a standard discrete opfibration: 

\begin{definition}
\label{zas_mne_boli_zapesti}
 A {\it discrete operadic
    opfibration} is an operadic  functor $p:\ttO\to \ttP$ which, as a functor, is a discrete opfibration.  That is, 
for any morphism $f : T\to S$ in $\ttP$ and any
  $t \in \ttO$ such that
  $  p(t) = T,$
  there exists a~unique $\sigma : t\to s$ in $\ttO$ such that
  $p(\sigma) = f$.
\end{definition}

Dualizing the steps in the proof of~\cite[Proposition~2.5]{duodel} one
can show that the dual Grothendieck 
construction is an equivalence between the category
of set-valued   $\ttP$-cooperads  and the category of discrete operadic 
opfibrations over $\ttP$.
As the following statement shows, discrete operadic opfibrations behave nicely with
respect to  chosen local terminal  objects.

\begin{lemma}
\label{Mam_tendinopatii.}
Operadic functors preserve local terminal objects. If $p : \ttO \to \ttP$ is a discrete operadic opfibration, then $t \in \ttO$ is
trivial  if and only if $p(t)$ is trivial. \end{lemma}

\begin{proof} Operadic functors send trivial objects to trivial objects. 
  Let $t$ be a local terminal in $\ttO$ and let $!:t\to U$ be the unique isomorphism to a trivial object. Then $p(!):p(u)\to P(U)$ is an isomorphism to a trivial object, and hence $p(t)$ is a local terminal.  

Suppose that $p$ is a discrete operadic opfibration. For $t\in \ttO$ let $!:t\to U$ be the unique map to a trivial object.  If $p(t)$ is trivial in $\ttP$, the map  $p(!):  p(t) \to p(U)$
is the identity, so
its lifts $!$ and $\id_t$ are two lifts of the identity $\id_{p(t)}$ with 
the common domain $t.$ Hence, $! = \id_t.$
\end{proof}

The next property of opfibrations has to be compared to
Remark~\ref{Koupil_jsem_si_silikonove_podpatenky.}. 

\begin{lemma}
Let $p : \ttO \to \ttP$ be a discrete operadic opfibration. 
If the unique fiber condition holds in $\ttP$, then it also holds in $\ttO$.
\end{lemma}

\begin{proof}
Suppose we have a situation $T \FIB T \to t$ in $\ttO$, with $t$
local terminal. By the first part of Lemma~\ref{Mam_tendinopatii.}, we have
$p(T) \FIB p(T) \to p(t)$ in $\ttP$ with $p(t)$ local
terminal. By the unique fiber condition for $\ttP$, $p(t)$ is a chosen
local terminal object in $\ttP$, so $t$ is a chosen local terminal object in
$\ttO$ by the second part of
Lemma~\ref{Mam_tendinopatii.}. 
\end{proof}

It turns out that analogs of 
Lemmas~\ref{noha_stale_boli},~\ref{kotnik} and
Propositions~\ref{Krtek_s_Laurinkou},~\ref{Vyhodil_jsem_sluchatka.}
hold also for discrete operadic opfibrations. As an example, we prove the
following variant of Lemma~\ref{noha_stale_boli}.

\begin{lemma} 
Let $p:\ttO\to\ttP$ be a discrete operadic opfibration and  $f:T\stackrel\sim\to
S$  a \qb\ in $\ttP$. Let $t\in \ttO$ be such that
$p(t)=T$. Then there exists a unique \qb\ $\sigma$ in~$\ttO$ 
with domain $t$ such that $p(\sigma) = f$.
\end{lemma}

\begin{proof}
By the lifting property of opfibrations, $f$  lifts to a unique
$\sigma$ so we only need to prove that $\sigma$ is a \qb. Since $p$ is
an operadic functor, it maps the fibers of $\sigma$ to the fibers of
$f$. Since the latter are trivial in $\ttP$, the former must be
trivial in $\ttO$ by Lemma~\ref{Mam_tendinopatii.}. So $\sigma$ is a \qb.
\end{proof}

An analog of Lemma~\ref{kotnik} for a discrete operadic opfibration
$p:\ttO\to\ttP$ reads as follows:

\begin{lemma} 
Let $p:\ttO\to\ttP$ be a discrete operadic opfibration. If all 
quasibijections  in $\ttP$ are invertible, then 
 all \qb{s} in $\ttO$ are also invertible. Moreover, for
each \qb\ $f:T\stackrel\sim\to S$ in $\ttP$ and $s\in\ttO$
  such that $p(s)=S$, there exists a unique quasibijection
  $\sigma:t\stackrel\sim\to s$ such that $p(\sigma)=f$.
\end{lemma}

We leave the proof of this lemma as an exercise, as well as the
verification that
Propositions~\ref{Krtek_s_Laurinkou} and~\ref{Vyhodil_jsem_sluchatka.}
hold verbatim for discrete operadic opfibrations as well.

\begin{example}
\label{Zaletame_si_jeste_do_konce_Safari?}
In Example~\ref{musim_na_ten_odber} we constructed the operadic
category  $\Grc$ of connected ordered graphs.
We introduce a set-valued $\Grc$-cooperad $G$ as follows.
For $\Gamma = (V,F) \in \Grc$ we put
\[
G(\Gamma) := {\rm Map}(V,\bbN) = \{g(v) \in \bbN\ | \ v \in V\}.
\]
The cooperad structure operations
\[
\Delta_\Phi : G(\Gamma') \longrightarrow G(\Gamma'') \times
G(\Gamma_1) \times \cdots \times G(\Gamma_s)
\] 
are, for a map
$\Phi : \Gamma' = (V',F') \to \Gamma'' = (V'',F'')$ with
fibers $\Gamma_i = (V_i,F_i)$ over $i \in V''$, given as
$\Delta_\Phi(g') := (g'',g_1,\ldots,g_s)$, where $g_i$ is  the restriction
of $g'$ to $V_i \subset V'$~and
\[
g''(i) := 
\sum_{v \in V_i}g_i(v) + \dim\big(H^1(\mathrm{B}(\Gamma_i); {\mathbb Z})\big),
\ i \in V'',
\] 
where $\mathrm{B}(\Gamma_i)$ is the  geometric realization of $\Gamma_i$.

The Grothendieck construction applied to $G$ produces the operadic
category $\ggGrc$ of {\em genus-graded\/} connected ordered graphs.  The
  morphisms in this category coincide with the morphisms of graphs as
  introduced in~\cite[Section~2]{getzler-kapranov:CompM98}, 
modulo the orders which we used to
  make $\ggGrc$ an operadic category.
\end{example}

\begin{example}
\label{Pani_Bilkova_zatim_neudelala_rezrvaci.}
We say that an ordered graph $\Gamma \in \Gr$ is {\em oriented\/} if
\begin{itemize}
\item[(i)]
each internal edge in $\Gamma$ is oriented, 
meaning that one of the half-edges forming this edge is 
marked as the input one, and the other as the output, and
\item[(ii)]
also the legs of $\Gamma$ are marked as either input or output ones.
\end{itemize}
We will call the above data an {\em orientation\/}  and
denote the set of all orientations of $\Gamma$ by~$\Or(\Gamma)$.
It is easy to see that $\Or$ is a cooperad over $\Grc$. The 
operadic category $\Whe$ resulting from the Grothendieck construction 
applied to $\Or$ consists of {\em oriented} ordered connected
graphs. We choose the notation $\Whe$  because algebras of the terminal
$\Whe$-operads are the
{\em wheeled PROPs\/} introduced in~\cite{mms}.   
\end{example}

\begin{example}
\label{Travelodge}
Let $C$ be an obvious modification of the operad of
Example~\ref{Zkusim_zavolat_Sehnalovi.} to the category $\Whe$. The
Grothendieck construction associated to this modified $C$ 
produces the operadic category
$\Dio$ of {\em simply-connected\/} oriented ordered graphs. 
The notation expresses that algebras of the terminal $\Dio$-operad 
are {\it dioperads}~\cite{gan}. 
\end{example}

The {\em valency\/} of a vertex $u$ in a graph $\Gamma$ is the number of
half-edges adjacent to $u$. For any $v \geq 2$, all operadic
categories mentioned above that consist of simply
connected graphs, i.e.~$\Tr$, $\PTr$, $\RTr$, $\PRTr$ and $\Dio$,
possess full operadic subcategories $\Tr_v$, $\PTr_v$, $\RTr_v$,
$\PRTr_v$ and $\Dio_v$ of graphs all of whose vertices have
valency $\geq v$.

\begin{example}
\label{Zrejme_brzy_podlehnu.}
We call an ordered simply-connected graph $\Gamma \in \Dio$ a
{\em $\frac12$graph\/} if each 
internal edge $e$ of $\Gamma$ satisfies the
following condition: \hfill\break 
\hphantom{m} $\bullet$ either 
$e$ is the unique outgoing edge of its initial
vertex, or \hfill\break
\hphantom{m} $\bullet$ $e$ 
is the unique incoming edge of its terminal vertex. \hfill\break 
Edges allowed in a $\frac12$graph 
are portrayed in the picture:
\[
\scalebox{0.7} 
{
\begin{pspicture}(0,-2.42)(6.66,2.42)
\psdots[dotsize=0.2](1.2,0.88)
\psline[linewidth=0.04cm,arrowsize=0.05291667cm 2.0,arrowlength=1.4,arrowinset=0.4]{->}(1.2,-0.9)(1.2,0.76)
\psline[linewidth=0.04cm,arrowsize=0.05291667cm 2.0,arrowlength=1.4,arrowinset=0.4]{->}(1.2,1.0)(1.2,2.4)
\psline[linewidth=0.04cm,arrowsize=0.05291667cm 2.0,arrowlength=1.4,arrowinset=0.4]{->}(1.24,0.96)(1.82,2.36)
\psline[linewidth=0.04cm,arrowsize=0.05291667cm 2.0,arrowlength=1.4,arrowinset=0.4]{->}(1.24,0.92)(2.32,1.92)
\psline[linewidth=0.04cm,arrowsize=0.05291667cm 2.0,arrowlength=1.4,arrowinset=0.4]{->}(1.22,0.9)(0.08,1.9)
\psline[linewidth=0.04cm,arrowsize=0.05291667cm 2.0,arrowlength=1.4,arrowinset=0.4]{->}(1.2,0.92)(0.52,2.32)
\psline[linewidth=0.04cm,arrowsize=0.05291667cm 2.0,arrowlength=1.4,arrowinset=0.4]{->}(1.92,-0.46)(1.26,0.76)
\psline[linewidth=0.04cm,arrowsize=0.05291667cm 2.0,arrowlength=1.4,arrowinset=0.4]{->}(2.36,0.12)(1.28,0.82)
\psline[linewidth=0.04cm,arrowsize=0.05291667cm 2.0,arrowlength=1.4,arrowinset=0.4]{->}(0.52,-0.48)(1.14,0.78)
\psline[linewidth=0.04cm,arrowsize=0.05291667cm 2.0,arrowlength=1.4,arrowinset=0.4]{->}(0.0,0.12)(1.1,0.84)
\psdots[dotsize=0.2](1.2,-0.88)
\psline[linewidth=0.04cm,arrowsize=0.05291667cm 2.0,arrowlength=1.4,arrowinset=0.4]{->}(1.18,-2.4)(1.2,-1.0)
\psline[linewidth=0.04cm,arrowsize=0.05291667cm 2.0,arrowlength=1.4,arrowinset=0.4]{->}(1.92,-2.22)(1.26,-1.0)
\psline[linewidth=0.04cm,arrowsize=0.05291667cm 2.0,arrowlength=1.4,arrowinset=0.4]{->}(2.38,-1.64)(1.3,-0.94)
\psline[linewidth=0.04cm,arrowsize=0.05291667cm 2.0,arrowlength=1.4,arrowinset=0.4]{->}(0.52,-2.24)(1.14,-0.98)
\psline[linewidth=0.04cm,arrowsize=0.05291667cm 2.0,arrowlength=1.4,arrowinset=0.4]{->}(0.0,-1.64)(1.1,-0.92)
\usefont{T1}{ptm}{m}{n}
\rput(1.3562304,-0.155){e}
\psdots[dotsize=0.2](5.46,0.88)
\psline[linewidth=0.04cm,arrowsize=0.05291667cm 2.0,arrowlength=1.4,arrowinset=0.4]{->}(5.46,1.0)(5.46,2.4)
\psline[linewidth=0.04cm,arrowsize=0.05291667cm 2.0,arrowlength=1.4,arrowinset=0.4]{->}(5.5,0.96)(6.08,2.36)
\psline[linewidth=0.04cm,arrowsize=0.05291667cm 2.0,arrowlength=1.4,arrowinset=0.4]{->}(5.5,0.92)(6.58,1.92)
\psline[linewidth=0.04cm,arrowsize=0.05291667cm 2.0,arrowlength=1.4,arrowinset=0.4]{->}(5.48,0.9)(4.34,1.9)
\psline[linewidth=0.04cm,arrowsize=0.05291667cm 2.0,arrowlength=1.4,arrowinset=0.4]{->}(5.48,0.86)(4.78,2.32)
\psline[linewidth=0.04cm,arrowsize=0.05291667cm 2.0,arrowlength=1.4,arrowinset=0.4]{->}(5.36,-0.82)(4.3,-0.2)
\psline[linewidth=0.04cm,arrowsize=0.05291667cm 2.0,arrowlength=1.4,arrowinset=0.4]{->}(5.56,-0.82)(6.52,-0.22)
\psline[linewidth=0.04cm,arrowsize=0.05291667cm 2.0,arrowlength=1.4,arrowinset=0.4]{->}(5.46,-0.9)(5.46,0.76)
\psdots[dotsize=0.2](5.46,-0.88)
\psline[linewidth=0.04cm,arrowsize=0.05291667cm 2.0,arrowlength=1.4,arrowinset=0.4]{->}(5.46,-2.4)(5.46,-1.0)
\psline[linewidth=0.04cm,arrowsize=0.05291667cm 2.0,arrowlength=1.4,arrowinset=0.4]{->}(6.18,-2.22)(5.52,-1.0)
\psline[linewidth=0.04cm,arrowsize=0.05291667cm 2.0,arrowlength=1.4,arrowinset=0.4]{->}(6.64,-1.64)(5.56,-0.94)
\psline[linewidth=0.04cm,arrowsize=0.05291667cm 2.0,arrowlength=1.4,arrowinset=0.4]{->}(4.78,-2.24)(5.4,-0.98)
\psline[linewidth=0.04cm,arrowsize=0.05291667cm 2.0,arrowlength=1.4,arrowinset=0.4]{->}(4.26,-1.64)(5.36,-0.92)
\psline[linewidth=0.04cm,arrowsize=0.05291667cm 2.0,arrowlength=1.4,arrowinset=0.4]{->}(5.48,-0.8)(6.16,0.12)
\psline[linewidth=0.04cm,arrowsize=0.05291667cm 2.0,arrowlength=1.4,arrowinset=0.4]{->}(5.42,-0.78)(4.7,0.08)
\usefont{T1}{ptm}{m}{n}
\rput(5.6762304,0.245){e}
\end{pspicture} 
}
\]
borrowed from~\cite{markl:handbook}. 
For $\Gamma \in \Dio$, let us define
\[
\textstyle
\mathtt{\frac12}(\Gamma) := 
\begin{cases}
1 & \hbox {if  $\Gamma$ is a $\frac12$graph}
\\
\emptyset   & \hbox {otherwise.}
\end{cases}
\]
It is easy to verify that the restriction of
$\mathtt{\frac12}$ to
$\Dio_3 \subset \Dio$ is an operad.
The Grothendieck construction applied to   
$\mathtt{\frac12}$ produces the operadic
category $\hGr_3$ of $\frac12$graphs whose vertices have valency 
$\geq 3$. Algebras for the terminal $\hGr_3$-operad are $\frac12$PROPs
as \hbox{in~\cite[Definition~4]{markl:ba}}.
\end{example}

The constructions above are summarized in the diagram
\begin{equation}
\label{Dnes_jsme_byli_na_vylete_u_more.}
\xymatrix@R=1em{&\RTr\ar[dr]^\FB &&& \ar[d]^\OF\ggGrc&&
\\
\PRTr   \ar[ur]^\FB\ar[dr]^\FB &&\ar@{^{(}->}[rr]^\FB \Tr\, 
&& \ar@{^{(}->}[rr]^\FB\Grc\, && \Gr
\\
&\PTr\ar[ur]^\FB &&&\ar[u]_\OF\Whe&\Dio\ar@{_{(}->}[l]_\FB&\, \Dio_3
\ar@{_{(}->}[l]_(.55)\OF&\hGr_3\ar@{_{(}->}[l]_(.52){\BF} 
}
\end{equation}
in which \FB\ denotes fibrations, \OF\ opfibrations and \BF\ the
inclusion that is both a fibration and an opfibration.

\begin{example}
The inclusion $\hGr_3 \to \Dio_3$ is a discrete operadic
fibration. The same is however not true for the inclusion  $\hGr \to
\Dio$ of the categories of graphs with vertices of arbitrary
valencies. While condition~(i) of  
Lemma~\ref{Vcera jsme byli s Jarkou na nakupech.} is satisfied,
condition~(ii) is violated e.g.~by the map $f:S \to T$ depicted below
\newgray{kgray}{.5}
\[
\psscalebox{1.0 1.0} 
{
\begin{pspicture}(0,-0.77539796)(6.6032753,0.77539796)
\psline[linecolor=black, linewidth=0.03, arrowsize=0.05291667cm 2.0,arrowlength=3,arrowinset=0.0]{->}(1.9016376,-0.37539795)(2.5016375,0.6246021)
\psline[linecolor=black, linewidth=0.03, arrowsize=0.05291667cm 2.0,arrowlength=3,arrowinset=0.0]{->}(3.1016376,-0.37539795)(2.5016375,0.6246021)
\psdots[linecolor=black, dotsize=0.20625](3.7016375,0.6021)
\psdots[linecolor=black, dotsize=0.20625](3.1016376,-0.37539795)
\psdots[linecolor=black, dotsize=0.20625](1.9016376,-0.37539795)
\psdots[linecolor=black, dotsize=0.20625](2.5016375,0.6246021)
\rput[bl](2.9016376,0.22460204){\scriptsize $e$}
\psline[linecolor=black, linewidth=0.03, arrowsize=0.05291667cm 2.0,arrowlength=3,arrowinset=0.0]{->}(3.1016376,-0.37539795)(3.7016375,0.6246021)
\rput[bl](4.1016374,0){\large $\stackrel f \longrightarrow$}
\psline[linecolor=black, linewidth=0.03, arrowsize=0.05291667cm 2.0,arrowlength=3,arrowinset=0.0]{->}(5.9016376,-0.37539795)(5.3016376,0.6246021)
\psdots[linecolor=kgray, dotsize=0.20600587](5.3016376,0.6246021)
\psline[linecolor=black, linewidth=0.03, arrowsize=0.05291667cm 2.0,arrowlength=3,arrowinset=0.0]{->}(5.9016376,-0.37539795)(6.5016375,0.6246021)
\psdots[linecolor=black, dotsize=0.20625](5.9016376,-0.37539795)
\psdots[linecolor=black, dotsize=0.20625](6.5016375,0.6246021)
\psline[linecolor=black, linewidth=0.03, arrowsize=0.05291667cm 2.0,arrowlength=3,arrowinset=0.0]{->}(0.10163757,-0.37539795)(0.70163757,0.6246021)
\psdots[linecolor=black, dotsize=0.20625](0.70163757,0.6246021)
\psdots[linecolor=black, dotsize=0.20625](0.10163757,-0.37539795)
\rput[bl](1.376,0.02460205){$\fib$}
\rput{-32.594162}(0.3773987,1.2359951){\psellipse[linecolor=black, linewidth=0.03, linestyle=dashed, dash=0.17638889cm 0.10583334cm, dimen=outer](2.1016376,0.02460205)(0.4,1.0)}
\psline[linecolor=black, linewidth=0.03, arrowsize=0.05291667cm 2.0,arrowlength=3,arrowinset=0.0]{->}(1.16375,-0.07539795)(0.70163757,0.6246021)
\rput[t](0.70163757,-0.77539796){$F$}
\rput[t](2.7016375,-0.77539796){$S$}
\rput[t](5.9016376,-0.77539796){$T$}
\end{pspicture}
}
\]
given by contracting  the subgraph in the dashed oval to the
gray vertex of the rightmost graph. Both the
target $T$ and the fiber $F$ are $\frac12$-graphs, but the domain $S$ is
not, since it contains the edge $e$ that is neither the unique outgoing,
nor the unique incoming one.
The inclusion $\hGr_3 \to \Dio$ is not a discrete operadic fibration
either; this time it is item~(i)~of  
Lemma~\ref{Vcera jsme byli s Jarkou na nakupech.} that is
violated.
\end{example}

\begin{example}
Let $\Ord$ be the  subcategory of $\Surj$ consisting of
order-preserving surjections. It is an operadic category
whose operads are the classical constant-free
non-symmetric
operads~\cite[Example~1.15]{duodel}. 
One has the
$\Ord$-cooperad $\Su$ with components
\[
\Su(\bar n) := \coprod_{m \geq n}\Sur(\bar m,\bar
n), \ n \geq 1,
\]
where $\Sur(\bar m,\bar n)$ denotes the set of all (not necessarily
order-preserving) surjections.
Its structure map $\Delta_f
: \Su(\bar n) \to \Su(\bar s) \times \Su(\inv f (1)) \times  \cdots
\times  \Su(\inv f (s))$ is, for $f : \bar n \to \bar s$, 
given by
\[
\Delta_f(\alpha) := \beta \times \alpha_1 \times \cdots \times
\alpha_s,\ 
\alpha \in \Su(\bar n),
\] 
where $\beta:= f \kompozice \alpha$ and $\alpha_i : (f\alpha)^{-1}(i) \to
f^{-1}(i)$ is the restriction of $\alpha$, $i \in \bar s$. The
Grothendieck construction of the cooperad $\Su$ leads to the operadic category
$\Per$ 
whose operads are
permutads,
introduced in~\cite{markl:perm}.
\end{example}

\begin{table}[t]
\def\arraystretch{1.3}
\begin{center}
  \begin{tabular}{|l|l|l|} 
\hline  \multicolumn{1}{|c|}{{property of $\ttP$}}
& \multicolumn{1}{c|}{{$p$ is fibration then $\ttO$ satisfies}} &
\multicolumn{1}{c|}{{$p$ is opfibration then $\ttO$ satisfies}}
\\
\hline \hline 
\Fac   & ? & \Fac
\\
\hline
\Fac\ \& \QBI  & \Fac\ \& \QBI  &  \Fac\ \& \QBI 
\\
\hline
\SBU & \SBU & \SBU
\\
\hline
\UFB & ? & \UFB
\\
\hline
\Rig & ? & \Rig
\\
\hline
$\ttP$ is  graded & $\ttO$ is graded & $\ttO$ is graded
\\
\hline
 \SGrad & ? & 
 \SGrad
\\
\hline 
\end{tabular}
\\
\rule{0em}{.8em}
\end{center}
\caption{How discrete operadic (op)fibrations $p:\ttO\to \ttP$ interact with properties of operadic categories.}
\label{analogies}
\end{table}

The leftmost column of  Table~\ref{analogies} lists properties
required in the second paper of the series~\cite{sydney2}. 
Its top four rows record results obtained in this section. The 5th and 6th rows
easily follow from the uniqueness of lifts in discrete operadic opfibrations,
while the grading in the last row is given by
formula~(\ref{Dnes_plynari_ale_ja_jsem_v_Australii.}) and does not 
require any additional assumptions on $p : \ttO \to \ttP$.

\begin{remark}
\label{Pujdu_zitra_na_seminar?}
If $p : \ttO \to \ttP$ is a discrete operadic opfibration and
$\ttP$ fulfills the properties listed in the leftmost column of
Table~\ref{analogies}, then $\ttO$ shares the same properties.
If  $p : \ttO \to \ttP$ is a~discrete operadic fibration, the situation in not
so simple. One may however invoke the
implication $\hbox {\rm {\UFB} \& {\WBU}} 
\Longrightarrow  \hbox {{\Rig}}$ of
Lemma~\ref{Podivam_se_do_Paramaty??} and conclude that if one
``manually'' verifies
$\UFB$ and the presence of a strict grading, then $\ttO$ satisfies all the
properties in the leftmost column also in the case of opfibrations.  
\end{remark}
\begin{example} According to Section \ref{Ceka_mne_Psenicka.} the operadic 
  category $\Gr$ satisfies  all properties in the leftmost column of Table~\ref{analogies} except for being strictly graded. Then the category $\Grc$ being fibered over $\Gr$  satisfies $\Fac,\QBI,\SBU.$ The fact that it satisfies $\UFB$ follows promptly from the fact that local terminals and chosen local terminals in $\Gr$ and $\Grc$ coincide and $\Gr$ satisfies $\UFB.$ Hence, $\Grc$ also satisfies $\Rig.$ The grading on $\Grc$ is strict, see Example \ref{grc is strict}. 
It now follows from  Example \ref{Zaletame_si_jeste_do_konce_Safari?} that the operadic category $\ggGrc$ is opfibered over $\Grc,$ and hence, satisfies all the properties listed in the rightmost column of the Table~\ref{analogies}. 

\end{example}

\section{Elementary morphisms} 
\label{Velikonoce s Jarkou na chalupe.}

While composition laws 
$\gamma_f: \oP(f) \otimes \oP(S)\to \oP(T)$  of
an operad~$\oP$ over an operadic category~$\ttO$ are associated to an
arbitrary morphism $f : T \to S$ in
$\ttO$, cf.~\cite[Definition~1.11]{duodel}, composition 
laws of Markl operads introduced in Section~\ref{section-markl} are associated to morphisms with only one nontrivial fiber. The precise
definition and properties of this class of morphisms are the subject
of this section. From now on all operadic categories will be {\em graded.}

\begin{definition}
\label{plysacci_postacci}
A morphism $\phi 
: T \to S \in \DO$ in an operadic category $\ttO$ 
is {\em elementary\/} if all its fibers
are trivial (= chosen local terminal) 
except precisely~one whose grade is $\geq 1$.
If~$\inv{\phi}(i)$ is, for  $i \in |S|$, the unique nontrivial fiber, we will
sometimes write $\phi$ as the pair $(\phi,i)$. If we want to name
the unique nontrivial fiber $F := \inv\phi( i)$ 
explicitly, we will write $F \fib_i T
\stackrel\phi\to S$, or  $F \fib T
\stackrel\phi\to S$ when the concrete $i \in |S|$ is not important.
\end{definition}

\noindent 
{\bf Notation.}
In the setup of Lemma~\ref{l3} with $\sigma$ a \qb, assume that the morphisms 
$f',f''$ are elementary, 
$\inv{f'}(a)$ is the only nontrivial fiber of $f'$, and $\inv{f''}(b)$ 
with $b:= |\sigma|(a)$ the only nontrivial fiber of $f''$. In this
situation, we denote by
\begin{equation}
\label{n1}
\overline \pi : = \pi_{(a,b)} :\inv{f'}(a) \to  \inv{f''}(b)
\end{equation}
the only nontrivial part of the derived sequence~(\ref{e1}).

\begin{remark}
\label{V_Srni_pred_Borovou_Ladou}
If $\pi$ is a \qb, the only nontrivial fiber of $f''$
{\em must be\/}  $\inv{f''}(b)$ with  $b:= |\sigma|(a)$. Indeed, the
maps in~(\ref{zitra_na_prohlidku_k_doktoru_Reichovi}) are \qb{s}, so
their fibers are, by definition, the chosen local terminal objects. When
$\inv{f''}(j)$ is the chosen local terminal object, then the (unique) fiber
of $\pi_{(i,j)}$ is  $\inv{f'}(i)$, so it must be, by Axiom~(iii) of an
operadic category, a~chosen local terminal
object too. 
\end{remark}

\begin{corollary}
\label{sec:sundry-facts-about-1}
Assume the  blow-up axiom and suppose that
in the corner for blow-up as on the left of the display
\begin{equation}
\label{cF}
\xymatrix@C=3.5em{S' \ar[d]_{f'}  &
\\
T' \ar[r]^\sigma_\sim &T''
}
 \  \  \  \  \  \  \   \  \  \  \  \  \  \ \xymatrix@C=3.5em{S' \ar[d]_{f'} \ar[r]^\pi  & S'' \ar[d]^{f''}
\\
T' \ar[r]^\sigma_\sim &T''
}
\end{equation}
 the map $f'$ is elementary, with the
unique nontrivial fiber  over $a \in |T'|$. Let $b := |\sigma|(a)$ and assume we
are given a map $\overline \pi : \inv{f'}(a) \to
F,$ where $e(F)\ge 1$. Then the corner in~(\ref{cF}) can be uniquely completed
into a commutative square as on the right
 in which $f''$ is elementary with the unique
nontrivial fiber $\inv{f''}(b) = F$ and the nontrivial part of the derived
sequence is $\overline \pi$.
\end{corollary}

\begin{proof}
By \SBU,~(\ref{eq:5}) is uniquely
determined by the maps between the fibers. The only map between nontrivial
fibers is $\overline \pi$, while 
all maps between trivial ones are unique by the terminality of trivial
objects, thus there is no room for
choices of the induced maps between fibers.
\end{proof}

\begin{definition}
\label{d3}
Let $T\stackrel{(\phi,j)}{\longrightarrow} S
\stackrel{(\psi,i)}{\longrightarrow} P$ be 
elementary morphisms. If  $|\psi|(j) = i$
we say that 
the fibers of $\phi$ and $\psi$ are {\em joint\/}.
If
 $|\psi|(j)\ne i$ we say that 
$\phi$ and $\psi$ have {\em disjoint fibers\/} or,  more
specifically, that the fibers of $\phi$ and $\psi$ are {\em
  $(i,j)$-disjoint\/}, cf.~the following picture.
\end{definition}

\newgray{kgray}{.9}
\[
\psscalebox{1.0 1.0} 
{
\begin{pspicture}(0,-2.8)(6.24,2.8)
\psline[linecolor=black, linewidth=0.01](0.31,-2.3375)(5.71,-2.3375)
\psline[linecolor=black, linewidth=0.01](0.31,0.0625)(5.71,0.0625)
\psline[linecolor=black, linewidth=0.01](5.71,2.4625)(5.71,2.4625)(0.31,2.4625)
\psline[linecolor=black, linewidth=0.03](0.71,2.4625)(1.71,0.0625)(2.71,2.4625)(2.71,2.4625)
\psline[linecolor=black, linewidth=0.03](1.71,0.0625)(1.71,-2.3375)
\psline[linecolor=black, linewidth=0.03](3.31,0.0625)(3.31,0.0625)(4.31,-2.3375)(5.31,0.0625)(5.31,2.4625)(5.31,2.4625)
\psline[linecolor=black, linewidth=0.03](3.31,0.0625)(3.31,2.4625)
\psellipse[fillcolor=kgray, linecolor=black, linewidth=0.03, fillstyle=solid, dimen=outer](1.71,2.4625)(1.0,0.4)
\rput(1.71,2.5){\scriptsize $\xi^{-1}(k)$}
\psellipse[fillcolor=kgray, linecolor=black, linewidth=0.03, fillstyle=solid, dimen=outer](4.31,2.4625)(1.0,0.4)
\psellipse[fillcolor=kgray, linecolor=black, linewidth=0.03, fillstyle=solid, dimen=outer](4.31,0.0625)(1.0,0.4)
\rput(4.3,2.5){\scriptsize $\xi^{-1}(i)$}
\rput(4.31,0.0625){\scriptsize $\psi^{-1}(i)$}
\psdots[linecolor=black, dotsize=0.14](1.71,0.0625)
\psdots[linecolor=black, dotsize=0.14](1.71,-2.3375)
\psdots[linecolor=black, dotsize=0.14](4.31,-2.3375)
\psline[linecolor=black, linewidth=0.03, arrowsize=0.05291666666666667cm 5.0,arrowlength=1.4,arrowinset=0.0]{->}(0.51,2.4625)(0.51,0.0625)
\psline[linecolor=black, linewidth=0.03, arrowsize=0.05291666666666667cm 5.0,arrowlength=1.4,arrowinset=0.0]{->}(0.51,0.0625)(0.51,-2.3375)
\psline[linecolor=black, linewidth=0.03, arrowsize=0.05291666666666667cm 5.0,arrowlength=1.4,arrowinset=0.0]{->}(4.11,2.0625)(4.11,0.4625)
\rput(6.11,2.4625){$T$}
\rput(6.11,0.0625){$S$}
\rput(6.11,-2.3375){$P$}
\rput(0.11,1.4625){$\phi$}
\rput(0.11,-0.9375){$\psi$}
\rput(4.51,1.2625){$\phi_i$}
\rput(2.11,0.3){$j$}
\rput(1.71,-2.7375){$k$}
\rput(4.31,-2.7375){$i$}
\rput(3.71,1.2625){$\sim$}
\end{pspicture}
}
\]

\begin{lemma}
\label{l7} 
If the
fibers of elementary morphisms $\phi$ and $\psi$ in
Definition~\ref{d3} are
joint, then the
composite $\xi = \psi(\phi)$ is elementary  as well,  with the
nontrivial fiber over $i$, and the induced morphism
$\phi_i:\xi^{-1}(i)\to \psi^{-1}(i) $ is elementary with the
nontrivial fiber over $j$ that equals $\phi^{-1}(j)$.  For $l\ne i$ the
morphism $\phi_l$ equals the identity $U_c\to U_c$ of trivial objects.

If the fibers of $\phi$ and $\psi$ are { $(i,j)$-disjoint} then the morphism
$\xi = \psi(\phi)$ has exactly two nontrivial fibers and these are
fibers over $i$ and $k: =|\psi|(j)$. Moreover, there is a canonical
induced quasibijection
\begin{subequations}
\begin{equation}
\label{harmonika}
\phi_i:\xi^{-1}(i)\to \psi^{-1}(i) \in \DO
\end{equation}
and we have the equality
\begin{equation}
\label{pisu_opet_v_Sydney}
\xi^{-1}(k) = \phi^{-1}(j).
\end{equation}
\end{subequations}
\end{lemma}

\begin{proof} 
By Axiom~(iv) of an operadic category, we have
$\phi_i^{-1}(j) = \phi^{-1}(j)$, thus $e(\phi_i^{-1}(k))\geq 1$. 
If  $k \in |\psi|^{-1}(i)$ is such
that  $k\ne j$, then
$\phi_i^{-1}(k) = \phi^{-1}(k) = U_c$.
Therefore $\phi_i$ is an elementary morphism.

Let us prove that $\xi$ is elementary as well. For $i =k\in |P|$, we
have $\phi_i: \xi^{-1}(i)\to \psi^{-1}(i)$, hence the grade of
$\xi^{-1}(i)$ must be greater than or equal to the grade of
$\phi_i^{-1}(j) = \phi^{-1}(j)$, which is greater than or equal to $1$.  For
$k\ne i$, $\phi_k:\xi^{-1}(k)\to \psi^{-1}(k) = U'$ has the unique
fiber equal to $\xi^{-1}(k).$ On the other hand for the unique $l$ such that
$|\psi|(l) = k$,
\[
\phi_k^{-1}(l) = \phi^{-1}(l) = U'',
\] 
hence $\xi^{-1}(k) = U''$, so $\xi$ is elementary.

Let us prove the second part of the lemma.
If $l\ne i,k$ then $\phi_l:\xi^{-1}(l)\to \psi^{-1}(l) =
 U'$, where $U'$ is a trivial object. So the unique fiber of $\phi_l$
equals $\xi^{-1}(l)$. Since $|\psi|$ is surjective,
there exists $l'\in |S|$
 such that $|\psi|(l') = l$, and such an $l'$ is unique because 
$\psi$ is elementary. Hence $\phi_l^{-1}(l') = \phi^{-1}(l')  = U'$
 and so $\xi^{-1}(l) = U'$. This proves that the only nontrivial
 fibers of $\xi$ can be those over $i$ and $k$. Their grades are
 clearly $\geq 1$.

Let us prove that $\phi_i$ is a quasibijection.  If $l\in
|\psi^{-1}(i)|$ then $\phi_i^{-1}(l) = \phi^{-1}(l).$
Since the fibers of $\phi$ and $\psi$ are $(i,j)$-disjoint by
assumption, we have $|\psi|(l) = i \not= |\psi|(j)$, hence $l \not=
j$. Since the only nontrivial fiber of $\phi$ is $\phi^{-1}(j)$, we
conclude that $\phi^{-1}(l)$ and therefore also $\phi_i^{-1}(l)$ is
trivial.
To prove that $\phi_i \in \DO$, notice that by~Axiom~(iii), $|\phi_i|$
is the map of sets $|\xi|^{-1}  \to |\psi|^{-1}$  induced by the
diagram
\[
\xymatrix@C=2em@R=1.2em{|T| \ar[rr]^{|\phi|} 
\ar[dr]_{|\xi|} && |S| \ar[ld]^{|\psi|}
\\
&|P|\,.&
}
\]

Regarding~(\ref{pisu_opet_v_Sydney}), 
by Axiom~(iv) we have  $\phi^{-1}(j) = \phi^{-1}_k(j)$. But
$\phi_k:\xi^{-1}(k)\to \psi^{-1}(k) = U''$ and hence its unique
fiber is equal to $\xi^{-1}(k)$. So, $\phi^{-1}(j) = \xi^{-1}(k)$.
\end{proof}

\begin{definition}
\label{har}
We will call the pair 
$T\stackrel{(\phi,j)}{\longrightarrow} S
\stackrel{(\psi,i)}{\longrightarrow} P$ of morphisms
in Definition~\ref{d3}  with disjoint fibers 
{\em harmonic\/}  if $\inv{\xi}(i) = \inv \psi
(i)$ and the map $\phi_i$ in~(\ref{harmonika}) is the identity.
\end{definition}

\begin{corollary}
\label{zitra_vylet_do_Sydney}
If the blow-up axiom is satisfied then all pairs with disjoint fibers
are harmonic.
\end{corollary}

\begin{proof}
The map $\phi_i$ in~(\ref{harmonika}) is a \qb\ in $\DO$, so it is the
identity by Corollary~\ref{zase_mam_nejaky_tik}.
\end{proof}

\begin{corollary} 
\label{move}
Assume that
\begin{equation}
\label{eq:3}
\xymatrix@R = 1em@C=4em{& {P'} \ar[dr]^{(\psi',\, i)} &
\\
T\ar[dr]^{(\phi'',\, l)} \ar[ur]^{(\phi',\, j)} && S
\\ 
&{P''}\ar[ur]^{(\psi'',\, k)}&
}
\end{equation}
is a commutative diagram of elementary morphisms. 
Assume that $|\psi''|(l) = i, |\psi'|(j) = k$ and $i\ne k.$
Let $F', F'',G',G''$ be the
only nontrivial fibers of $\phi',\phi'',\psi',\psi''$, respectively.
Then one has canonical quasibijections
\begin{equation}
\label{za_tyden_poletim_do_Prahy}
\sigma': F' \longrightarrow G'' \ 
\mbox { and }\ \sigma'': F'' \longrightarrow G'.  
\end{equation}
If both pairs in~(\ref{eq:3}) are harmonic, then $F' = G''$, $F'' =
G'$ and $\sigma',\sigma''$ are the identities.
\end{corollary}

\begin{proof} 
Let $\xi : T \to S$ be the composite $\psi'\phi' = \psi''\phi''$.
One has  $G' =
\inv{\psi'}(j)$, $G'' = \inv{\psi''}(k)$ and, by Lemma~\ref{l7},  
$F' = \inv{\phi'}(j) = \inv{\xi}(k)$ and $F'' = \inv{\phi''}(l) = \inv{\xi}(i)$.
We define
\[
\sigma' :   F' = \inv{\xi}(k) \stackrel{\phi''_k}\longrightarrow
\inv{\psi''}(k) = G''
\ \and \
\sigma'' :   F'' = \inv{\xi}(i) \stackrel{\phi'_i}\longrightarrow
\inv{\psi''}(i) = G'.
\]
These maps are  \qb{s} by Lemma~\ref{l7}. The second part of the
corollary follows directly from the definition of harmonicity.
\end{proof}

\section{Markl operads}
\label{section-markl}

The aim of this section is to introduce Markl operads and their
algebras in the context
of operadic categories, and formulate assumptions under which these
notions agree with the standard ones introduced in~\cite{duodel}.
\begin{assu}
\label{A}
We assume that $\ttO$ is a strictly graded
factorizable operadic category in which all
\qb{s} are invertible, the  blow-up axiom and unique fiber condition are
fulfilled, and a
morphism $f$ is an isomorphism if and only if $e(f) =
0$; recall that by Lemma~\ref{Podivam_se_do_Paramaty?}
this happens 
if and only if all fibers of $f$ are local terminal. In brief, we~require
\[
\hbox{\rm \Fac\ \& \SBU\ \& \QBI\ \& \UFB\ \& \SGrad.}
\]
\end{assu}
Denoting by $\Iso$ the subcategory of $\ttO$ consisting of \label{Co
  zitra zjisti?}
all isomorphisms we therefore have by~\SGrad
\[
\Iso = \{ f : S \to T;\ e(f) = 0\}
=  \{ f : S \to T;\ e(F) = 0   \hbox { for each fiber $F$ of $f$}   \}.
\]
Another  consequence of the strict grading  assumption is that $T \in \ttO$ is
local terminal if and only if \hbox {$e(T) = 0$}.
 
\begin{definition}
\label{markl} 
A {\em Markl $\ttO$-operad\/} in a symmetric monoidal category $\ttV$
is a presheaf $\Markl: \Iso^{\rm op} \to \ttV$ with values in $\ttV$ equipped, for each
elementary morphism $F\fib T \stackrel\phi\to S$ as in
Definition~\ref{plysacci_postacci}, 
with a ``circle product''
\begin{equation}
\label{ten_prelet_jsem_podelal}
\circ_{\phi}: \Markl(S)\otimes \Markl(F)\to \Markl(T).
\end{equation}
These operations must satisfy the following set of axioms.

\begin{itemize}
\item[(i)]
 Let
$T\stackrel{(\phi,j)}{\longrightarrow} S
  \stackrel{(\psi,i)}{\longrightarrow} P$
  be elementary morphisms such that $|\psi|(j) = i$ and let
  $\xi: T \to P$ be the composite $\psi \kompozice \phi$.  Then
the  diagram
\begin{equation}
\label{vymena}
\xymatrix@R = 1em{& {         \Markl(P)\otimes \Markl(\xi^{-1}(i))} 
\ar[dr]^(.65){\circ_\xi} &
\\
 \Markl(P)\otimes \Markl(\psi^{-1}(i))\otimes \Markl(\phi^{-1}(j)) 
\ar[dr]_{\circ_\psi\ot \id} \ar[ur]^{\id \ot \circ_{\phi_i}} &&  \Markl(T)
\\ 
&{ \Markl(S)\otimes \Markl(\phi^{-1}(j))}\ar[ur]_(.65){\circ_\phi}&
}
\end{equation}
commutes.

\item[(ii)]  
Let us consider the diagram 
\begin{equation}
\label{den_pred_Silvestrem_jsem_nachlazeny}
\xymatrix@R = 1em@C=4em{& {P'} \ar[dr]^{(\psi',\, i)} &
\\
T\ar[dr]^{(\phi'',\, l)} \ar[ur]^{(\phi',\, j)} && S
\\ 
&{P''}\ar[ur]^{(\psi'',\, k)}&
}
\end{equation}
of elementary morphisms with disjoint fibers as in Corollary~\ref{move}. 
Then the diagram
\begin{equation}
\label{Napadne jeste tento rok snih?}
\xymatrix@R = 2.5em@C=4em{
\Markl(S) \ot \Markl(G') \ot \Markl(F')\ar[r]^(.55){\circ_{\psi'} \ot \id} &
 \Markl(P') \ot \Markl(F')\ar[d]^(.5){\circ_{\phi'}}
\\
\Markl(S) \ot \Markl(F'') \ot \Markl(G'') 
\ar[u]^{\id \ot (\sigma''^{-1})^* \ot   \sigma'^*}
& \Markl(T)
\\
\ar[u]^{\id \ot \tau}
\Markl(S) \ot \Markl(G'') \ot \Markl(F'')\ar[r]^(.55){\circ_{\psi''}  \ot \id} & 
\Markl(P'') \ot \Markl(F'') \ar[u]_(.5){\circ_{\phi''}}
}
\end{equation}
commutes whenever
$
F' \fib  T \stackrel{\phi'}\to P',\
F'' \fib  T \stackrel{\phi''}\to P'',\
G' \fib  P' \stackrel{\psi'}\to S \hbox { and }
G'' \fib  P'' \stackrel{\psi''}\to S,
$
and the maps $({\sigma''^{-1}})^*$ and $\sigma'^*$ are induced by 
\qb{s}~(\ref{za_tyden_poletim_do_Prahy}).

\item[(iii)]
For every commutative diagram
\[
\xymatrix@C=4em@R=1.4em{T'\ar[r]_\cong^\omega \ar[d]^{\phi'}
& T'' \ar[d]^{\phi''}
\\
S'\ar[r]^\sigma_\sim & S''
}
\] 
where $\omega$ is an isomorphism, $\sigma$ a \qb, 
and
$F' \fib_i T' \stackrel{\phi'}\to S'$, $F'' \fib_j 
T'' \stackrel{\phi''}\to S''$,
the diagram
\begin{equation}
\label{posledni_nedele_v_Sydney}
\xymatrix@C=4em{\ar[d]_{\bar{\omega}^* \ot \sigma^*}^\cong
\Markl(F'') \ot \Markl(S'') 
\ar[r]^(.65){\circ_{\phi''}}& \Markl(T'') \ar[d]^{\omega^*}_\cong
\\
\Markl(F') \ot
\Markl(S')\ar[r]_(.65){\circ_{\phi'}}&  \Markl(T')\,,
}
\end{equation}
in which $\bar{\omega} : F' \to F''$ is the
induced map~(\ref{zitra_na_prohlidku_k_doktoru_Reichovi}) of fibers, commutes.
\end{itemize}
A Markl operad $\Markl$ is {\em  unital\/} if one is given, for each
trivial $U$, a map $\eta_U
:\bfk \to \Markl(U)$  such that the diagram
\begin{equation}
\label{Holter_se_blizi.}
\xymatrix{\Markl(U) \ot \Markl(T) \ar[r]^(.6){\circ_!} & \Markl(T)
\\
\ar[u]^{\eta_U \ot \id}
\bfk \ot \Markl(T)  \ar@{=}[r]^(.58)\cong & \Markl(T) \ar@{=}[u]
}
\end{equation}
where $T \fib T \stackrel!\to U$ is the unique map, commutes whenever $T$ is such that $e(T) \geq 1.$ 

\end{definition}

\begin{remark} The definition above is more general than we need in the 
  rest of the paper but we believe it will be useful in a future. Since we assume the strong blow-up axiom, all pairs of morphisms with disjoint fibers are
harmonic by Corollary~\ref{zitra_vylet_do_Sydney}. Thus in Axiom~(ii)
the morphisms $\sigma'$ and $\sigma''$ are the
identities. Denoting $F:=F' = G''$ and $G := G' = F''$,
diagram~\eqref{Napadne jeste tento rok snih?} takes the form
\[
\xymatrix@R = 1.5em@C=4em{
\Markl(S) \ot \Markl(G) \ot \Markl(F)\ar[r]^(.55){\circ_{\psi'} \ot \id} &
 \Markl(P') \ot \Markl(F)\ar[d]^(.5){\circ_{\phi'}}
\\
& \Markl(T)
\\
\ar[uu]^{\id \ot \tau}
\Markl(S) \ot \Markl(F) \ot \Markl(G)\ar[r]^(.55){\circ_{\psi''}  \ot \id} & 
\ \Markl(P'') \ot \Markl(G) \ar[u]_(.5){\circ_{\phi''}}\, .
}
\]
\end{remark}

\label{Jsem znepokojen.}
Let $\LT$ be the operadic subcategory of $\ttO$ consisting of all local
terminal objects of $\ttO.$ Denote by $\term_\Term : \LT \to \ttV$ the constant
functor, i.e.~the functor such that $\term_\Term (u) = \bfk$
for each local terminal $u\in \ttO$. Since 
$\LT$ is equivalent, as a category, to the
discrete groupoid of trivial objects,
for a unital 
Markl operad $\Markl$
the collection $\{\eta_U :\bfk \to \Markl(U)\}$ of unit maps
extends  uniquely into a transformation
\begin{equation}
\label{2x2}
\eta :   \term_\Term \to \iota^* \Markl
\end{equation}
from the constant functor $ \term_\Term$ to the
restriction of $\Markl$ along the inclusion
$\iota : \LT \hookrightarrow \ttO$. The component $\eta_u : \bfk \to
\Markl(u)$ of that extension is, for $u$ local terminal, given by
$\eta_u :=\ !^* \eta_U$, where $!: u \to U$ is the unique map to a
trivial $U$. Transformation~(\ref{2x2}) of course amounts to a family
of maps $\eta_u : \bfk \to \Markl(u)$
given for each local terminal $u \in
\LT$, such that the diagram
\begin{equation}
\label{Je_vedro.}
\xymatrix@R=1.2em{\Markl(u)  \ar[r]^(.5){!^*} & \Markl(v)
\\
\bfk  \ar@{=}[r] \ar[u]^{\eta_u }   &
  \bfk
\ar[u]_{\eta_v}
}
\end{equation}
commutes for each (unique) map $! : v \to u$ of local terminal
objects. We will call the components   $\eta_u : \bfk \to
\Markl(u)$ of the transformation~\eqref{2x2} the {\em extended units\/}.

For each $T$ with $e(T) \geq 1$ and $F \fib T \stackrel!\to u$ with
$u$ a local terminal object, 
one has a map $\vartheta(T,u) :  \Markl(F) \to  \Markl(T)$ defined by the diagram
\begin{equation}
\label{proc_ty_lidi_musej_porad_hlucet}
\xymatrix@R=1em{\Markl(u) \ot \Markl(F) \ar[r]^(.6){\circ_!} & \Markl(T)
\\
\Markl(U) \ot \Markl(F) \ar[u]^{!^* \ot \id}  &
\\
\ar[u]^{\eta_U \ot \id}
\bfk \ot \Markl(F) \ar@{=}[r]^\cong &\ \Markl(F)\,. \ar[uu]_{\vartheta(T,u)}
}
\end{equation}   
Notice that the composite of the maps in the left column equals
$\eta_u \ot \id$, where $\eta_u$ is a~component of the extension~(\ref{2x2}).

The unitality  offers
a generalization of Axiom~(iii) of Markl operads
which postulates for each commutative diagram
\begin{equation}
\label{moc_se_mi_na_prochazku_nechce}
\xymatrix@C=4em{T'\ar[r]^\omega_\cong \ar[d]^{\phi'}\ar[rd]^{\phi} 
& T'' \ar[d]^{\phi''}
\\
S'\ar[r]^\sigma_\cong & S''
}
\end{equation} 
where the horizontal maps are isomorphisms and 
the vertical
maps are elementary, with
$F' \fib_i T' \stackrel{\phi'}\to S'$, $F'' \fib_j 
T'' \stackrel{\phi''}\to S''$, the commutativity of 
the diagram
\begin{equation}
\label{X}
\xymatrix@C=4em
{\Markl(F) \ot \Markl(S'')&\Markl(F'') \ot \Markl(S'') \ar[l]_{\omega^*_j \ot \id}^\cong
\ar[r]^(.65){\circ_{\phi''}}& \Markl(T'') \ar[d]^{\omega^*}_\cong
\\
\ar[u]^{\vartheta(F,\inv{\sigma}(j)) \ot \id}
\Markl(F') \ot \Markl(S'')\ar[r]^{\id \ot \sigma^*}_\cong &\Markl(F') \ot
\Markl(S')\ar[r]_(.65){\circ_{\phi'}}&  \Markl(T')
}
\end{equation}
in which $F := \inv{\phi}(j)$ and $\omega_j : F \to F''$ is the
induced map of fibers. Notice that if $\sigma$ is a \qb,
(\ref{X}) implies~(\ref{posledni_nedele_v_Sydney}).

\begin{definition}
\label{svedeni}
A Markl operad $\Markl$ is {\em strictly unital\/} if all the maps
$\vartheta(T,u)$ in~(\ref{proc_ty_lidi_musej_porad_hlucet})
are identities. It is {\em $1$-connected\/} 
if~the unit maps $\eta_U : \bfk \to \Markl(U)$ are isomorphisms for each
trivial~$U$.
\end{definition}

If $\Markl$ is strictly unital, then $\Markl(F) = \Markl(F')$
in~(\ref{X}), so this diagram 
takes a
particularly simple form, namely
\begin{equation}
\label{Ve_ctvrtek_letim_do_Prahy.}
\xymatrix@C=4em{
\Markl(F'') \ot \Markl(S'') \ar[d]_{\omega^*_j \ot \sigma^*}^\cong
\ar[r]^(.65){\circ_{\phi''}}& \Markl(T'') \ar[d]^{\omega^*}_\cong
\\
\Markl(F') \ot
\Markl(S')\ar[r]^(.65){\circ_{\phi'}}&\,  \Markl(T')\,.
}
\end{equation}
The following lemma is an easy exercise on definitions.

\begin{lemma}
\label{Zavolam mu zitra?}
A $1$-connected Markl operad is strictly unital if and only if, for each  $F
\fib T \to u$ with $u$ a local terminal object, one has $\Markl(F) =
\Markl(T)$ and the diagram
\begin{equation}
\label{Zitra prijde obleva.}
\xymatrix@C=5em{
\Markl(u) \ot \Markl(F) \ar[rr]^(.6){\circ_!}& & \Markl(T)
\\
\Markl(U) \ot \Markl(F) \ar[u]_{!^* \ot \id}\ar[r]^(.55){(\eta_U \ot \id)^{-1}}&
\bfk \ot \Markl(F) \ar[r]^\cong
&\, \Markl(F)\ar@{=}[u]
}
\end{equation}
commutes. Here
$!$ denotes the corresponding unique map to local terminal
objects.
\end{lemma}

Let us introduce similar terminology for ``standard''
$\ttO$-operads. In this framework, $\term_{\LT}$ will denote the
constant $\LT$-operad. As for Markl operads, the collection 
 $\{\eta_U :\bfk \to \oP(U)\}$ of unit maps of an $\ttO$-operad $\oP$
extends  uniquely to a transformation
\begin{equation}
\label{7_dni_do_odletu_ze_Sydney}
\eta : \term_{\LT} \to \iota^* \oP
\end{equation}
of $\LT$-operads. One  has an obvious analog
of diagram~(\ref{proc_ty_lidi_musej_porad_hlucet}), and the strict
unitality and $1$-connectedness for $\ttO$-operads is
defined analogously. The main result of this section~reads:

\begin{theorem} 
\label{Mcat}
There is a natural forgetful functor from the category of 
strictly unital \hbox{$\ttO$-operads} to the category of strictly
unital Markl $\ttO$-operads, which restricts to an isomorphism of
the subcategories of $1$-connected operads.
\end{theorem}

\begin{example}
Constant-free May operads recalled in the introduction are
operads over the operadic category $\Surj$ of non-empty finite ordinals and their
surjections. Let us analyze  the meaning of the above definitions and results in
this particular case. With respect to the
canonical grading, cf.~Section~\ref{Pojedu_vecer_nebo_ted?}, 
elementary morphisms in the operadic category of finite
ordinals $\Surj$ are precisely order-preserving 
surjections
\[
\pi(m,i,n) : \overline{\rule{0em}{.7em}   m+n-1} \epi \overline
{\rule{0em}{.7em}m},
\  m \geq 1,\ n \geq 2,
\]
uniquely determined by the property that
\begin{equation}
\label{zitra_do_Rataj_slavit_Silvestra}
|\inv{\pi(m,i,n)}(j)|
= 
\begin{cases}
\hbox{$1$ if $j \not= i$}
\\
\hbox {$n$ if $j=i$.}
\end{cases}
\end{equation}
Since $\overline{\rule{0em}{.7em}1}$ is the only local terminal object of $\Surj$, 
the strict unitality is the same as the ordinary one and all
\im{s} are \qb{s}. A $(\Surj)_{\tt iso}$-presheaf turns out to be a collection  
$\{\Markl(n)\}_{n\geq 1}$  of $\Sigma_n$-modules, and elementary
maps~(\ref{zitra_do_Rataj_slavit_Silvestra}) induce operations
\[
\circ_i := \circ_{\pi(m,i,n)} : \Markl(m) \ot \Markl(n) \to
\Markl(m+n-1),
 \ n \geq 2,\ 1 \leq i \leq m,
\]
which satisfy the standard axioms listed 
e.g.~in~\cite[Definition~1.1]{markl:zebrulka}. 
Theorem~\ref{Mcat} in this case states the well-known fact that the category of unital May
operads with $\oP(1) = \bfk$ is isomorphic to the category of unital
Markl operads with $\Markl(1) = \bfk$.
\end{example}
 
\begin{proof}[Proof of Theorem~\ref{Mcat}.] Let $\oP$ be a strictly 
  unital $\ttO$-operad with composition laws~$\gamma_f$. 
If $\omega : T' \to T''$ is an isomorphism, we define
$\omega^* : \oP(T'') \to \oP(T')$ by the diagram
\begin{equation}
\label{dnes_jsem_byl_v_Eastwoodu}
\xymatrix{\oP(T'') \ot \oP(\omega) \ar[r]^(.6){\gamma_\omega}
& \oP(T')
\\
\ar[u]_{\id \ot \eta_\omega}
\oP(T'') \ot \bfk  \ar@{=}[r]^(.58)\cong& \oP(T'') \ar[u]_{\omega^*}
}
\end{equation}
in which $\oP(\omega)$ denotes the product $\oP(u_1) \ot \cdots \ot \oP(u_s)$ over
the fibers $\Rada u1s$ of $\omega$ and, likewise, $\eta_\omega :=
\eta_{u_1} \ot \cdots \ot \eta_{u_s}$. It is simple to show
that this construction is functorial, making 
$\oP$ an $\Iso$-presheaf in $\ttV$. In particular, $\omega^*$ 
is an isomorphism. For an
elementary $F \fib_i T
\stackrel\phi\to S$ we define $\circ_\phi : \oP(S)\ot \oP(F) \to \oP(T)$ by
the commutativity of the
diagram
\[
\xymatrix{\oP(S) \ot \oP(U_1) \!\ot\! \cdots  \!\ot\! \oP(U_{i-1}) \!\ot\! \oP(F) \!\ot\!   
\oP(U_{i+1}) \!\ot\! \cdots  \!\ot\! \oP(U_{|S|})\ar[r]^(.84){\gamma_\phi} & \oP(T)
\\
\oP(S) \ot 
\bfk^{\ot (i-1)} \ot \oP(F) \ot  \bfk^{\ot (|T|-i)} \ar[u] \ar@{=}[r]^(.6)\cong
&\oP(S) \ot \oP(F)  \ar[u]_{\circ_\phi}
}
\]
in which the left vertical map is induced by the unit morphisms
of $\oP$ and the
identity automorphism of $\oP(F)$. 
We claim that the $\Iso$-presheaf $\oP$
with operations $\circ_\phi$ defined above is a~Markl~operad. 

It is simple to check that these $\circ_\phi$'s satisfy the associativities
(i) and (ii) of a Markl operad. To prove Axiom~(iii), consider
diagram~(\ref{moc_se_mi_na_prochazku_nechce}) and invoke Axiom~(i) of
an operad over an operadic category, see Definition~\ref{Jarca_u_mne_prespala!}, 
once for $\phi =
\sigma\phi'$ and once for $\phi = \phi''\omega$
in place of $h =fg$. We will get two commutative squares sharing the
edge $\gamma_\phi$. 
Putting them side-by-side as in
\[
\xymatrix@C=4em@R=4em{
\bigotimes_k \oP(\phi'_k) \ot \oP(\sigma) \ot \oP(S'')
\ar[r]^(.6){\otimes_k \gamma_{\phi_k'} \ot \id}
\ar[d]_{\id \ot \gamma_\sigma}
& \oP(\phi) \ot
\oP(S'') \ar[d]^{\gamma_{\phi}} & \oP(\omega) \ot \oP(\phi'') \ot
\oP(S'')
\ar[l]_(.56){\bigotimes_i \gamma_{\omega_i} \ot \id}
\ar[d]^{\id \ot \gamma_{\phi''}}
\\
\oP(\phi') \ot \oP(S') \ar[r]^{\gamma_{\phi'}}   & \oP(T')&
\oP(\omega) \ot \oP(T'') \ar[l]_{\gamma_\omega}
}
\]
produces the central hexagon in the
diagram
\[
\xymatrix@R = 1.41em@C=-3em{
&\boxed{\oP(\phi'') \ot \oP(S'')} \ar@{=}[rrr] \ar[d]^\cong_{\bigotimes_k \omega_i^* \ot \id}
&&&&\ \ \bfk \ot \oP(\phi'') \ot \oP(S'') 
\ar[d]^{\eta_\omega \ot \id \ot \id}
&
\\
&\boxed{\oP(\phi) \ot \oP(S'')} \ar@{-->}[ddrrr]^{\gamma_\phi} 
&&\rule{9em}{0em}&&\ar[llll] 
\oP(\omega) \ot \oP(\phi'') \ot \oP(S'') \ar[d]_{\id \ot \gamma_{\phi''}}&
\\
&\bigotimes_k \oP(\phi'_k) \ot \oP(\sigma)\ar[u]^{\bigotimes_k
  \gamma_{\phi'_k} \ot \id}\ar[rd]^{\id \ot \gamma_\sigma} \ot \oP(S'')
&&&& \boxed{\oP(\omega) \ot \oP(T'')}\ar[ld]_{\gamma_\omega}&
\\
\oP(\phi')\ot \bfk \ot \oP(S'') \ar[ur]^(.3){\id \ot \eta_\sigma \ot \id}  
&& \boxed{\oP(\phi') \ot \oP(S')}
\ar [rr]^{\gamma_{\phi'}} && \boxed{\oP(T')} &&  \bfk \ot  \oP(T'')
\ar[lu]_{\eta_\omega \ot \id}
\\ 
\ar@{=}[u]\boxed{\oP(\phi')  \ot \oP(S'')} \ar[rru]_\cong^{\id \ot \sigma^*} &&&&&&
\,\oP(T'')\,.\ar@{=}[u]\ar[llu]_{\omega^*}^\cong
}
\]
The remaining arrows of this diagram
are constructed using the $\Iso$-presheaf structure of $\oP$ and the
extended units.
The boxed terms in the above diagram form  the internal hexagon~in
\[  
\xymatrix@R=1.3em@C=1em{
\oP(F) \ot \oP(S'')
 \ar[rd]&&
\oP(F'') \ot \oP(S'')
\ar[ll]^\cong_{\omega_j^* \ot \id} \ar[rr]^(.55){\circ_{\phi''}}\ar[d] && \ar[ld] \oP(T'') \ar[ddd]_\cong^{\omega^*}
\\
&\oP(\phi) \ot \oP(S'')&
\oP(\phi'') \ot \oP(S'')\ar[l]_\cong \ar[r]
 &  \oP(\omega)\ot \oP (T'')   \ar[d]&
\\
&\oP(\phi') \ot \oP(S'')  \ar[r]^\cong \ar[u]&\ar[r]
\oP(\phi') \ot \oP(S')
 &  \oP (T')&
\\
\oP(F') \ot \oP(S'') \ar[uuu]^{\vartheta(F,\inv{\sigma}(j)) \ot \id}
\ar[ur] \ar[rr]_\cong^{\id \ot \sigma^*} 
&&
\oP(F') \ot \oP(S') \ar[rr]^(.55){\circ_{\phi'}} \ar[u] && \, \oP(T')\,.\ar@{=}[lu]
}
\]
The commutativity of the outer hexagon
follows from the commutativity of the inner one. We recognize
in it diagram~(\ref{X}) with $\oP$ in place
of $\Markl$. Since~(\ref{X}) 
implies~(\ref{posledni_nedele_v_Sydney}) for $\sigma$ a \qb,
Axiom~(iii) is verified. Is is easy to check that the strict extended
unit~(\ref{7_dni_do_odletu_ze_Sydney}) is also the one for $\oP$
considered as a~Markl operad. 

Conversely, let $\Markl$ be a Markl operad. 
We are going to define, for each $f: S \to T$
with fibers $\Rada F1s$, the composition law
\begin{equation}
  \label{po_navratu_z_Prahy}
\gamma_f:
\Markl(T) \ot \Markl(F) \longrightarrow  \Markl(S)
\end{equation}
where, as several times before, $\Markl(F)$ denotes $\Markl(F_1) \ot \cdots \ot \Markl(F_s)$.
If $f$ is an \im, then all its fibers are local terminal,
so $\Markl(F) \cong \bfk$ by the strict  unitality and
the $1$-connectivity of $\Markl$. In this case we define
$\gamma_f$ as the composite
\begin{equation}
\label{vecer_Deminka?}
\Markl(T) \ot \Markl(F) \cong \Markl(T)
\stackrel{f^*}\longrightarrow 
\Markl(S)
\end{equation}
using the $\Iso$-presheaf structure of $\Markl$.

Assume now that $f \in \DO$ and that all local 
terminal fibers of $f$ are trivial. 
If $f$ is an \im\ it must be the identity 
by Corollary~\ref{zase_mam_nejaky_tik}. If it is not the case, at least one
fiber of $f$ has grade $\geq 1$ and we decompose $f$, using the
strong blow-up axiom, into a chain of elementary morphisms. The operation
$\gamma_f$ will then be defined as the composite of
$\circ$-operations corresponding to these elementary morphisms. Let us make this
procedure more~precise.

To understand the situation better, consider two elementary morphisms
$\phi$, $\psi$ with
{$(i,j)$-disjoint} fibers as in Lemma~\ref{l7} and their composite
$\xi = \psi(\phi)$. 
Notice that 
\[
\Markl(\xi) \cong  \Markl(\inv\xi(i)) \ot \Markl(\inv\xi(k))
\] 
by the
strict  unitality and the $1$-connectivity of $\Markl$. 
In this particular case we
define $\gamma_\xi$ by the commutativity of the diagram
\[
\xymatrix{\Markl(P) \ot \Markl(\inv\xi(i)) \ot \Markl(\inv\xi(k))
\ar[r]^(.65)\cong \ar[d]_{\id \ot (\phi^*_i)^{-1} \ot \id}
& \Markl(P) \ot \Markl(\xi)\ar[dd]^{\gamma_\xi}
\\
\ar[d]_{\id \ot \circ_\psi}
\Markl(P) \ot \Markl(\inv\psi(i)) \ot \Markl(\inv\phi(k))&
\\
\Markl(S)  \ot \Markl(\inv\phi(k))\ar[r]^(.6){\circ_\phi}& \Markl(T)\,,
}
\]
or, in shorthand, by $\gamma_\xi := \circ_\phi(\id \ot \circ_\psi)$.

Now take $f: S \to T \in \DO$ whose fibers of grade $\geq 1$ are
$\Rada F1k$ and the remaining fibers are trivial. Using the strong blow-up
axiom we factorize $f$ into a chain
\begin{equation}
\label{Z-142}
S = S_1 \stackrel{\phi_1}\longrightarrow  
S_2 \stackrel{\phi_2}\longrightarrow \cdots 
\stackrel{\phi_k}\longrightarrow S_k = T  
\end{equation} 
in which each $\phi_i$ is elementary with the unique nontrivial fiber
$F_i$, $1 \leq i \leq k$; we leave the details on how to obtain such a
factorization to the reader. We then define 
\[
\gamma_f := \circ_{\phi_1} ( \circ_{\phi_2} \ot \id) \cdots 
( \circ_{\phi_k} \ot \id^{\ot(k-1)}):
\Markl(T) \ot \Markl(F_k) \ot \cdots \ot \Markl(F_1) \longrightarrow \Markl(S).
\]

If $f: S \to T$ is a general morphism in $\ttO$, we
use~Lemma~\ref{Podival_jsem_se_do_Paramaty!} to factorize it as 
$f: S \stackrel\omega\to X \stackrel\psi\to T$ with an isomorphism $\omega$ and 
$\psi \in \DO$ all of whose  fibers are trivial. Notice that, due to the 
strict unitality and $1$-connectivity, one has  $\Markl(\psi) \cong \Markl(f)$. We then
define $\gamma_f$ by the commutativity of the diagram
\[
\xymatrix{\Markl(\psi) \ot \Markl(T) \ar[r]^\cong \ar[d]^{\gamma_\psi} & 
\Markl(f) \ot \Markl(T)\ar[d]_{\gamma_f}
\\
\Markl(X)\ar[r]^{\omega^*}&\ \Markl(S)\,.&
}
\]
The extended units are given by the extended units of $\Markl$ in the
obvious way.

Our definition of the $\gamma_f$-operations 
does not depend on the choices: the commutativity
of~(\ref{X}) that holds for  unital
Markl operads guarantees the independence on the factorization $f =
\psi \kompozice \omega$, while the commutativity
of~(\ref{den_pred_Silvestrem_jsem_nachlazeny}) implies the independence
on the choice of the decomposition~(\ref{Z-142}). We leave to the
reader the tedious
but straightforward verification that $\Markl$ with the above
composition laws is a strictly  unital $\ttO$-operad. 
\end{proof}

We are going to adapt the notion of algebra  over
operad, cf.~Definition \ref{Jaruska_ma_chripecku}, 
to the realm of
Markl operads. 

\begin{definition}
\label{Zapomel_jsem_si_pocitac_v_Koline_ja_hlupak}
An {\em algebra\/} over a $1$-connected Markl operad $\Markl$ in a symmetric
monoidal category $\ttV$ is 
a collection $A = \{A_c \ | \ c \in \pi_0(\ttO)\}$ of objects of $\ttV$ 
together with structure maps
\begin{equation}
\label{Ben_Ward_in_Prague}
\big\{
\alpha_T:
\Markl(T) \ot  \hskip -.5em 
\bigotimes_{c \in \pi_0(s(T))} A_c  \longrightarrow
A_{\pi_0(T)}\big\}_{T \in \ttO} \ .
\end{equation}
These operations are required to satisfy the following axioms.

\begin{itemize}
\item[(i)]
Unitality: for each component $c \in \pi_0(\ttO)$ the diagram
\[
\xymatrix{\Markl(U_c) \ot A_c \ar[r]^(.65){\alpha_{U_c}} & A_c
\\
\bfk \ot A_c\ar[r]^\cong \ar[u]^{\eta_{U_c}}  & A_c \ar@{=}[u]
}
\]
commutes.

\item[(ii)]
Equivariance: 
let $f : S\to T$ be an \im\ with fibers $\Rada u1s$. For $1 \leq
i \leq s$ put $c_i := \pi_0(s_i(S))$ and $d_i := 
\pi_0(s_i(T))$. 
Then the diagram
\[
\xymatrix@C=0em{
\Markl(T) \ot A_{c_1} \ot \cdots \ot A_{c_s}  \ar[rr]^(.45)\cong 
\ar[d]_{f^* \ot \id^{\ot s}} && 
\Markl(T) \ot \bfk \ot  A_{c_1} \ot \cdots \ot \bfk \ot A_{c_s} 
\ar[d]^{\id \ot \eta_{u_1} \ot \cdots \ot \eta_{u_s}}
\\
\Markl(S) \ot A_{c_1} \ot \cdots \ot A_{c_s}\ar[d]_{\alpha_S}  && 
\Markl(T) \ot \Markl(u_1) \ot  A_{c_1} \ot \cdots \ot \Markl(u_s)  \ot A_{c_s}
\ar[d]^{\id \ot \alpha_{u_1} \ot \cdots \ot \alpha_{u_s}}
\\
A_{\pi_0(S)} \ar@{=}[r] 
&A_{\pi_0(T)}& \Markl(T) \ot  A_{d_1} \ot \cdots \ot  A_{d_s}\ar[l]_(.6){\alpha_T}
}
\]
commutes.

\item[(iii)]
Associativity:
for an elementary map $F \fib_i S \stackrel\phi\to T$, the diagram
\[
\xymatrix{
\Markl(S) \ot A_{c_1} \ot \cdots  A_{c_{i-1}} \ot 
 A_{c_{i}} 
 \ot \cdots \ot   A_{c_{t+s-1}} \ar[r]^(.78){\alpha_S}
 &  A_{\pi_0(S)}\ar@{=}[ddd]
\\
\ar[u]^{\circ_\phi \ot \id^{\ot t+s-1}}
\ar[d]_{\id \ot \tau \ot \id^{\ot t-i}}
\Markl(T) \ot \Markl(F) \ot A_{c_1} \ot \cdots\ot  A_{c_{i-1}} \ot 
 A_{c_{i}} 
 \ot \cdots \ot   A_{c_{t+s-1}} & 
\\
\ar[d]_{\id \ot  \id^{\ot i} \ot \alpha_F \ot \id^{\ot t-i}}
\Markl(T)\ot A_{c_1} \ot \cdots \ot A_{c_{i-1}}  \ot \Markl(F) \ot 
 A_{c_{i}} 
 \ot \cdots \ot   A_{c_{t+s-1}}
&
\\
\Markl(T)\ot A_{c_1} \ot \cdots  A_{c_{i-1}}\ot
 A_{\pi_0(F)} 
 \ot \cdots \ot   A_{c_{t+s-1}}\ar[r]^(.78){\alpha_T}
& A_{\pi_0(T)}
}
\]
commutes, 
where $s = |S|$, $t = |T|$, $c_j := \pi_0(s_j)$ for $1\leq j \leq
s+t-1$ and 
\[
\tau :  \Markl(F) \ot A_{c_1} \ot \cdots\ot  A_{c_{i-1}} \longrightarrow
A_{c_1} \ot \cdots\ot  A_{c_{i-1}} \ot \Markl(F)
\]
the commutativity constraint in $\ttV$.
\end{itemize}
\end{definition}

Notice that in the situation of item~(ii) of
Definition~\ref{Zapomel_jsem_si_pocitac_v_Koline_ja_hlupak},  one has 
$s_i(S) = s(u_i)$,  $\pi_0(s_i(T)) = \pi_0(u_i)$ and $\pi_0(S) =
\pi_0(T)$. Likewise in~(iii),
\begin{equation}
\label{za_necely_tyden_do_Bari}
\pi_0(s_j(T)) =
\begin{cases}
\pi_0(s_{\inv{|\phi|}(j)}(S)) & \hbox {if $j \not= i$}
\\
\pi_0(F) & \hbox {otherwise.}
\end{cases}
\end{equation}

\begin{proposition}
 The category of algebras of a strictly
unital $1$-connected Markl operad~$\Markl$ is isomorphic to
 the category of algebras of the corresponding operad $\oP$.
\end{proposition}

\begin{proof}
An exercise in the axioms of operads and their algebras.
\end{proof}

Let us formulate 
Definition~\ref{Zapomel_jsem_si_pocitac_v_Koline_ja_hlupak} for the
particular case of
algebras in the category of 
graded vector spaces.

\begin{definition}
\label{vcera_s_Mikesem_na_Jazz_Bluffers}
An {\em algebra\/} over a $1$-connected Markl operad $\Markl$ in the category
$\Vect$ of graded $\bfk$-vector spaces is 
a collection $A = \{A_c \ | \ c \in \pi_0(\ttO)\}$ together with
structure maps
\[
\nonumber 
\Markl(T) \ot  \hskip -.5em 
\bigotimes_{c \in \pi_0(s(T))} A_c \ni x \ot a_1 \ot \cdots \ot a_s
\longmapsto  x(\Rada a1s) \in A_{\pi_0(T)}
\]
given for each $T \in \ttO$.
These operations are required to satisfy the following axioms. 

\begin{itemize}
\item[(i)]
Unitality: for a local terminal $u$,  $1 \in \bfk \cong \Markl(u)$ and $a \in
A_{\pi_0(s(u))}$, 
denote $u  a : = 1(a)$. Then $U  a = a$ for $U$  a chosen
local terminal object.

\item[(ii)]
Equivariance: for an \im\ $f : S\to T$ with fibers $\Rada u1s$
 and $x \in \Markl(T)$,
\[
 f^*(x)(\Rada a1s) = x(\rada{u_1a_1}{u_sa_s}).
\]
\item[(iii)]
Associativity: 
for an elementary map $F \fib_i S \stackrel\phi\to T$, $x\in \Markl(T)$ and
$y \in \Markl(F)$,  
\[
\circ_\phi(x,y)(\rada{a_1}{a_{i-1}},a_{i},\ldots,
{a_{t+s-1}})
=
(-1)^\varepsilon \cdot
x(\rada{a_1}{a_{i-1}},y(a_{i},\ldots),\ldots
{a_{t+s-1}}),
\]
where $\varepsilon := {|y|(|a_1| + \cdots + |a_{i-1}|)}$, $s = |S|$
and $t = |T|$.
\end{itemize}
\end{definition}

\def\fl{{\mathfrak l}}
\def\fdl{{\mathfrak D}_{\mathfrak l}}

\begin{example}
\label{coboundary}

A Markl operad $\Markl$
in $\Vect$
such that $\Markl(T)$ is for each $T$ 
a $1$-dimensional vector space is called a {\em cocycle\/} following the terminology of~\cite{getzler-kapranov:CompM98}.
An important cocycle is the operad ${\sf 1}_\ttO$ such that  ${\sf
  1}_\ttO(T) := \bfk$ for each $T \in \ttO$, with all composition
laws the identities. Slightly imprecisely, we will call ${\sf 1}_\ttO$ the {\em
  terminal\/} $\ttO$-operad since it is the linearization of the terminal
$\ttO$-operad over the Cartesian monoidal category of sets.

Less trivial cocycles can be constructed as follows. We
say that a graded vector space~$W$ is {\em invertible\/} if $W\ot
W^{-1} \cong \bfk$ for some  $W^{-1} \in \Vect$. This
clearly means that $W$ is an iterated (de)suspension of the ground
field $\bfk$.
Suppose we are given a map $\fl : \pi_0(\ttO) \to \Vect$ that assigns
to each $c \in  \pi_0(\ttO)$ an invertible graded vector space $\fl(c)$.
With the notation used in~(\ref{Ben_Ward_in_Prague}) we
introduce the cocycle $\fdl$ by
\[
\fdl(T) := \fl(\pi_0(T)) \ot \bigotimes_{c \in \pi_0(s(T))} \inv{\fl(c)}
\]
with the trivial action of $\Iso$.
To define,  for $F \fib_i S \stackrel\phi\to T$,  the composition laws
\[
\circ_\phi : \fdl(F) \ot \fdl(T) \to \fdl(S)
\]
we need to specify a map
\[
\fl(\pi_0(F)) \ot \bigotimes_{c \in \pi_0(s(F))} \inv{\fl(c)} \ot
\fl(\pi_0(T)) \ot \bigotimes_{c \in \pi_0(s(T))} \inv{\fl(c)} \longrightarrow
\fl(\pi_0(S)) \ot \bigotimes_{c \in \pi_0(s(S))} \inv{\fl(c)}.
\]
To do so, we notice that  there is an equality of unordered lists
\[
\pi_0(s(F)) \sqcup \pi_0(s(T)) = \pi_0(s(S)) \sqcup \{  \pi_0(s_i(T))
\},\] and  \[
\pi_0(S) = \pi_0(T) \  
 \ \pi_0(F) =   \pi_0(s_i(T)),
\]
cf.~(\ref{za_necely_tyden_do_Bari}). Keeping this in mind, the
composition law $\circ_\phi$ is defined as the canonical
isomorphism $\fdl(F) \ot \fdl(T) \cong \fdl(S)$.

Cocycles of the above form are called {\em coboundaries\/}. Notice
that  ${\sf 1}_\ttO = \fdl(T)$ with $\fl$ the constant function such
that $\fl(c) := \bfk$ for each   $c \in \pi_0(s(T))$.
\end{example}

Markl operads in $\Vect$ form a symmetric monoidal category,
with the monoidal structure given by the level-wise tensor product
and  ${\sf 1}_\ttO$ the monoidal unit. As an exercise we recommend to
prove the following very useful proposition.

\begin{proposition}
\label{za_chvili_zavolam_Jarusce}
The categories of $(\Markl \ot \fdl)$-algebras and of 
$\Markl$-algebras in $\Vect$ are isomorphic. More
precisely, there is a natural one-to-one correspondence between
\hfill\break
\noindent
--  $\Markl$-algebras
with underlying collection  $A = \{A_c \ | \ c \in \pi_0(\ttO)\}$,
and
\hfill\break 
\noindent
--  $(\Markl\ot \fdl)$-algebras
with underlying collection  
$A = \{A_c \ot \inv{\fl(c)} \, | \ c \in \pi_0(\ttO)\}$.
\end{proposition}

Proposition~\ref{za_chvili_zavolam_Jarusce} should be compared to
Lemma~II.5.49 of~\cite{markl-shnider-stasheff:book}.
In the classical operad theory, algebras can equivalently be described as
morphism to the endomorphism operad. We are going to give similar
description also in our setup. While the classical construction assigns the
endomorphism operad $\End_V$ to a vector space $V$, here we start with
a collection 
\begin{equation}
\label{Jarka_hovori}
V = \{V_c \ | \ c \in \pi_0(\ttO)\}
\end{equation}
of graded vector spaces indexed by the components of $\ttO$. We moreover
assume that to each local terminal object $u\in \ttO$ we are given a linear
map (denoted $u$ again) 
\begin{equation}
\label{Asi_jsem_dostal_premii!}
u: V_{\pi_0(s_1(u))} \to  V_{\pi_0(u)}
\end{equation}
such that, for each map $u \to v$ of local terminal objects 
with fiber $t$, the triangle
\begin{equation}
\label{Kdy_zacnu_jezdit_na_kole?}
\xymatrix{
V_{\pi_0(s_1(u))} \ar[rr]^u \ar[rd]^t   &&V_{\pi_0(u)}
\\
&V_{\pi_0(t)} \ar[ru]^v  &
}
\end{equation}
commutes. This diagram makes sense since $\pi_0(s_1(u)) = \pi_0(s_1(t)),$ by Axiom (iv) of operadic categories applied to  $f=\id_u$, $\pi_0(u) = \pi_0(v)$
and $\pi_0(s_1(v)) = \pi_0(t)$. 
We moreover assume that the maps corresponding to the chosen local terminal
objects are the identities. For $T \in \ttO$  we put
\[
\End_V(T) := \Vect\big(\textstyle 
\bigotimes_{c \in \pi_0(s(T))} V_c,V_{\pi_0(T)}\big).
\]

We define an action $
\End_V(T) \ni \alpha \mapsto f^*(\alpha) \in \End_V(S)$  
of  an \im\ $f : S\to T$ with fibers
$\Rada u1s$ by
\begin{equation}
\label{Popozitri_letim_do_Bari.}
f^*(\alpha)(\Rada a1s) := \alpha(\rada{u_1a_1}{u_sa_s}), \ 
a_1 \ot \cdots \ot a_s \in \textstyle \bigotimes_{c \in \pi_0(s(T))} V_c.
\end{equation}
This turns $\End_V$ into a functor  $\Iso^{\rm op} \to \Vect$.  
The composition law
\[
\circ_\phi: \End_V(F) \ot \End_V(T) \to \End_V(S) .
\]
is, for an elementary morphism
$F \fib_i S \stackrel\phi\to T$, defined as follows.
Assume
\[
\alpha : \textstyle \bigotimes_{c \in \pi_0(s(F))} V_c \longrightarrow
V_{\pi_0(F)} \in \End_V(F),\
\beta : \textstyle \bigotimes_{c \in \pi_0(s(T))} V_c \longrightarrow
V_{\pi_0(T)}\in 
\End_V(T)
\]
and notice that
\[
\pi_0(s(S)) = \pi_0(s(F)) \sqcup (\pi_0(s(T))
\setminus  \{  \pi_0(s_i(T))\}) \ \hbox { and } \ \pi_0(F) = \pi_0(s_i(T)).
\]
Then
\[
\circ_\phi(\alpha \ot \beta) : \textstyle \bigotimes_{c \in
  \pi_0(s(S))} V_c \to V_{\pi_0(S)} \in \End_V(S)
\]
is the map that makes the diagram
\[
\xymatrix{
 \textstyle \bigotimes_{c \in
  \pi_0(s(S))} V_c \ar[r]^(.3)\cong
\ar[dd]_{\circ_\phi(\alpha \ot \beta)} & \textstyle \bigotimes_{c \in
  \pi_0(s(F))} V_c  
\ot
 \bigotimes_{c \in
  \pi_0(s(S))\setminus  \{  \pi_0(s_i(T))\}} V_c \ar[d]^{\alpha \ot \id}
\\
& \hskip 2em \textstyle
 V_{\pi_0(s_i(T))}
\ot
 \bigotimes_{c \in
  \pi_0(s(S))\setminus  \{  \pi_0(s_i(T))\}} V_c\ar[d]^\cong
\\
 V_{\pi_0(S)} &  \textstyle
 \bigotimes_{c \in
  \pi_0(s(S))} V_c \ar[l]_\beta
}
\]
commute. The result of the above construction is the Markl version
of the 
{\em   endomorphism operad\/}.

Notice that the components $\eta_u : \bfk \to \End_V(u)$ of
transformation~(\ref{2x2}) for $\Markl=\End_V$ are  given by the maps
in~(\ref{Asi_jsem_dostal_premii!}) as
\[
\eta_u(1) := u : 
V_{\pi_0(s_1(u))} \to  V_{\pi_0(u)}
\in \End_V(u).
\]
It is simple to verify that the commutativity of~(\ref{Je_vedro.}) is
precisely~(\ref{Kdy_zacnu_jezdit_na_kole?}). The induced maps 
\[
\vartheta(T,u) : \End_V(F) \to \End_V(T)
\] 
in~(\ref{proc_ty_lidi_musej_porad_hlucet}) are given by the
composite, $\vartheta(T,u)(\phi) := u \kompozice \phi$, with the
map~(\ref{Asi_jsem_dostal_premii!}). 

\begin{remark}
The above analysis shows that the morphisms  $\vartheta(T,u)$ need not
be the identities for a general $\End_V$. Endomorphism operads are
therefore examples of unital operads that need not be 
{\em strictly\/} unital.
\end{remark}

We have the following expected result.

\begin{proposition}
\label{Udelal_jsem_si_ciruvky_zelanky.}
There is a one-to-one correspondence between $\Markl$-algebras with
underlying collection~(\ref{Jarka_hovori}) and
morphisms $\Markl \to \End_V$ of Markl operads.
\end{proposition}

\begin{proof}
Direct verification.
\end{proof}

\begin{theindex}
{\normalsize
\item Algebra over a Markl operad,
  Definition~\ref{Zapomel_jsem_si_pocitac_v_Koline_ja_hlupak}
  and Definition~\ref{vcera_s_Mikesem_na_Jazz_Bluffers}
\item Canonical grading, page~\pageref{canon},
  Section~\ref{Ceka_mne_Psenicka.}  
\item Cocycle, Example~\ref{coboundary}
\item Coboundary, Example~\ref{coboundary}
\item Constant-free operadic category, Definition~\ref{constant-free}
\item Contraction, Definition~\ref{pisu_jednou_rukou}
\item Discrete operadic fibration Definition~\ref{psano_v_Myluzach}
\item Discrete operadic opfibration Definition~\ref{zas_mne_boli_zapesti}
\item Derived sequence, Equation~\eqref{e1}
\item Elementary morphism, Definition~\ref{plysacci_postacci}
\item Endomorphism operad, page~\pageref{Jarka_hovori},
  Section~\ref{section-markl} 
\item Factorizability, \Fac,
  Definition~\ref{dnes_prednaska_na_Macquarie}
\item Grading of an operadic category,  Definition~\ref{grad}
\item Harmonic pair, Definition~\ref{har}
\item Invertibility of \qb{s}, \QBI,
  Lemma~\ref{I_laboratorni_vysetreni_mne_ceka.}
\item Local isomorphism, Definition~\ref{pisu_jednou_rukou}
\item Local reordering morphism, Definition~\ref{pisu_jednou_rukou}
\item Markl operad, Definition~\ref{markl}
\item Order preserving morphism, Definition~\ref{pisu_jednou_rukou}
\item Ordered graph,
Definition~\ref{Radeji_bych_sel_s_Jaruskou_na_vyhlidku_sam.}
\item $\Iso \subset \ttO$, subcategory of isomorphisms, page~\pageref{Co
  zitra zjisti?}, Section~\ref{section-markl}
\item $\LT \subset \ttO$, subcategory of local terminal objects,
  page~\pageref{Jsem znepokojen.}, Section~\ref{section-markl}
\item $\DO\subset \ttO$, subcategory of order-preserving morphisms,
  page~\pageref{9 dni}, 
Section~\ref{a0}
\item $\QO\subset \ttO$, subcategory of \qb{s},
  page~\pageref{9 dni}, Section~\ref{a0}
\item Pairs with disjoint fibers, Definition~\ref{d3}
\item Preordered graph, Definition~\ref{pre}
\item Pure contraction, Definition~\ref{pisu_jednou_rukou}
\item Rigidity, \Rig, 
      Definition~\ref{Porad_nevim_jestli_mam_jit_na_ty_narozeniny.}
\item Strict grading, \SGrad, Definition~\ref{sgrad}
\item Strictly unital Markl operad, Definition~\ref{svedeni}
\item Strong blow-up axiom, \SBU, page~\pageref{bu}, Section~\ref{Pojedu_vecer_nebo_ted?}
\item Strongly factorizable operadic category, \SFac, 
      Definition~\ref{zase_jsem_podlehl}
\item Unique fiber axiom (condition), \UFB, Definition~\ref{Kveta_asi_spi.}
\item Weak blow-up axiom, \WBU, page~\pageref{wbu},  Section~\ref{Pojedu_vecer_nebo_ted?}
\item $1$-connected Markl operad, Definition~\ref{svedeni}
}
\end{theindex}



%

\providecommand{\doi}[1]{}
\renewcommand{\doi}[1]{\href{https://doi.org/\detokenize{#1}}{DOI: \detokenize{#1}}}%
\newcommand{\arxiv}[1]{\href{http://arxiv.org/abs/#1}{arXiv:#1}}

\providecommand{\fulldoi}[1]{}
\renewcommand{\fulldoi}[1]{\href{\detokenize{#1}}{\detokenize{#1}}}%
\renewcommand{\fulldoi}[1]{\url{\detokenize{#1}}}%

\label{lastpage}

\end{document}